\documentclass[12pt]{amsart}
\usepackage[margin=2.5cm]{geometry}
\usepackage{amsmath,amssymb,amsfonts,amsthm,amsrefs,mathrsfs}
\usepackage{epic}
\usepackage{amsopn,amscd,graphicx}
\usepackage{color,transparent} 
\usepackage[dvipsnames]{xcolor}
\usepackage[colorlinks, breaklinks,bookmarks = false,linkcolor = NavyBlue,urlcolor = ForestGreen,citecolor = ForestGreen,hyperfootnotes = false]{hyperref}
\usepackage{scalerel}
\usepackage{stackengine,wasysym}
\usepackage{palatino, mathpazo}
\usepackage{multirow}
\usepackage[all,cmtip]{xy}
\usepackage{stmaryrd}
\usepackage{tikz-cd}
\usepackage{yhmath}
\usepackage{enumerate}
\usepackage{diagbox}

 1
 1
 1

\newtheorem{theorem}{Theorem}[section]
\newtheorem{thm}[theorem]{Theorem}
\newtheorem{fact}[theorem]{Fact}

\newtheorem{pro}[theorem]{Proposition}
\newtheorem{lem}[theorem]{Lemma}
\newtheorem{exe}[theorem]{Example}
\newtheorem{rem}[theorem]{Remark}
\theoremstyle{definition}
\newtheorem{Def}[theorem]{Definition}

\newtheoremstyle{named}{}{}{\itshape}{}{\bfseries}{.}{.5em}{\thmnote{#3}#1}
\theoremstyle{named}

\numberwithin{equation}{subsection}

\newcommand{\bbC}{\mathbb{C}}
\newcommand{\bbP}{\mathbb{P}}
\newcommand{\bbR}{\mathbb{R}}
\newcommand{\bbN}{\mathbb{N}}

\newcommand{\cC}{\mathcal{C}}
\newcommand{\cD}{\mathcal{D}}
\newcommand{\cE}{\mathcal{E}}
\newcommand{\cU}{\mathcal{U}}

\newcommand{\sH}{\mathscr{H}}

\newcommand{\gO}{\Omega}

\newcommand{\ga}{\alpha}
\newcommand{\gb}{\beta}
\newcommand{\go}{\omega}
\newcommand{\gl}{\lambda}

\newcommand{\invs}{^{-1}}

\newcommand{\ol}{\overline}
\newcommand{\ti}[1]{\tilde{#1}}
\newcommand{\wt}{\widetilde}
\newcommand{\vol}{\operatorname{vol}}

\newcommand{\Dom}{\mathrm{Dom}}
\newcommand{\End}{\mathrm{End}}
\newcommand{\supp}{\mathrm{supp}}
\newcommand{\nb}{\nabla}

\renewcommand{\dbar}{\bar{\partial}}
\newcommand{\dr}{\partial}
\newcommand{\kla}{\square}

\renewcommand{\(}{\left(}
\renewcommand{\)}{\right)}
\renewcommand{\[}{\left[}
\renewcommand{\]}{\right]}
\newcommand{\<}{\left\langle}
\renewcommand{\>}{\right\rangle}

\title[Toeplitz Quantization]{Semi-Classical Asymptotic Expansions for Toeplitz Quantizations on Complex Manifolds and Orbifolds}

\begin{document}

\author[Yi-Hsin Tsai]{Yi-Hsin Tsai}
\address{Department of Mathematics, National Taiwan University}
\thanks{Master's thesis at National Taiwan University, under the supervision of Prof. Chin-Yu Hsiao.}
\email{r12221003@ntu.edu.tw}

\begin{abstract}
In this thesis, we introduce complex manifolds with local spectral gaps and study their asymptotic behavior using the scaling method. With these asymptotics, we obtain an asymptotic expansion for the Bergman kernel of a Hermitian holomorphic orbifold line bundle satisfying the local spectral gap condition. Furthermore, we establish the full asymptotic expansion of both the Bergman kernel and the Toeplitz operator, using the observations of the scaled Bergman kernel and the stationary phase formula. In addition, we establish the deformation quantization for Toeplitz operators with pseudodifferential operators. 
\end{abstract}

\date{\today}
\maketitle

\tableofcontents
\vspace{-10pt}
\textbf{Keywords:} Complex manifolds, Deformation quantization

\textbf{Mathematics Subject Classification:} 32Qxx, 32A25, 53D55 
\tableofcontents

\section{Introduction}
\label{Introduction}

The Bergman kernel is a fundamental object in complex analysis and differential geometry, with deep connections to the geometry of Kähler manifolds and orbifolds. Introduced by Stefan Bergman in 1922  \cite{Bergmann1922}, the Bergman kernel has since been extensively studied and applied in various fields, including mathematical physics, several complex variables, and algebraic geometry.

Let $(L,h^L)\rightarrow M$ be a Hermitian line bundle over a complex manifold $M$. We denote $(L^k, h^{L^k})$ the $k$-th tensor of this line bundle and the Bergman projection $B_k$ is the orthogonal projection from the space of square-integrable sections $L^2(M, L^k)$ to the space of holomorphic square-integrable sections $H^0(M, L^k)$. Let $B_{k}(z,w)$ be the distribution kernel of $B_k$, known as the Bergman kernel. The study of large $k$ behavior of the Bergman kernel plays an important role in complex geometry and mathematical physics. 

For the asymptotic expansion of the Bergman kernel $B_k(z,w)$ on complex manifolds, Tian, via Hörmander's $L^2$-estimate  \cite{Tian1990} and Bouche via the Heat kernel \cite{Bouche1990}, obtained the first term of the asymptotic expansion of the Bergman kernel in 1990. Bouche \cite{Bouche96} and Berman  \cite{berman2002} gave another proof using the scaling method. In 1998 and 1999, Zelditch  \cite{Zelditch1998}and Catlin \cite{Catlin1999} independently obtained a full asymptotic expansion of Bergman kernels on the diagonal by using Boutet de Monvel-Sjöstrand theorem for Szegő kernel on strictly pseudoconvex Cauchy-Riemann Manifold\cite{boutet1975singularite}. Then, in 2006, Dai, Liu, and Ma \cite{Dai-Liu-Ma2006} and Ma and Marinescu \cite{Ma-Marinescu2006} obtained the full asymptotic expansion for generalized Bergman kernels of the spin$^c$-Dirac operator using the Heat kernel method and the localization technique of Bismut-Lebeau. There are other proofs by Berman, Berndtsson, and Sjöstrand in 2008 \cite{Berman-Berndtosson-Sjöstrand2008} by the approximate Bergman kernel and  Hezari, Kelleher, Seto, and Xu in 2014 \cite{Hezari-Kelleher-Seto-Xu2014asymptotic} \cite{Seto2015} by the approximate Bergmann-Fock kernel. In this thesis, we establish the leading behavior and full asymptotic expansion of the Bergman kernel in several new settings by using the scaling method and the stationary phase formula (see Theorems \ref{MFD-Main-Thm}, \ref{spectral kernel theorem}, \ref{Bergman Full Expansion Theorem}).

Once we have the full asymptotic expansion for the Bergman kernel, one of the next goals is to compute its coefficients. The coefficients of the full expansion are related to several geometric problems \cite{donaldson2001, Fine2010}. However, computing the coefficients is a cumbersome process; we just have some results. In the polarized case, Lu \cite{lu1999lowerordertermsasymptotic} computed the first four coefficients by using the peak section method, while Ma and Marinescu \cite{Ma2011} calculated the first three coefficients for Toeplitz operators by using the heat kernel method. Hsiao \cite{hsiao2011coefficientsasymptoticexpansionkernel} computed the first three coefficients for both the Toeplitz operators and the Bergman kernel by using the complex stationary phase formula. Additionally, Xu \cite{Xu2012} connected these coefficients with graph theory. In this thesis, we introduce a different, relatively direct, and elementary method to compute the coefficients (see Subsection~\ref{bergman full expansion}).

 The scaling method is a widely used approach in the study of partial differential equations (PDEs), particularly for equations like the heat equation and the wave equation. However, its application to the Bergman kernel remains relatively unexplored, with notable exceptions in the work of Bouche, Berman, Berndtsson, and Sjöstrand \cite{Bouche96,berman2002, Berman-Berndtosson-Sjöstrand2008}, who focused on the behavior of the kernel along the diagonal. In this thesis, we develop a scaling method to analyze the entire Bergman kernel. More precisely, we introduce the scaling kernel $B_{(k)}(z,w) = k^{-n} B_k\(\frac{z}{\sqrt{k}}, \frac{w}{\sqrt{k}}\)$ which also captures the behavior of the off-diagonal part  (see Section~\ref{Standard Results}).

This thesis will also study the asymptotic expansion of the Bergman kernels and Toeplitz operators on complex orbifolds  (see Section~\ref{Orbifold Case}). Orbifolds generalize the concept of manifolds by allowing for certain types of singularities, making them a rich subject for complex geometric methods. The study of orbifolds dates back to the 1956 with the pioneering work of Satake \cite{Satake56} \cite{Satake57}, who introduced the notion of V-manifolds, which called orbifold recently, and further developed by Thurston \cite{Thurston79} \cite{Thurston82} \cite{Thurston97} in the context of geometric structures on 3-manifolds.

The theory of pseudodifferential operators($\Psi$DO) was introduced by Kohn and Nirenberg \cite{kohn1965}, Hörmander \cite{hormander1965}, and several other mathematicians in the mid-1960s. Pseudodifferential operators have significant applications, such as in the Atiyah–Singer index theorem \cite{atiyah1968}. Toeplitz operators are closely related to deformation quantization \cite{schlichenmaier2010}. In particular, we can replace the multiplication operator in a Toeplitz operator with a pseudodifferential operator. In this thesis, we derive the full asymptotic expansion of Toeplitz operators with pseudodifferential symbols and develop a method to compute their expansion coefficients. Additionally, we study the commutator of such Toeplitz operators and establish its connection to deformation quantization  (see Subsections~\ref{Toeplitz operator} and \ref{Toeplitz operator with PDO}). Quantization of the Toeplitz operators with pseudodifferential operators on compact complex manifolds is a fundamental problem in complex geometry and deformation quantization. In this thesis,  we answer this fundamental problem.

We also have the following related results: Ross and Thomas gave the weighted Bergman kernel and weighted projective embedding with positive line bundles in 2011 \cite{RossThomas2011} \cite{RossThomas2011Bergman}. In 2014, Hsiao and Marinescu provided local asymptotic results under certain spectral gap conditions \cite{Hsiao-Marinescu2014}. The same year, Hsiao also gave an analytic proof of the Kodaira Embedding Theorem using the asymptotic behavior of the Bergman kernel \cite{hsiao2014bergman}. In 2018, Puchol proved the Holomorphic Morse inequalities for orbifolds using the heat kernel \cite{Puchol2018}. In this thesis, we modify the method by Hsiao in 2014 to get an analytic proof of Kodaira-Baily Embedding Theorem  (see Theorem~\ref{Kodaira-Baily Theorem}).

The outline of the thesis is the following. Section~\ref{Preliminaries} sets the stage by introducing necessary preliminaries, including the definitions and standard results concerning complex manifolds, differential operators, and the Bergman kernel. In Section~\ref{Standard Results}, we explore the leading term of the Bergman kernel in different settings: the Euclidean case, the compact manifold case, and manifolds without spectral gaps. Section~\ref{Orbifold Case} is dedicated to orbifold cases, where we prove the leading term of the Bergman kernel on an orbifold and apply this result to provide a pure analytic proof of the Kodaira-Baily Embedding Theorem and study the asymptotic behavior of Toeplitz operators. In Section~\ref{Full Expansion}, we explain the method to get the full expansion and how to compute the coefficient of the Bergman kernel. Also, we apply our method to the full expansion of the Toeplitz operator and generalize the Toeplitz operator for the pseudodifferential operator cases.

Our approach leverages various mathematical tools, such as the Toeplitz operators and spectral projection methods, to achieve a comprehensive understanding of the Bergman kernel asymptotic behavior. The Kodaira-Baily Embedding Theorem is one of the notable results we prove using our asymptotic expansions. The original version of this theorem, proved by Kodaira \cite{Kodaira1954} and Baily \cite{Baily1957} in the 1950s, is a cornerstone in the theory of complex manifolds.

Overall, the findings presented in this thesis not only contribute to the theoretical foundation of complex geometry but also have potential applications in mathematical physics and other related fields. By advancing our understanding of the Bergman kernel on orbifolds, we open new avenues for research in complex geometry and its applications.

\subsection{Statement of Main Results and Applications}
We now present the main results of the thesis. We refer to section \ref{Preliminaries} and subsection \ref{orbifold} for relevant notations and terminology used here. Let $X$ be a complex orbifold of dimension $n$. Let $(L,h^L)$ be a Hermitian orbifold line bundle and $L^k : =  L^{\otimes k}$, for $k\in \bbN$.  For a local trivialization $s$ of $L$ over $U\subset X$, $|s|^2_{h^L} = e^{-2\phi}$ with $\phi\in \cC^\infty(U)$ which called the local weight of $h$ with respect to $s$, we also have $h^{L^k}$ on $L^k$ with local weight $k\phi$. Fix a  Hermitian metric $\Omega$ on $X$ and $h^L$ on $L$, which induce $L^2$-norm $\|\cdot\|_{k,q}$ and $L^2$-inner product $\(\cdot| \cdot\)_{k,q}$ with the $L^2$ space $L^2_{(0,q)}(X, L^k)$ be the completion of $\Omega^{(0,q)}_c(X,L^k)$ with respect to the norm (see \eqref{inner product}, \eqref{integral on orbifold} and subsections \ref{Complex Manifold}, \ref{orbifold}). When $q=0$, We write $\(\cdot|\cdot\)_k = \(\cdot|\cdot\)_{k,q}$, $\|\cdot\|_{k} = \|\cdot \|_{k,0}$ and $L^2(X,L^k) = L^2_{(0,0)}(X,L^k) $.

Let $\dbar_k: \Omega(X,L^k)= \cC^\infty(X,L^k) \rightarrow \Omega^{(0,1)}(X,L^k)$ be the Cauchy–Riemann operator and $\dbar_k^*:\Omega^{(0,1)}(X,L^k) \rightarrow \Omega(X,L^k)$ be the formal adjoint of $\dbar_k$ with respect to 
$\( \cdot | \cdot \)_{k,1} $ and the Kodaira Laplacian $ \kla_k = \dbar^*_k \dbar_k:\cC^\infty(X,L^k)\rightarrow \cC^\infty(X,L^k)$. Note that we can define the Gaffney extension of Kodaira Laplacian $\kla_k: \Dom \kla_k \subset L^2(X, L^k) \rightarrow L^2(X, L^k)$, see \eqref{Gaffney Extension}. Then we let $$P_k: L^2(X, L^k) \rightarrow \ker \kla_k$$ be the orthogonal projection, known as the Bergman projection, where $\ker \kla_k = \{u\in\Dom \kla_k| \kla_k u = 0\} $.

To state our results, we first need to define the local spectral gap property.

\begin{Def}
    \label{Local Spectral Gap}
    For any open set $U\subset X$, we say that $\kla_{k} $ has local spectral gap on $U$ if there are  $C>0 $ and $r\in\bbR$ such that for all $k$ large enough $$\|(I- P_k) u \|_k^2 \le C k^r \(\kla_k u |u \)_k \text{ for all } u \in \Omega_c(U, L^k), $$ where $P_k$ is the Bergman projection.
\end{Def}

\begin{Def}
    \label{Ck Local Spectral Gap}
    For any open set $U\subset X$, we say that $\kla_{k} $ has $C_k$ local spectral gap on $U$ if there is  $C_k>0 $ satisfying $\lim_{k\to \infty} \frac{kC_k}{\log^{1-\varepsilon}k} = 0$ for some $\varepsilon>0$ such that for all $k$ large enough $$\|(I- P_k) u \|_k^2 \le C_k \(\kla_k u |u \)_k \text{ for all } u \in \Omega_c(U, L^k), $$ where $P_k$ is the Bergman projection.
\end{Def}

Let $L^\infty(X)$ be the space of essentially bounded measurable functions. Given $f\in L^\infty(X)$, we can define the Toeplitz operator $T_{f,k} = P_k \circ M_f\circ P_k$ where $M_f$ is the multiple operator. Let $T_{f,k}(x,y) \in \cC^\infty(X\times X, L^k\boxtimes (L^k)^*)$ be the distribution kernel of $T_{f,k}$. Note that $T_{f,k}(x,y) = \int_X P_k(x,z) f(z) P_k(z, y)d\vol_X(z)$, where $d\vol_X$ is the volume form of the orbifold $X$ induced by the Hermitian form $\Omega$. 

Let $p\in X$. Notice that near $p$, we can always find a local holomorphic coordinate $z$, $z(p)=0$, and  local holomorphic trivialization $s: U\subset X \rightarrow L$ has local weight $\phi(z) = \sum_j \lambda_j |z_j|^2 + O(|z|^3) $, 
\begin{equation}
\label{weight}
    \begin{cases}
        \phi = \phi_0+\phi_1 \\
        \phi_0 = \sum_{j=1}^n \lambda_j |z_j|^2\\
        \phi_1 = O(|z|^3),
    \end{cases}
\end{equation} 
and  $\Omega = \sum_{j,\ell=1}^n (\delta_{j,\ell} + O(|z|))dz_j\wedge d\overline{z}_\ell$ (see Lemma \ref{chern-moser coordinate}). We will fix this trivialization $(z,s)$, which is called the Chern-Moser trivialization on $U$ centered at $p$, in the following and hence identify $L^2(U, L^k)$ with the $L^2$ space $L^2(U,d\vol_X)$ by 

$$
\begin{aligned}
\tau:\Omega_c(U,L^k) &\rightarrow \Omega_c(U)\\
u\otimes s^k &\mapsto  ue^{-k\phi}.
\end{aligned}
$$ With the inner product $$\(u|v\)_{U,q} =\int_{U} \<u|v\>_{\Omega,q} d\vol_M,\, \forall u,v\in \Omega^{(0,q)}_c(U),$$ the $L^2$-norm $\|\cdot\|_{U,q}$ and the $L^2$-space. Sometimes, for simplicity, we will omit the subscript $q$ when $q=0$. This identify also give the localized Cauchy-Riemann operator $\dbar_{k,s}:\Omega_c(U) \rightarrow \Omega_c^{(0,1)}(U)$ satisfy $\tau(\dbar_k u) = \dbar_{k,s}\tau(u)$, its formal adjoint $\dbar_{k,s}^*:\Omega_c^{(0,1)}(U) \rightarrow \Omega_c(U)$ with respect to the $L^2$-inner product $\(\cdot|\cdot\)_{U,q}$ and the Gaffney extension localized Kodaira Laplacian $\kla_{k,s}:\Dom \kla_{k,s}\subset L^2(U,d\vol_X) \rightarrow L^2(U,d\vol_X)$. The localization Bergman projection $$P_{k,s} :L^2_{comp}(U,d\vol_X)\rightarrow L^2(U,d\vol_X) $$ which satisfying $$P_k(u\otimes s^k) = (e^{k\phi}P_{k,s}(e^{-k\phi}u)) \otimes s^k $$ and the localization Toeplitz operator $T_{f,k,s}: L^2_{comp}(U,d\vol_X) \rightarrow L^2(U,d\vol_X)$ with $T_{f,k}(u\otimes s^k) = (e^{k\phi}T_{f,k,s}(e^{-k\phi}u)) \otimes s^k$. Also, let $P_{k,s}(x,y)$ and $T_{f,k,s}(x,y)$ be their distribution kernels respectively.

For $\phi(z)\in\cC^\infty(\bbC^n, \bbR)$, we can define $$\(u|v\)_{\phi} = \int_{\bbC^n} u(z)\overline{v}(z) e^{-2\phi(z)} d\lambda(z),  $$ $\|\cdot\|_\phi$ and the $L^2$-space $L^2(\bbC^b, \phi)$. The model case of the asymptotic is the orthogonal projection of the Bergman-Fock space, that is, the orthogonal projection
$$B_{\phi_0}: L^2(\bbC^n, \phi_0)\rightarrow H^0(\bbC^n, \phi_0)= \{u\in L^2(\bbC^n,\phi_0)| \dbar u =0 \} ,$$ with $\phi_0 = \sum_{j=1}^n \lambda_j |z_j|^2$ with $\lambda_j>0$, as in \eqref{weight}. The distribution kernel of $B_{\phi_0}$ is known as following 
$$(B_{\phi_0}u)(x) = \int_{\bbC^n}B_{\phi_0}(x,y) u(y) d\lambda(y),\, \forall u\in L^2(\bbC^n, \phi_0) $$ where $d\lambda(z) = i^n dz_1d\overline{z}_1 \cdots dz_nd\overline{z}_n$ and 
$$B_{\phi_0}(x,y) = \(\prod_{j=1}^n \frac{\lambda_j}{\pi}\)e^{2\sum_j \lambda_j(x_j \overline{y}_j -|y_j|^2)}.$$
We will denote $P_{\phi_0}$ be the conjugate of $B_{\phi_0}$  by $e^{k\phi_0}$. That is $P_{\phi_0} = e^{-\phi_0}B_{\phi_0} e^{\phi_0} $ and 
\begin{equation}
\label{conjugate bergman kernel}
P_{\phi_0}(x,y)=\(\prod_{j=1}^n \frac{\lambda_j}{\pi}\)e^{\sum_j \lambda_j(2x_j \overline{y}_j-|x_j|^2 -|y_j|^2)}.
\end{equation}

After the above preparation, we can now state the main result. We begin by considering the case of orbifolds.

\begin{thm}[Orbifold Case]
    \label{Orbifold Case Main Theorem in Introduction}
     For any $p\in X$ and an open neighborhood $U\subset X$ with $L$ positive over $U $, $f\in L^\infty(X)\cap \cC^0(X)$ and $\kla_k$ satisfying $C_k$ local spectral gap  on $U$, then there is a Chern-Moser trivialization $(z,s)$ on $U$ centered at $p$ such that 
     $$T_{f,k} (x,y)=k^n \sum_{g\in G_p} P_{\phi_0}(\sqrt{k}\ti{x} , \sqrt{k}(g\cdot\ti{y})) f(p) + \varepsilon_k(x,y)  $$ where $G_p$ is the isotropic group of $p$,  $\pi(\ti{x})= x $ is the natural projection and  $\varepsilon_k(x,y)$ satisfying $$\left\|\varepsilon_k\(\frac{\ti{x}}{\sqrt{k}}, \frac{ \ti{y}}{\sqrt{k}}\)\right\|_{\cC^\ell(K)}=o(k^n)$$ for all $ \ell>0$ and any compact subset $K\subset B(0,\log k):= \{z\in \bbC^n\mid |z|< log k\}$.
    In particular, when $f=1$ is just the Bergman kernel asymptotic expansion. Recall that $\phi_0$ and $P_{\phi_0}$ are as in \eqref{weight} and \eqref{conjugate bergman kernel} respectively.
\end{thm}

\begin{rem}
    With the same notations as above. For $f\in L^\infty(X) \cap \cC^s(X)$, we have $$T_{f,k} (x,y)=k^n \sum_{g\in G_p} P_{\phi_0}(\sqrt{k}\ti{x} , \sqrt{k}(g\cdot\ti{y})) f(\ti{x}) + \varepsilon_k(x,y)$$ with 
    $$ \left\|\varepsilon_k\(\frac{\ti{x}}{\sqrt{k}}, \frac{\ti{y}}{\sqrt{k}}\)\right\|_{\cC^\ell(K)} =o(k^n), \, \text{ for all } \ell<s, $$ and any compact subset $K\subset B(0,\log k).$
\end{rem}

We then consider the smooth manifold cases.
\begin{thm}[Manifold Case]
    \label{Manifold Case Main Theorem in Introduction}
    Let $M$ be a complex manifold with a Hermitian line bundle $(L,h^L)\rightarrow M$. With the same notations used in orbifold case, for any $p\in M$ and an open neighborhood $U\subset M$ with $L$ positive over $U $, $f\in L^\infty(M)\cap \cC^0(M)$ and $\kla_k$ satisfying local spectral gap on $U$, then there is a Chern-Moser trivialization  $(z,s)$ on $U$ centered at $p$ such that $$ \left\|\frac{1}{k^n}T_{f,k,s}(\frac{x}{\sqrt{k}},\frac{y}{\sqrt{k}}) - P_{\phi_0}(x, y) f(p)\right\|_{\cC^{\ell}(K)} \rightarrow 0, \, \text{ for all } \ell \in \bbN_0,\, K\Subset B(0,\log k).$$ 
    In particular, when $f=1$ is just the Bergman kernel asymptotic expansion. 
\end{thm}

We will divide the proof of the main theorem into two parts, the first is in section~\ref{Standard Results}, where we will derive the asymptotic of the Bergman kernel from $\bbC^n$, a positive compact complex manifold, and then near the positive points in a semi-positive complex manifold. The second part is using the previous result to get the asymptotic of the Bergman kernel on positive points of a complex orbifold with a local spectral gap in section~\ref{Orbifold Case}.
Note that because the proof of the convergence of Toeplitz operators is similar, we will only prove the orbifold case.

As the application of Theorem \ref{Orbifold Case Main Theorem in Introduction}, we will give an analytic proof of the Kodaira-Baily embedding theorem \cite{Baily1957}. 
\begin{thm}[Kodaira-Baily Embedding Theorem(cf. Theorem \ref{Kodaira-Baily Theorem})]
    The Kodaira-Baily map $\Phi_k$ is injective and immersion near every regular point. Near singular point, let $p$ be a singular point and let $(\ti{U}, G_U)$ be an orbifold chart defined near $p$. Then, there is an immersion $\hat{\Phi}_k$ defined near $\ti{V}$ such $\Phi_k(x)=\sum_{g\in G_U}\hat{\Phi}_k(g\cdot x)$ where $\pi(\ti{p})=p, \pi:\ti{U}\to U$ is the natural projection.
\end{thm}
Our proof is analogous to the proof by Hsiao \cite{hsiao2014bergman}. But with orbifold points, the immersion will be the problem. There is an easier proof if we assume an extra condition mentioned in Remark \ref{Extra Condition}. For the other cases, we just prove the Kodaira-Baily map can be represented by embedding in the chart acted by all elements in the isotropic group.

Inspired by the leading term of the Bergman kernel, we want to determine the difference between $B_k$ ($P_k$) and $B_{k\phi_0}$ in a small open set. When we use the self-adjoint and off-diagonal property of the Bergman kernel, we notice that we can use the composition of $B_{k\phi_0}$ and its adjoint to represent $B_{k\phi}$. Then we apply the Stationary Phase formula on this representation, which gives us the full-expansion of $B_{k\phi}$ and another way to compute the coefficient.

To state the next result, we first introduce the concept of asymptotic sum.
\begin{Def}
\label{Asymptotic sum}
Let $A_k, a_j$ be smooth functions on a compact manifold $M$. We denote that $A_k(x) \sim \sum_{j=0}^\infty k^{n-j} a_j(x)$ in $\cC^\infty(M) $ if for all $N,\ell\in\bbN$, there is a positive constant $C_{N,\ell} $ satisfying $$\|A_k(x) - \sum_{j=0}^N k^{n-j} a_j(x) \|_{\cC^\ell(M)} \le C_{N,\ell} k^{n-N-1},\, \text{ for all } k\gg 1,\,\ell\in \bbN_0.$$
\end{Def}

\begin{thm}[cf. Theorem \ref{Bergman Full Expansion Theorem}]
    Let $M$ be a compact complex manifold $M$ with Hermitian line bundle $(L,h^L)$, $R^L>0$ on $M$. There are $a_j(z) \in \cC^\infty(M)$ for $j=1,2,\ldots$ such that $$P_{k,s}(z,z) = B_k(z)\sim \sum_{j=0}^\infty k^{n-j} a_j(z)$$ in $\cC^\infty(M)$.
    More precisely, let $p\in M$ and $(z,s)$ be the Chern-Moser trivialization on $U$ centered at $p$, we can compute the localized Bergman kernel by $$P_{k,s}(0,0)= k^nC_0+k^nC_0 \sum_{\ell=1}^\infty \sum_{j=1}^\infty \frac{(2k)^{-j}}{j!} \Delta^j_{\ell} u_{k,\ell,0} (0)$$ where $\phi = \phi_0 + \phi_1$, $\phi_1=O(|z|^3)$, is the local weight of the local trivialization $s$, we denote $$\psi_k(z) = \tau\(\frac{\sqrt{k} z}{\log k}\) \phi_1(z),\, \vol_{M,k}(z)  = \(1+ \tau\(\frac{\sqrt{k}z }{\log k}\) \(\vol_M(z)-1\)\)$$ with some cut-off function $\tau$ equal to $1$ near the origin,
    $$\begin{aligned}
    &u_{k,\ell,z}(w^1,\ldots, w^\ell) \\
    &~= \(e^{2k(\psi_k(w^\ell)-\psi_k(z))}\vol_{M,k}^{-1}(w^\ell)\vol_{M,k}(z)-1\)\prod_{j=1}^{\ell-1} \(e^{2k(\psi_k(w^j)-\psi_k(w^{j+1}))}\vol_{M,k}^{-1}(w^j)\vol_{M,k}(w^{j+1})-1\).\end{aligned}$$
    and 
    $$\Delta_{\ell} = \sum_{j=1}^n \frac{1}{\lambda_j} \(\sum_{\nu=1}^\ell \frac{\dr^2}{\dr w_j^\nu \dr \overline{w}_j^\nu} +\sum_{\nu<\mu}\frac{\dr^2}{\dr \overline{w}_j^\nu \dr w_j^{\mu}} \).$$
\end{thm}

We say that $P$ is a pseudodifferential operator of order $m$ on $L$ if $P$ is a continuous operator from $\cC^\infty(M, L)$ to $\cC^\infty(M,L)$ and for each $p\in M$, take a Chern-Moser trivialization $(z,s)$ on $U$ centered at $p$,  we can write $P(u\otimes s^k) = (e^{k\phi} P_s(e^{-k\phi} u))\otimes s^k$ with $P_s $ is a pseudodifferential operator of order $m$. We will denote the symbol of $P_s$ by $p(z,\theta) \sim \sum_{j=0}^\infty p_j(z,\theta)\in S^m_{1,0}(U\times\bbR^{2n}) $ with $p_j \in S^{m-j}_{1,0}(U\times \bbR^{2n})$, where $S^N_{1,0}(U\times \bbR^{2n}) $ denotes the Hörmander symbol space of order $N$ and of type $(1,0)$ \cite{grigis1994microlocal}. In this thesis, we only use the type $(1,0)$, so we omit the subscript in the following. Moreover, we say that $P$ is classical if $p_j(z,\lambda\theta)= \lambda^{m-j} p_j(z,\theta)$ for all  $\lambda\ge 1$, $\theta \ne 0 $ and $j$.
\begin{thm}
\label{full-pseuToeplitz}
    Using the notations above. Let $M$ be a compact complex manifold $M$ with Hermitian line bundle $(L,h^L)$, $R^L>0$ on $M$, and $P$ is a classical pseudodifferential operator of order $m$ on $L$. There are $t_j(z) \in \cC^\infty(M)$ for $j=1,2,\ldots$ such that $$T_{P,k}(z)\sim \sum_{j=0}^\infty k^{n+m-j} t_j(z)$$ in $\cC^\infty(M)$.

    More precisely, let $p\in M$ and let $(z,s)$ be the Chern-Moser trivialization on $U$ centered at $p$, we can compute the Toeplitz operator by
    $$
    \begin{aligned}
        \chi_k T_{P,k}\chi_k(0,0) &= k^m B_{k\phi_0}(0,0)\sum_{j,\ell=0}^\infty k^{-j}\sum_{j_1,j_2=0}^\infty \frac{(2k)^{-j_1}}{j_1!}\frac{(2k)^{-j_2}}{j_2!}\sum_{\ell_1+\ell_2=\ell}\\
        &\Delta_{\ell+1}^{j_1} \[u_{k,\ell_1,w^{\ell_1+1}}\frac{e^{k\psi_k(w^{\ell_1+1})-k\psi_k(w^{\ell_1+2})}}{\sqrt{\vol(w^{\ell_1+1})}}(\Delta_{z,\theta}^{j_2}\wt{b_j})(w^{\ell_1+1},w^{\ell_1+2})u_{k,\ell_2,0}\](0),
    \end{aligned}
    $$ 
    if $\ell_2 = 0 $, $w^{\ell_1+2} = 0$, where $u_{k,\ell,z}$ and $\Delta_{\ell} $ are same as above, $p(z,\theta)\sim \sum_{j=0}^\infty p_j(z,\theta)$, $$b_\ell(z,\theta,x,y)= e^{-k\psi_k(z)+k\psi_k(y)}p_\ell(x,\theta)\tau_k(z)\sqrt{\vol(z)},$$
    and $$\Delta_{\theta,z} = \sum_{j=1}^n  2\lambda_j \(\frac{\dr^2 }{\dr \theta_j^1\dr \theta_j^1 }+\frac{\dr^2 }{\dr \theta_j^2\dr \theta_j^2 }\)-2i \sum_{j=1}^n \(\frac{\dr^2}{\dr z_j^1\dr \theta_j^1}+ \frac{\dr^2}{\dr z_j^2\dr \theta_j^2}\),$$
    with $\ti{z}_j=x_j$, $\theta_j^1(x,y) = 2i \lambda_j(x_j^1- \overline{y}_j)$, $\theta_j^2(x,y) = 2i \lambda_j(x_j^2- i \overline{y}_j)$ and  $\ti{\quad}$ means an almost analytic extension (see Theorem \ref{stationary phase 2}).
\end{thm}

Using the notation introduced above and with more details provided in Section~\ref{Full Expansion}, let $P$ and $Q$ be pseudodifferential operators of order $m_1$ and $m_2$, respectively. The result of Theorem~\ref{full-pseuToeplitz} yields the following deformation quantization that generalizes the classical result of Kontsevich to pseudodifferential operators: 
\begin{thm}
$$[T_{P,k},T_{Q,k}] = T_{[P,Q],k} + \frac{1}{k}T_{D_1(P,Q)-D_1(Q,P),k}+O(k^{-2})$$ in the $L^2$-sense, where $[\cdot,\cdot] $ denotes the commutator of two pseudodifferential operators. The operator  $D_1(P,Q)-D_1(Q,P)$ is a pseudodifferential operator of order $m_1+m_2$, whose principal symbol satisfies $$\sigma_0(D_1(P,Q)-D_1(Q,P))(z,-Jd\phi(z))  = i\{p_0(z,-Jd\phi(z)),q_0(z,-Jd\phi(z))\}_L$$ where $\sigma_0$ denotes the principal symbol and $\{\cdot,\cdot\}_L$ is the Poisson bracket associated with the line bundle $L$ (cf. Definition~\ref{Line bundle Poisson bracket}). 
\end{thm}

\section{Preliminaries}
\label{Preliminaries}
\subsection{Some Standard Notations}
We denote $\bbN:=\{1,2,\ldots\}$ by the set of natural numbers and $\bbN_0 = \bbN\cup \{0\}$. For a multi-index $\ga = (\ga_1,\ldots, \ga_n)\in \bbN_0^n$, we denote $|\ga|:= \sum_{j=1}^n \ga_j$ and $\ga!= \ga_1!\cdots \ga_n!$. In $\bbC^n$, $B(z,r) $ denoted the open ball with radius $r>0$ and center $z\in \bbC^n$. We also denote the natural inner product on $\bbC^n$ to be $\<z|w\>_{\bbC^n}= \sum_{j=1}^n z_j \ol{w}_j$ where $z= (z_1,\ldots, z_n) $ and $w= (w_1,\ldots,w_n)$. Let $\delta_{j,\ell} $ be Kronecker delta, which equal to $1$ if $j=\ell$ and $0$ otherwise.

\subsection{Complex Manifold}
\label{Complex Manifold}

We omit certain foundational concepts of complex geometry, including complex manifolds, Hermitian metrics, and holomorphic sections. For detailed discussions on these topics, we refer to  \cite{huybrechts2005complex} and   \cite{demailly1997complex, demailly2012analytic}

Let $M$ be an $n$-dimensional complex manifold. We introduce some standard notations of function spaces. Let $U$ be an open subset of $M$. Denote $\cC^\infty(U)= \gO(U)$ the space of smooth functions on $U$ and $\cC^\infty_c(U) = \gO_c(U) $ be the subspace of $\cC^\infty(U) $ whose elements have compact support in $U$. Let $H^0(U)$ denote the space of holomorphic functions on $U$. Let $E\rightarrow M$ be a complex vector bundle, we denote $\cC^\infty(U,E)= \gO(U,E)$ the space of smooth sections over $U$ and $\cC^\infty_c(U,E) = \gO_c(U,E) $ be the subspace of $\cC^\infty(U,E) $ whose elements have compact support in $U$. Also, we denote $\cD'(U, E)$ the space of distribution sections of $E$ over $U$ and $\cE'(U, E) $ the subspace of $\cD'(U, E)$ whose elements have compact support in $U$. For $E$ a holomorphic vector bundle over $U$, we let $H^0(U, E)$ be the space of holomorphic sections of $ E$ over $U$. 

Let $M$ be a complex manifold with complex structure $J: TM\rightarrow TM$. Hence, $J$ induces an eigenspace decomposition $\bbC TM = T^{(1,0)}M\oplus T^{(0,1)} M$ where $\bbC TM$ denotes the complexified tangent bundle, $T^{(1,0)} M $ is the $\sqrt{-1}$-eigenspace and $T^{(0,1)}M$ is $-\sqrt{-1}$-eigenspace of $J$. We also have the eigenspace decomposition for $\bbC T^* M = T^{*,(1,0)} M\otimes T^{*,(0,1)} M$ where $\bbC T^*M$ denotes the complexified cotangent bundle, $T^{*,(1,0)}M $ and $T^{*,(0,1)}M $ are the dual bundles of $T^{(1,0)}M $ and $T^{(0,1)}M $ respectively. Denote the $(p,q)$-forms $\gO^{(p,q)}(M) = \cC^\infty(M,T^{*,(p,q)}M)$ where $T^{*,(p,q)}M = \bigwedge^p T^{*,(1,0)}M\oplus\bigwedge^q T^{*,(0,1)} M$, also the $(p,q)$-form valued in vector bundle $E$ over $M$, $\gO^{(p,q)}(M,E) = \cC^\infty(M, T^{*,(p,q)}M \otimes E) $.

When we fix a positive $(1,1)$-form $\go$ on $M$, which called Hermitian form, then it can induce an inner product $\<\cdot|\cdot\>_{\go}$ on $ T^{(1,0)}M$ and can be extended to $\bbC TM$. Locally, we have $$\go= \sum_{j,\ell=1}^\infty \go_{j,\ell} dz_j\wedge d\ol{z}_\ell,\, \<\frac{\dr}{\dr z_j}| \frac{\dr }{\dr z_\ell}\>_{\go} = \go_{j,\ell},\,\<\frac{\dr}{\dr z_j} | \frac{\dr }{\dr\ol{z}_\ell}\>_{\go} =0, \text{ and }\<\frac{\dr}{\dr \ol{z}_j}|\frac{\dr}{\dr \ol{z}_j}\>_{\go}  = \ol{\go}_{j,\ell}. $$
By duality, we know $$\<dz_j|dz_\ell\>_{\go}  = \go^{\ell,j},\, \<dz_j| d\ol{z}_\ell\>_{\go}  =0,\text{ and } \<d\ol{z}_j| d\ol{z}_\ell\>_{\go}  = \ol{\go}^{\ell,j}$$ where $\(\go^{j,\ell}\)_{j,\ell=1}^n $ is the inverse of $\(\go_{j,\ell}\)_{j,\ell=1}^n$. For $(0,q)$-form, we define $$\<u_Idz_I| v_Jdz_J\>_{\go,q}= u_I\ol{v}_J \frac{\det (\ol{\go}^{j_\ell,i_\ell})_{\ell=1}^n}{q!} $$ with $I=\{i_1,\ldots, i_q\},\,J=\{j_1,\ldots,j_q\}$ and $dz_I = d\ol{z}_{i_1} \wedge \cdots \wedge d\ol{z}_{i_q}$. Also, we will let $d\vol_M$ be the volume form of $M$ induced by $\go$, i.e. $d\vol_M = \frac{\go^n}{n!}$.

For a holomorphic line bundle $L\rightarrow M$ with Hermitian metric $h^L$, given a holomorphic trivialization $\{(U_j, s_j)\}$ of $L \rightarrow M$, where $\{U_j\} $ is an open covering of $M$ and $s_j: U_j\rightarrow L$ is a non-zero holomorphic section of $L$, we define the \textbf{local weight} $\phi_j\in \cC^\infty(U_j, \bbR)$ of $h^L$ with respect to $s_j$ by $|s_j|_{h^L}^2= e^{-2\phi_j} $. For $u,v\in \cC^\infty(U_j, L)$, we can write $u= f_j s_j $ and $v= g_js_j$ for some $f_j, \, g_j \in \cC^\infty(U_j, \bbC) $ with $\<u|v\>_{h^L}= f_j\ol{g}_je^{-2\phi_j} $. For $L^k := L^{\otimes k}$, there is a natural Hermitian metric $h^{L^k}$ on $L^k$ induced by $h^L$. Thus, the local weight of $L^k$ is $k\phi_j$ if $\phi_j$ is the local weight of $L$. Let $R^L(h^L) = R^L$ be the \textbf{curvature form} on $M$ induced by $h^L$. Recall that for a holomorphic trivialization $(U,s)$ of $(L, h^L) $ over $U$ and $\phi$ be the local weight of $h^L$, the curvature form is locally given by $$R^L(h^L) = 2\dr \dbar \phi = 2 \sum_{j,\ell = 1} ^n \frac{\dr^2 \phi}{\dr z_j \dr \ol{z}_\ell } dz_j \wedge d\ol{z}_\ell.$$ We can define the \textbf{curvature operator} $\dot{R}^L \in \cC^\infty(M,\End (T^{(1,0)} M )) $ as $$\sqrt{-1}R^L(h^L)(p) (v\wedge\ol{w}) = \< \dot{R^L}(h^L)(p) v| w\>,\, v,w\in T^{1,0}M.$$ 

\begin{lem}
\label{chern-moser coordinate}
    Let $M$ be a complex manifold, and let $(L,h^L)$ be a Hermitian line bundle over $M$. Fix a point $p\in M$, we can choose a local complex coordinate $(z_1, \ldots, z_n) $ on an open neighborhood $U\subset M$ of $p$ and a holomorphic trivializing section $s\in H^0(U, L) $ such that $z_j(p)=0 $, $\<\frac{\dr}{\dr z_j}|\frac{\dr }{\dr z_\ell}\>  = \delta_{j,\ell}+O(|z|)$ for $j,\ell=1,\ldots, n$, and $|s(z)|^2_{h^L} = e^{-2\phi(z)} $  with local weight $\phi(z) = \sum_{j=1}^n \lambda_{j,p} |z_j|^2+ O(|z|^3)$, where $2\lambda_{j,p} $ are eigenvalues of curvature operator at $p$. We usually denote $\phi_0(z) = \sum_{j=1}^n \lambda_{j, p} |z_j|^2$.
\end{lem}
We will call $(z,s)$ a Chern-Moser trivialization on $U$ centered at $p$.
\begin{proof}
    Let $w = (w_1,\ldots, w_n)$ be a complex coordinate of $M$ on an open neighborhood $U$ of $p$ such that $\<\frac{\dr}{\dr w_j}|\frac{\dr }{\dr w_\ell}\>  = \delta_{j,\ell}+O(|w-w(p)|)$ and $\dot{R^L}(h^L)(p)\frac{\dr}{\dr w_j} = \mu_j\frac{\dr}{\dr w_j}$. Consider $z_i(x) = w_i(x) - w_i(p) $, then $z(p) = 0 $ and still be a complex coordinate. We consider a small enough open neighborhood $V\subset U$ of $p$ such that $L|_V$ can be trivialized by $s_1: V\rightarrow L$. Suppose $|s_1|_{h^L}^2 = e^{-2\phi_1(z)}$ with $\phi_1\in \cC^\infty(V,\bbR)$. By Taylor expansion of $\phi_1$, 
    we have 
    $$
    \begin{aligned}
    \phi_1(z) &= \phi_1(0)+ \sum_{j=1}^n [\frac{\dr \phi_1}{\dr z_j }(0) z_j + \frac{\dr \phi_1}{\dr \ol{z}_j} (0)\ol{z}_j] \\
    &+ \sum_{j,\ell=1}^n [\frac{\dr^2 \phi_1}{\dr z_j \dr z_\ell}(0) z_jz_\ell+\frac{\dr^2 \phi_1}{\dr z_j \dr \ol{z}_\ell}(0) z_j\ol{z}_\ell+\frac{\dr^2 \phi_1}{\dr \ol{z}_j \dr \ol{z}_\ell}(0) \ol{z}_j\ol{z}_\ell]+ O(|z|^3).
    \end{aligned}
    $$
    Consider $\psi(z) = \phi_1(0)+ 2\sum_{j=1}^n \frac{\dr \phi_1}{\dr z_j }(0) z_j + \sum_{j,\ell=1}^n \frac{\dr^2 \phi_1}{\dr z_j \dr z_\ell}(0) z_jz_\ell$ and $s= e^{\psi(z)} s_1(z) $. The local weight $\phi$ of $h $ with respect to $s$ then satisfy $$\phi(z) = \phi_1 - \frac{\psi(z) + \ol{\psi(z) }}{2} = \sum_{j,\ell}  \frac{\dr^2 \phi_1 }{\dr z_j \dr \ol{z}_\ell} (0) z_j \ol{z}_\ell + O(|z|^3).$$

    Since that we take the coordinate $z$ is some translate of $w$, and thus $\begin{pmatrix}
        \frac{\dr^2\phi}{\dr z_j \dr\ol{z}_\ell}(0)
    \end{pmatrix}_{j,\ell}$, is diagonal Hermitian matrix.
    By the above discussion, we have the expected section $s$ and corresponding local weight $\phi$, which proves this lemma. 
\end{proof}

Recall that, given a Hermitian metric $\go$ on $M$, we have a volume form $d\vol_M= \frac{\go^n}{n!}$ and inner product on $(0,q)$-form.  Using this, we can introduce the \textbf{$L^2$-inner product} on various function spaces. First, for  $\gO_c(M)$, we have $$(f|g)_M = \int \<f|g\>_{\bbC} d\vol_M \text{ for all } f, g\in \gO_c(M).$$  Let $L^2(M)$ be the completion of $\gO_c(M)$ with respect to the inner product $(\cdot | \cdot)_M$ and its corresponding norm $\|\cdot \|_M$. For any open set $U$ in $M$, denote $$(f|g)_U = \int_U \<f|g\>_{\bbC} d\vol_M \text{ for all } f,g\in \gO_c(U). $$ $L^2(U,d\vol_M)$ is the completion of $\gO_c(U) $ with respect to $(\cdot | \cdot)_U$ and its corresponding norm $\|\cdot \| _U $. If $d\vol_M = dV_{\bbC^n}$ on $U$, we will denote $L^2(U,d\vol_M)$ as $L^2(U)$ for simplify.

For the sections, we can similarly define $$\(u|v\)_{k} = \int \<u|v\>_{h^{L^k}} d\vol_M\text{ for all } u,v\in \gO_c(M,L^k). $$ Then $\|\cdot \|_{k}$, $L^2(M,L^k)$, $\(\cdot| \cdot\)_{k,U}$, $\|\cdot\|_{k,U}$, and  $L^2(U,L^k)$ are also similar defined. 

For $(0,q)$-form,  we locally have  $$ \<u|v\>_{h^{L^k},\,\go,q} = \<\hat{u}\otimes s| \hat{v}\otimes t\>_{h^{L^k},\,\go} = \<\hat{u}|\hat{v}\>_{\go,q}\<s|t\>_{h^{L^k}} \text{ for } u,v\in \cC^\infty(M, T^{*,(0,q)}M\otimes L^k).$$ Then we also have an inner product defined by 
$$(u,v)_{M,q} =\int_{M}\<u|v\>_{\go,q} d\vol_M \text{ for all } u,v\in \gO_c^{(0,q)} (M)$$
and
\begin{equation}
\label{inner product}
(u|v)_{k,q} = \int \<u|v\>_{h^{L^k},\,\go,\,q} d \vol_M \text{ for all } u,v\in \gO_c^{(0,q)} (M,L^k).
\end{equation}
Also, the related norm and $L^2$-space will use similar notation to the section case. 

Note that we will fix a Hermitian form $\go$ in the following, and sometimes, for simplicity, we will ignore the subscripts $q$, $\go$, and $M$.

\subsection{Differential Operators}
\label{Differntial Operators}

In this section, we will just discuss manifold cases. The orbifold cases are straightforward. We will refer  \cite{huybrechts2005complex},  \cite{demailly1997complex,demailly2012analytic} and  \cite{ma2007holomorphic} for detail either manifold or orbifold cases.

Let $\dbar_k^{(q)}: \gO^{(0,q)}(M,L^k)\rightarrow \gO^{(0,q+1)} (M,L^k) $ be the \textbf{Cauchy-Riemann operator}, which can extend to $\Dom \dbar_k^{(q)} = \{u\in L^2_{(0,q)} (M,L^k)| \dbar_k^{(q)} u \in L^2_{(0,q+1)}(M,L^k)\}\subset L^2_{(0,q)}(M,L^k)$, and $\dbar_k^{*,(q+1)}: \Dom \dbar_k^{*,(q+1)} \subset L^2_{(0,q+1)} (M,L^k) \rightarrow L^2_{(0,q)}(M,L^k)$ be the Hilbert space adjoint with respect to $(\cdot | \cdot)_{k}$. The \textbf{Kodaira Laplacian} is define by $$\kla_k^{(q)} = \dbar_k^{*,(q+1)}\dbar_k^{(q)}+ \dbar_k^{(q-1)} \dbar_k^{*,(q)}: \gO^{(0,q)}(M,L^k)\rightarrow \gO^{(0,q)}(M,L^k)$$ which have the semi-positive and self-adjoint \textbf{Gaffney extension} \cite{Gaffney1955HilbertSM}, \cite{ma2007holomorphic} \begin{equation} \label{Gaffney Extension}\kla_k^{(q)}: \Dom \kla_k^{(q)} \subset L^2_{(0,q)}(M, L^k) \rightarrow L^2_{(0,q)}(M, L^k)\end{equation} with $$\Dom \kla_k^{(q)} := \{u\in \Dom \dbar_k^{(q)} \cap \Dom \dbar_k^{*,(q)} | \dbar_k^{(q)} u \in \Dom \dbar_k^{*,(q+1)}, \, \dbar_k^{*,(q)} u \in \Dom \dbar_k^{(q-1)}\}.$$ In the following, we will denote $\kla_k^{(0)} $ as $\kla_k $. 

For any local holomorphic trivialization $s: U \rightarrow L$ with local weight $\phi$, we can have a \textbf{$L^2$ identification} as follows. Through this trivialization, for any $u_1$, $u_2\in \gO_c^{(0,q)}(U)$ and $u_1\otimes s^k$, $u_2\otimes s^k\in \gO^{(0,q)}(U, L^k)$, $$\(u_1 \otimes s^k| u_2 \otimes s^k\)_k = \( u_1 e^{-k\phi}| u_2e^{-k\phi}\)_{U}.$$ We localize the smooth sections from $U$ to $L^k$ by a unitary map $$ L^2_{(0,q)}(U, L^k) \simeq  L^2_{(0,q)}(U,d\vol_M) \text{ by } s^k\otimes u \leftrightarrow e^{-k\phi} u .$$ We denote this identification $\tau: s^k\otimes u \rightarrow e^{-k\phi}u$. This unitary map gives an idea to define the \textbf{localized Cauchy-Riemann operator} $\dbar_{k,s}$ such that $$s^k\otimes e^{k\phi}(\dbar_{k,s}^{(q)} e^{-k\phi}u) =  \dbar_k^{(q)}(s^k \otimes u), $$  and its adjoint $\dbar_{k,s}^*$, which satisfy $$s^k\otimes e^{k\phi}(\dbar_{k,s}^{*,(q)} e^{-k\phi}u) =  \dbar_k^{*,(q)}(s^k \otimes u), $$ with respect to $(\cdot|\cdot)_U$, and the \textbf{localized Kodaira Laplacian} $$\kla_{k,s}^{(q)}= \dbar_{k,s}^{*,(q+1)}\dbar_{k,s}^{(q)}+ \dbar_{k,s}^{(q-1)} \dbar_{k,s}^{*,(q)}.$$

For $e\in T^{*,(0,1)}M $, $e^{\wedge,*}: T^{*,(0,q+1)} M\rightarrow T^{*,(0,q)}M $ denote the adjoint of $e^\wedge: T^{*,(0,q)}M \rightarrow T^{*,(0,q+1)}M$ with respect to $\<\cdot | \cdot \>$, i.e. $\<e^\wedge u | v\> = \<u|e^{\wedge,*} v\>$ for all $u\in T^{*,(0,q)} M $ and $v\in T^{*,(0,q+1)}M$. Then we can easily check that 
$$\dbar_{k,s} = \dbar + k(\dbar \phi)^{\wedge},$$
$$\dbar_{k,s}^* = \dbar^*+ k(\dbar \phi)^{\wedge,*}$$
and 
$$\kla_{k,s} =\dbar_{k,s}^*  \dbar_{k,s}+\dbar_{k,s}\dbar_{k,s}^* $$
where $\dbar^*$ is the formal adjoint of $\dbar$ with respect to the standard $L^2$-inner product in $U$. For $p\in U$, we take $\{e_j\}_{j=1}^n$ be an orthonormal frame of $T^{*,0,1}_p U$, let $Z_j$ be the corresponding dual of $e_j$ and $Z_j^*$ be its formal adjoint of $Z_j$ with respect to the standard inner product. Then we have the $$\dbar_{k,s} = \sum_{j=1}^n (e_j^\wedge\circ (Z_j+ kZ_j(\phi)) +(\dbar e_j)^\wedge e_j^{\wedge,*})  $$
and $$\dbar_{k,s}^* = \sum_{j=1}^n (e_j^{\wedge,*}\circ (Z_j^*+ k\ol{Z}_j(\phi)) +e_j^\wedge(\dbar e_j)^{\wedge,*}).$$ We can give an expression for the localized Kodaira Laplacian as follows, 
    \begin{equation}
    \label{Localized Kodaira Laplacian}
    \begin{aligned}
        \kla_{k,s}^{(q)} & = \dbar_{k,s}^{(q-1)}\dbar^{(q),*}_{k,s}+  \dbar_{k,s}^{(q+1),*}\dbar_{k,s}^{(q)}\\
        &= \sum_{j,\ell=1}^n (e_j^\wedge\circ (Z_j+ kZ_j(\phi)) +(\dbar e_j)^\wedge e_j^{\wedge,*}) (e_\ell^{\wedge,*}\circ (Z_\ell^*+ k\ol{Z}_\ell(\phi)) +e_\ell^\wedge(\dbar e_\ell)^{\wedge,*})\\
        &\quad+ (e_\ell^{\wedge,*}\circ (Z_\ell^*+ k\ol{Z}_\ell(\phi)) +e_\ell^\wedge(\dbar e_\ell)^{\wedge,*}) (e_j^\wedge\circ (Z_j+ kZ_j(\phi)) +(\dbar e_j)^\wedge e_j^{\wedge,*})\\
        &= \sum_{j,\ell =1}^n e_j^\wedge e_\ell^{\wedge,*}(Z_j+ kZ_j(\phi))(Z_j^*+ k\ol{Z}_j(\phi)) + e_\ell^{\wedge,*}e_j^\wedge(Z_j^*+ k\ol{Z}_j(\phi))(Z_j+ kZ_j(\phi))+ O(1)\\
        &= \sum_{j=1}^n (Z_j+ kZ_j(\phi))(Z_j^*+ k\ol{Z}_j(\phi)) + \sum_{j,\ell=1}^n e_j^\wedge e_\ell^{\wedge,*}[(Z_j+ kZ_j(\phi)),(Z_\ell^*+ k\ol{Z}_\ell(\phi))]+O(1).\\
    \end{aligned}
    \end{equation}
    where $O(1) $ denote the term of the form $$
    \begin{aligned}
    \sum_{j,\ell} &(e_j^\wedge e_\ell^\wedge(\dbar e_\ell)^{\wedge,*}) (Z_j+ kZ_j(\phi))+(e_\ell^{\wedge,*} (\dbar e_j)^{\wedge}e_j^{\wedge,*}) (Z_\ell^*+ k\ol{Z}_\ell(\phi))\\&+ (\dbar e_j)^{\wedge}e_j^{\wedge,*} e_\ell^\wedge(\dbar e_\ell)^{\wedge,*}+e_\ell^\wedge(\dbar e_\ell)^{\wedge,*} (\dbar e_j)^{\wedge}e_j^{\wedge,*}\\
    &= \sum_{j,\ell} a_{j,\ell}(Z_j+ kZ_j(\phi))+b_{j,\ell} (Z_\ell^*+ k\ol{Z}_\ell(\phi))+ c_{j,\ell}\end{aligned}$$ which is vanishing after we scaling. 

We will give some estimates about the Kodaira Laplacian in the Sobolev space. We refer  \cite{ma2007holomorphic},  \cite{taylor2010partial} and  \cite{taylor2010partial2} for the detail. As in the smooth case, for $m\in \bbN \cup \{\infty\} $, denote $\cC^m(M, E) $ and $\cC^m_0(M, E) $ be the space $\cC^m$ section of $E$ on $M$. For a compact set $K\subset M$ we denote $\cC^m_0(K,M):= \{ s\in \cC^m_0(M,E)| \supp(s) \subset K\}.$ Then we define \textbf{ $\cC^m(M)$-norm} and \textbf{Sobolev norm} $\|\cdot \|_m$ for $s\in \cC^m_0 (M,E)$ such that $$|s|_{\cC^m(M)} =\sum_{\ell= 0} ^m \sup_M| (\nb^E)^\ell s| $$
and $$\|s\|^2_m = \sum_{\ell= 0} ^m \| (\nb^E)^\ell s\|^2 $$
where $\|\cdot\|$ is the $L^2$-norm defined in subsection \ref{orbifold}. 
Let $K$ be a compact subset of $M$, then the \textbf{Sobolev space} $H^m_0(K,E)$ is the completion of $\cC^\infty_0(K,E) $ with respect to $\|\cdot \|_k$. Assume $K= \ol{U}$ be a compact subset and $U$ be the relative open subset, we have $H^m_0(K, E)$ coincides with the closure of $\cC^\infty_0(U,E) $ with respect to $\|\cdot \|_m$ and denote by $H^m_0(U,E)$.

Recall that a generalized Laplacian is an operator $H$ of the form $$H = \Delta^E + Q$$ where $\Delta^E$ is the Bochner Laplacian on $E$ and $Q$ is a Hermitian section of $\End(E)$. Then we have the following estimates. 
\begin{thm}[Elliptic estimate, cf.  \cite{ma2007holomorphic}]
    Let $K\subset M$ be compact. For any $m\in \bbN$, there exists $C_1, C_2>0$ such that for any $s\in H^{m+2}_0(K,E) $, we have $$\|s\|_{m+2} \le C_1 \| Hs\|_m + C_2\|s\|.$$
\end{thm}

\begin{thm}[Gårding's inequality, cf.  \cite{ma2007holomorphic}]
    Let $K\subset M$ be compact. There exists $C>0$ such that for any $s\in H^1_0(K,E)$, we have $$\|s\|^2_1 \le C(\<Hs, s\>+\|s\|^2).$$
\end{thm}

\subsection{Bergman Kernel}
\label{pre of Bergman kernel}

To introduce the kernels, we first recall the Spectral Theorem, 
\begin{thm}[Spectral Theorem \cite{davies1995spectral}]
    Let $H$ be a self-adjoint operator on a Hilbert space $\sH$ with spectrum $S$. Then there exists a finite measure $\mu$ on $S\times \bbN$ and an unitary operator $$U: \sH\rightarrow L^2:=L^2(S\times \bbN, d\mu) $$ with the following properties. If $h: S\times \bbN \rightarrow \bbR$ is the function $h(s,n) = s $ then the element $\xi$ of $\sH$ lies in $\Dom H$ if and only if $h\cdot U(\xi) \in L^2$. We have $$UHU\invs\psi = h\psi$$ for all $\psi \in U\{\Dom (H)\} $, and also $$Uf(H)U\invs \psi = f(h) \psi $$ for all $f\in \cC_0(\bbR) $ and $\psi \in L^2(S\times \bbN, d\mu)$
\end{thm}
From the self-adjoint and semi-positive of Kodaira Laplacian $\kla_k$ and the spectral Theorem, we can define the \textbf{spectral projection} $$P_{k,\mu_k}= \chi_{[0,\mu_k]}(\kla_k): L^2(M,L^k) \rightarrow E_{\mu_k}(M, L^k) $$ where $\chi_I(x)  = \begin{cases}
    1& \text{if } x\in I \\
    0 & \text{otherwise}
\end{cases}$, $\chi_{[0,\mu_k]} (\kla_k)$ denotes the functional calculus of $\kla_k$ with respect to $\chi_{[0,\mu_k]}$  and $E_{\mu_k}(M,L^k)$ denote the image of $P_{k,\mu_k} $. 
For $\mu_k=0$, we denote $P_k = P_{k,0}$ and called \textbf{Bergman Projection}. Using the following Schwartz Kernel Theorem, we can introduce the\textbf{ Bergman kernel} and \textbf{Spectral kernel}, denote by $P_k(x,y)$ and $P_{k,\mu_k}(x,y) $. 
\begin{fact}[Schwartz Kernel Theorem \cite{hörmander2003analysis}]
\label{Schwartz Kernel Theorem}
\quad
\begin{enumerate}
    \item Let $A: \cC^\infty(M,E) \rightarrow \cD'(M,F)$ be a linear continuous operator. Then there exists a unique distribution $K\in \cD'(M \times M, F \boxtimes E^*)$, called the Schwartz kernel distribution such that $A = A_K$, i.e., $A(u)(v) = K(v\otimes u) $ for any $u\in \cC^\infty_0(M,E),\, v\in \cC^\infty_0(M, F^*)$. 
    \item Let $\cE'(M,E)$ be the space of distributions of compact support. Assume that $A: \cE'(M,E) \rightarrow \cC^\infty(M, F) $ is linear continuous. Then there exists a smooth kernel $K\in \cC^\infty(M\times M, F \boxtimes E^*) $ called the Schwartz kernel of $A$ such that $(Au)(x) = \int_M K(x,y) u(y) d\vol_y,\, \text{for any }u\in \cE'(M,E). $ 
    \end{enumerate}
\end{fact}
By using elliptic partial differential equations, it is easy to check that the Bergman projection satisfies the condition in the second statement of Fact \ref{Schwartz Kernel Theorem}, thus the Bergman kernel is smooth.

With the $L^2$ identification, we have localized Kodaira Laplacian $\kla_{k,s}$, and we can define the \textbf{localized spectral projection} $$P_{k,\mu_k,s}: L^2_{comp}(U,d\vol_M)\rightarrow L^2(U,\vol_M) $$ which satisfy $P_{k,\mu_k}(s^k \otimes u) = s^k \otimes e^{k\phi}P_{k,\mu_k,s}(e^{-k\phi}u ),  $ also we have the \textbf{localized Bergman projection} $P_{k,s}$. We denote $P_{k, \mu_k,s} (x,y)$ the \textbf{localized spectral kernel}, and $P_{k,s}(x,y)$ the  \textbf{localized Bergman kernel}. 

\subsection{Stationary Phase Formula}
\label{Stationary Phase Formula}

To derive several results in this thesis, we rely on the following lemmas, which are adapted from Chapter 7 of \cite{hörmander2003analysis}.

\begin{lem}[cf  \cite{hörmander2003analysis} Theorem 7.7.5]
    Let $K\subset \bbR^n$ be a compact set, $X$ an open neighborhood of $K$, and $k$ a positive integer. If $u\in \cC^{2k}_0(K) $, $f\in \cC^{3k+1}(X) $ and $\operatorname{Im} f\ge 0 $ in $X$, $\operatorname{Im} f(x_0) = 0$, $f'(x_0)=0 $, $\det f''(x_0) \ne 0$, $f'\ne 0 $ in $K\backslash \{x_0\}$ then $$|\int u(x) e^{i\omega f(x)} dx - e^{i\go f(x_0)} (\det (\frac{\go f''(x_0)}{2\pi i}))^{\frac{-1}{2}} \sum_{j< k} \go^{-j} L_j u|\le C \go^{-k} \sum_{|\ga|\le 2k} \sup |D^\ga u|, \, \go>0.$$
    Here $C$ is bounded when $f$ stays in a bounded set in $\cC^{3k+1}(X)$ and $\frac{|x-x_0|}{|f'(x)|}$ has a uniform bound. With $$g_{x_0}(x) = f(x) -f(x_0) - \frac{1}{2}\<f''(x_0)(x-x_0), x-x_0\>$$ which vanishes of third order at $x_0$ we have $$L_j u = \sum_{\nu - \mu= j}\sum_{2\nu \ge 3 \mu} i^{-j} 2^{-\nu} \<f''(x_0)\invs D , D\>^\nu \frac{(g^\mu_{x_0} u)(x_0)}{\mu! \nu!}.$$
    This is a differential operator of order $2j$ acting on $u$ at $x_0$. The coefficients are rational homogeneous functions of degree $-j$ in $f''(x_0),\ldots, f^{(2j+2)}(x_0)$  with denominator $(\det f''(x_0))^{3j}$. In every term, the total number of derivatives of $u$ and of $f''$ is at most $2j$.
\end{lem}

We also require another variant. To state it clearly, we pause to introduce some concepts and notation. For details, we refer to \cite{MelinSj75} and appendix A in \cite{hsiao2008projections}.
\begin{Def}
    \label{analytic extension}
    Let $U$ be an open subset of $\bbC^N$. For $f\in \cC^\infty(U)$, we say that $f$ is almost analytic if for any compact subset $K$ of $U$ and $M\in \bbN$, there is a constant $C_{K,M}>0$ such that $$ |\dbar f|^2 \le C_{K,M}\( |\operatorname{Im} z|\)^M,\, \forall z\in K.$$
\end{Def}

\begin{Def}
    For $f_1,f_2\in \cC^\infty(U)$, we say that $f_1$ and $f_2$ are equivalent if for any compact subset $K$ of $U$ and $M\in \bbN$, there is a constant $C_{K,M}>0$ such that $$|(f_1-f_2)(z)| \le C_{K,M}\( |\operatorname{Im} z|\)^M,\, \forall z\in K.$$
\end{Def}

\begin{pro}
    Let $U\subset \bbC^N$ be an open set, and let $U_{\bbR^N} = U \cap \bbR^N$. If $f\in\cC^\infty (U_{\bbR^N})$, then there exists an almost analytic extension of $f$, unique up to equivalence. 
\end{pro}

\begin{pro}
\label{extension critical value}
    Let $f(x,w) $ be a complex valued smooth function in a neighborhood of $(0,0)$ in $\bbR^{n+m}$ with $$\operatorname{Im}f\ge 0, \operatorname{Im}f(0,0)= 0, f_x'(0,0) = 0, \det f''_{xx} (0,0) \ne 0. $$ Let $\ti{f}(z,w)$, $z =x+iy$, $w\in\bbC^m$, denote an almost analytic extension of $f$ to a complex neighborhood of $(0,0)$ and let $z(w)$ denote the solution of $$\frac{\dr \ti{f}}{\dr z}(z(w),w) =0 $$ in a neighborhood of $0\in \bbC^m$. Then for $w$ is real $$\frac{\dr }{\dr w} (\ti{f}(z(w),w) - \frac{\dr \ti{f}}{\dr w } (z,w)|_{z=z(w)}$$ vanishes to infinite order at $0\in \bbR^m$. Moreover, there is a constant $c>0$ such that near the origin and $w\in \bbR^m$ we have $$\operatorname{Im}\ti{f}(z(w), w) \ge c|\operatorname{Im}z(w)|^2,$$ and $$\operatorname{Im}\ti{f}(z(w),w)\ge c\inf_{x\in\gO}\(\operatorname{Im} f(x,w)+ |d_xf(x,w)|^2\)$$ where $\gO$ is some open neighborhood of the origin of $\bbR^n$. We call $\ti{f}(z(w),w)$ the corresponding critical value. 
\end{pro}

Then we can go back to our stationary phase formula. 
\begin{thm}
\label{stationary phase 2}
    Let $f(x,w)$ be as in Proposition \ref{extension critical value}. Then there are neighborhoods $U$ and $V$ of the origin in $\bbR^n$ and $\bbR^m$ respectively and differential operators $L_{f,j}$ in $x$ of order $\le 2j$ which are $\cC^\infty$ functions of $w\in V$ such that  $$|\int u e^{it f} dx - e^{it \ti{f}(z(w),w)} (\det (\frac{t \ti{f}''_{zz}(z(w),w)}{2\pi i}))^{\frac{-1}{2}} \sum_{j< N} t^{-j} (L_{f,j}\ti{u})(z(w),w) |\le C_N t^{-N- \frac{n}{2}}, \, t\ge 1,$$ where $u\in \cC^\infty_c(U\times V)$. Here $\ti{f}$ and $\ti{u}$ are almost analytic extensions of $f$ and $u$ respectively. The function $(\det (\frac{t \ti{f}''_{zz}(z(w),w)}{2\pi i}))^{\frac{-1}{2}}$ is the branch of the square root of  $(\det (\frac{t \ti{f}''_{zz}(z(w),w)}{2\pi i}))^{-1}$ which is continuously deformed into $1$ under the homotopy $s\in [0,1]\rightarrow i^{-1} (1-s)\ti{f}_{zz}(z(w),w) +sI \in GL(n,\bbC)$.
\end{thm}

\section{Asymptotic Behavior of Bergman Kernels on Complex Manifolds}
\label{Standard Results}
In this section, our goal is to derive the large $k$-behavior of the Bergman kernel using the scaling method in three different cases.

In subsection \ref{Euclidean Case}, we use the scaling method to derive the leading term of the Bergman kernel in $\bbC^n$. In subsection \ref{Compact Complex Manifold}, we use the spectral gap of the localized Kodaira Laplacian and the result in the Euclidean Case to get the leading term of the Bergman kernel. In subsection \ref{No Spectral Gap Manifold Case}, we use spectral project and spectral kernel to derive the leading term of the Bergman kernel near points with positive curvature.

\subsection{Euclidean Case}
\label{Euclidean Case}
Let $\phi\in \cC^\infty(\bbC^n,\bbR)$, $\phi(z)= \phi_0(z) + \phi_1(z) = \sum_{j=1}^n \gl_j |z_j|^2 + O(|z|^3) $, $\phi_1=O(|z|^3)$, and $\tau\in \cC^\infty_c(\bbC^n,[0,1])$ with $\tau = 1$ on $B(0,1)$ and vanish outside $B(0,2)$. We will denote the weight $\hat{\phi}_k = \phi_0 + \tau(\frac{\sqrt{k}z}{\log k}) \phi_1$, the Hermitian form $\hat{\go}_k =i \sum_{j,\ell=1}^n( \delta_{j,\ell} + \tau(\frac{\sqrt{k}z}{\log k })O(|z|))dz_j\wedge d\ol{z}_\ell$, and the volume form $d\hat{vol}_{k}(z) = (1+ \tau(\frac{\sqrt{k}z}{\log k}) O(|z|)) d\lambda(z)$ where $d\gl(z) = i^n d z_1 d\ol{z}_1 \cdots d z_n d\ol{z}_n $ is the standard volume form. Notice that $\hat{\phi}_k = \phi_0$ and $d\hat{\vol}$ is flat outside $B(0,\frac{2\log k}{\sqrt{k}})$. The Hermitian metric on $\bbC^n$ induces an inner product on $\gO^{(0,q)}_c(\bbC^n)$ which given by $$\(u|v\)_{k\hat{\phi}_k}^{(q)} := \int_{\bbC^n} \<u|v\>_{\hat{\go}_k, q} e^{-2k\hat{\phi}_k} d\hat{\vol}_k(z) $$ where $\<u_Id\ol{z}_I| v_Jd\ol{z}_J\>_{\hat{\go}_k}= u_I\ol{v}_J \frac{\det (\ol{\hat{\go}_k}^{j_\ell,i_\ell})_{\ell=1}^n}{q!} $ and $\(\hat{\go}_k^{i,j}\)_{i,j=1}^n$ is the inverse of $\(\hat{\go}_{k,i,j}\)_{i,j=1}^n = \(\<dz_i|dz_j\>_{\hat{\go}_k}\)_{i,j=1}^n$.

We denote $\dbar_{k\hat{\phi}_k}^{(q),*}$ to be the formal adjoint of the Cauchy-Riemann operator with respect to this inner product. That is $$\(\dbar^{(q)} u | v\)_{k\hat{\phi}_k}^{(q+1)} = \(u|\dbar_{k\hat{\phi}_k}^{(q+1),*}v\)_{k\hat{\phi}_k}^{(q)},\, \forall u\in \gO_c^{(0,q)} (\bbC^n) \text{ and } v\in \gO^{(0,q+1)}(\bbC^n). $$ Then the  Gaffney extension of  Kodaira Laplacian with respect to $k\hat{\phi}_k$ is given by $\kla_{k\hat{\phi}_k}^{(q)} = \dbar^{(q-1)} \dbar_{k\hat{\phi}_k}^{(q),*}+\dbar_{k\hat{\phi}_k}^{(q+1),*}\dbar^{(q)}$. To simplify the notations, we will denote $\kla_{k\hat{\phi}_k} = \kla_{k\hat{\phi}_k}^{(0)}$ and $\dbar_{k\hat{\phi}_k}^* =\dbar_{k\hat{\phi}_k}^{(1),*} $.

In this section we will focus on the weighted $L^2$ space $L^2_{(0,q)}(\bbC^n, k\hat{\phi}_k,d\hat{\vol}_{k})$, which is the completion of $\gO^{(0,q)}_c(\bbC^n)$ with respect to $(\cdot| \cdot)_{k\hat{\phi}_k}^{(q)}$,  and the Bergman space $ H^0(\bbC^n,k\hat{\phi}_k,d\hat{\vol}_{k})  = L^2(\bbC^n ,k\hat{\phi}_k,d\hat{\vol}_{k}) \cap \ker \dbar$. The natural projection from $L^2(\bbC^n,k\hat{\phi}_k,d\hat{\vol}_{k})$ to $ H^0(\bbC^n,k\hat{\phi}_k,d\hat{\vol}_{k})$ is called the Bergman projection, denote by $B_{k\hat{\phi}_k}$.
We denote $B_{k\hat{\phi}_k}(x,y) $ to be its distribution kernel, that is $$(B_{k\hat{\phi}_k}u)(x) = \int_{\bbC^n} B_{k\hat{\phi}_k}(x,y) u(y) d\hat{\vol}_k.$$ The Bergman projection has the following properties
\begin{equation}
\label{Bergman kernel condition}
    \begin{cases}
        \kla_{k\hat{\phi}_k} B_{k\hat{\phi}_k} u = 0 & \forall u\in L^2(\bbC^n,k\hat{\phi}_k,d\hat{\vol}_{k}) \\
        B_{k\hat{\phi}_k} u = u &\forall u\in H^0(\bbC^n,k\hat{\phi}_k,d\hat{\vol}_{k})
    \end{cases},
\end{equation}
and its kernel is smooth. 

We will use a semi-classical method to study the large $k$ behavior of the Bergman kernel. First, we define the scaling Kodaira Laplacian $\kla_{(k\hat{\phi}_k)}$ by the Partial differential equation $$\delta_k(\kla_{k\hat{\phi}_k} u ) = k \kla_{(k\hat{\phi}_k)} \delta_k(u),$$ where $\delta_k(u) (z) = u\(\frac{z}{\sqrt{k}}\), u\in \cC^\infty(\bbC^n) ,$ and the scaling Bergman kernel $$B_{(k\hat{\phi}_k)} (x,y) = k^{-n} B_{k\hat{\phi}_k}\(\delta_k(x) , \delta_k(y)\).$$ From this kernel, we can define an operator $$B_{(k\hat{\phi}_k)}:L^2(\bbC^n, \delta_k(k\hat{\phi}_k),d\hat{\vol}_{(k)}) \rightarrow H^0(\bbC^n,\delta_k(k\hat{\phi}_k),d\hat{\vol}_{(k)})$$ which is called the scaling Bergman projection. This operator satisfy similar properties as (\ref{Bergman kernel condition}), 
\begin{equation}
    \begin{cases}
         \kla_{(k\hat{\phi}_k)} B_{(k\hat{\phi}_k)} u = 0 & \forall u\in L^2(\bbC^n,\delta_k(k\hat{\phi}_k),d\hat{\vol}_{(k)}) \\
        B_{(k\hat{\phi}_k)} u = u &\forall u\in H^0(\bbC^n,\delta_k(k\hat{\phi}_k),d\hat{\vol}_{(k)})
    \end{cases}.
\end{equation}
Moreover, we know that $B_{(k\hat{\phi}_k)}$ is a bounded continuous operator, that is $$\|B_{(k\hat{\phi}_k)} u\|_{\delta_k(k\hat{\phi}_k)} \le \|u\|_{\delta_k(k\hat{\phi}_k)},\, \forall u\in L^2(\bbC^n, \delta_k(k\hat{\phi}_k),d\hat{\vol}_{(k)}).$$

We could also define a operator $ (I+\kla_{(k\hat{\phi}_k)})^{-1}$ since  $$I+\kla_{(k\hat{\phi}_k)}: \Dom \kla_{(k\hat{\phi}_k)}\subset L^2(\bbC^n, \delta_k(k\hat{\phi}_k),d\hat{\vol}_{(k)}) \rightarrow  L^2(\bbC^n, \delta_k(k\hat{\phi}_k),d\hat{\vol}_{(k)}) $$ is $1$-$1$ and onto, which is satifying  $$(I+\kla_{(k\hat{\phi}_k)})^{-s} B_{(k\hat{\phi}_k)} = B_{(k\hat{\phi}_k)}.$$  

Our main theorem in the Euclidean case is that the large $k$ behavior of the scaling Bergman kernel is similar to the well-known Bergman kernel with respect to $\phi_0(z) = \sum_j \lambda_j|z_j|^2$. That is 
\begin{thm}[Main Theorem in subsection \ref{Euclidean Case}]
\label{EUC-Main-Thm}
    If $\dr\dbar \phi>0 $ on $\bbC^n$, then $$\lim_{k\to \infty} B_{(k\hat{\phi}_k)}(x,y) = B_{\phi_0} (x,y) $$ locally uniformly in $\cC^\infty$-topology on $\bbC^n\times \bbC^n$.
\end{thm}

Before the proof of the Main theorem, we will provide the following lemmas about the estimate of the Bergman kernel and Bergman projection.

\begin{lem}
\label{Euc-lem-1}
    Let $u\in\cC^\infty_c(\bbC^n)$. For all $\ell \in \bbN$, $U\Subset W\subset \bbC^n$ be open subsets,  there is a constant $C_{\ell,U,W}$ independent of $k$ such that
    $$\| B_{(k\hat{\phi}_k)}  u \|_{2\ell,U}\le C_\ell\|u\|_{-2\ell,W}$$ for all $k\gg 1$. 
    Sometimes, we omit the subscripts $U$ and $W$ for simplicity.
\end{lem}

\begin{proof}
    Fix $s\in \bbN$. For any bounded open sets $U\Subset V \Subset W\Subset D$, by applying Gårding inequality, we have the following estimate 
    \begin{equation}
    \label{ECU-GARDING}
    \begin{aligned}
        \|B_{(k\hat{\phi}_k)}u \|_{2s, U} &\approx \| (I+ \kla_{(k\hat{\phi}_k)})^{-s} B_{(k\hat{\phi}_k)} u \|_{2s,U}\\
        &\le C_1( \| B_{(k\hat{\phi}_k)} u \| _{0 , V} + \|(I+ \kla_{(k\hat{\phi}_k)})^{-s} B_{(k\hat{\phi}_k)} u\|_{0,V} )\\
        &\le C_2 \| B_{(k\hat{\phi}_k)} u \|_{0,V} 
    \end{aligned}
    \end{equation}
    where $C_1$ and $C_2$ are positive constant which is independent of $k$ and $u$. From this we know $$B_{(k\hat{\phi}_k)}: L^2_{comp}(D,\delta_k(k\hat{\phi}_k),d\hat{\vol}_{(k)}) \rightarrow H^{2s}_{loc}(D,\delta_k(k\hat{\phi}_k),d\hat{\vol}_{(k)}),\, \forall s\in\bbN_0.$$ Then we take the adjoint operator, we have $$B_{(k\hat{\phi}_k)}:H^{-2s}_{comp} (D,\delta_k(k\hat{\phi}_k),d\hat{\vol}_{(k)}) \rightarrow L^2_{loc}(D,\delta_k(k\hat{\phi}_k),d\hat{\vol}_{(k)}), $$
    and $$\|B_{(k\hat{\phi}_k)}u\|_{0,V} \lesssim \|u\|_{-2s,W},$$ which complete the proof of this lemma.
\end{proof}

\begin{lem}
\label{Euc-lem-2}
    Fix a compact subset $K\subset\bbC^n$. For all $\ell \in \bbN$, there exists a constant $C_\ell $ independent of $k$ and $(x,y)$ such that $$\sup_{\genfrac{}{}{0pt}{}{\ga+\gb \le \ell}{ x,y\in K}} |\dr_x^\ga \dr_y^\gb B_{(k\hat{\phi}_k)}(x,y)|\le C_\ell.$$ 
\end{lem}

\begin{proof}
    For any $x, y\in K\subset U\subset \bbC^n$, fixed a cut-off function $\chi$ support near $0$, and denote $\chi_\varepsilon(z) = \varepsilon^{-2n}\chi(\varepsilon^{-1} z) $. Using Fourier transform we have for any open subset $V$, $\| \chi_\varepsilon\|_{-2n-1,V}\le \| \chi_\varepsilon\|_{-2n-1}< C$ for some $C$ is independent of $ \varepsilon$. From the support of $\chi_\varepsilon$ will be contained in $K$ for $\varepsilon$ small enough then by Lemma \ref{Euc-lem-1}, we have
    $$
    \begin{aligned}
    |B_{(k\hat{\phi}_k)}(x,y)|&= \lim_{\varepsilon\to 0 }|\int B_{(k\hat{\phi}_k)} (x,z) \chi_\varepsilon(z-y)  d\lambda(z)|\\
    &= \lim_{\varepsilon\to 0 }|\int_U B_{(k\hat{\phi}_k)} (x,z) \chi_\varepsilon(z-y)  d\lambda(z)| \le \|B_{(k\hat{\phi}_k)}\|_{2n+1,U} \| \chi_{\varepsilon}\|_{-2n-1,U-y},
    \end{aligned}
    $$ both have uniform bound independent of $(x,y)\in K\times K$ and $k$. For $\cC^\ell $ norm, we using the fact $$| \dr_x^\ga \dr_y^\gb B_{(k\hat{\phi}_k)}(x,y) | = \lim_{\varepsilon, \delta\to 0 } |\int B_{(k\hat{\phi}_k)} (z,w) (\dr_z^\ga \chi_\delta(z-x))(\dr_w^\gb \chi_{\varepsilon}(w-y))d\lambda(z)d\lambda(w)|.$$ This has the uniform bound give by $\|\dr_z^\ga \chi_\delta(z-x)\|_{-2n-1 +|\alpha|,U}\| \dr_w^\gb \chi_{\varepsilon}(w-y)\|_{-2n-1+ |\beta|,U}  $ which is independent of $(x,y)\in K\times K$ and $k$.
    Thus $\sup_{\genfrac{}{}{0pt}{}{\ga+\gb \le \ell}{ x,y\in K}} |\dr_x^\ga \dr_y^\gb B_{(k\hat{\phi}_k)}(x,y)|\le C_\ell$  for some constant $C_\ell $ independent of $(x,y) $ and $k$.
\end{proof}

From Lemma \ref{Euc-lem-2}, we know $\{B_{(k\hat{\phi}_k)}(x,y)\}$ is locally uniform bounded and equi-continuous. Thus, by Arzelà–Ascoli theorem, there exists a subsequence of $B_{(k\hat{\phi}_k)}(x,y)$ that converges locally uniformly to some $B(x,y)$. Thus, we can define an operator $B$ from the limit kernel. We will use Hörmander $L^2$-Estimate(cf.  \cite{hormander19652,Introductiontocpxanainserveralvariables-Hormander},  \cite{Demailly1982} and  \cite{chen2001partial}) to check that $B$ is exactly $B_{\phi_0}$.

We pause and introduce some notations. 
A function $\phi$ defined on $\gO\subset \bbC^n$ is plurisubharmonic if it is upper semi-continuous and subharmonic on each complex line, that is, for any complex line $L$, $\phi|_{\gO\cap L} $ is subharmonic. This is equivalent to that $\sum_{j,\ell=1} ^n \frac{\dr^2 \phi}{\dr z_j \dr\ol{z}_\ell} (p) \xi_j\ol{\xi}_\ell$ is semi-positive Hermitian form for all $p$ and we call the function is strictly plurisubharmonic if it is positive definite. We call a domain pseudoconvex if it has a strictly plurisubharmonic exhaustion function.
\begin{thm}[Hörmander $L^2$-Estimate \cite{berndtsson_l2-methods_nodate}]
\label{Hörmander L2 Estimate}
     Let $\gO$ be a pseudoconvex domain in $\bbC^n$, and let $\phi$ be a plurisubharmonic in $\gO$. Suppose $\phi= \psi + \xi$, where $\xi$ is an arbitrary plurisubharmonic function, and $\psi$ is a smooth, strictly plurisubharmonic function. Then for any $f$, a $\dbar$-closed $(0,1)$-form in $\gO$, we can solve $\dbar u= f $, with $u$ satisfying $$\int |u|^2 e^{-\phi} \le \int \sum \psi^{j\ol{k}} f_j\ol{f}_k e^{-\phi}$$ if right hand side is finite. Here $(\psi^{j\ol{k}}) = (\psi_{j\ol{k}})^{-1}$.
\end{thm}

\begin{proof}[Proof of Theorem \ref{EUC-Main-Thm}]
    We will prove that $B$ satisfy the Bergman projection condition \ref{Bergman kernel condition} for $\kla_{\phi_0} $. 
    First, we compute the Kodaira Laplacian $$\begin{aligned}
        \kla_{k\hat{\phi}_k} & =  \sum_{j,\ell=1}^n (e_j^\wedge\circ (Z_j+ kZ_j(\hat{\phi}_k))) (e_\ell^{\wedge,*}\circ (Z_\ell^*+ k\ol{Z}_\ell(\hat{\phi}_k)))\\
        &\quad+ (e_\ell^{\wedge,*}\circ (Z_\ell^*+ k\ol{Z}_\ell(\hat{\phi}_k))) (e_j^\wedge\circ (Z_j+ kZ_j(\hat{\phi}_k)))\\
        &= \sum_{j,\ell =1}^n e_j^\wedge e_\ell^{\wedge,*}(Z_j+ kZ_j(\hat{\phi}_k))(Z_j^*+ k\ol{Z}_j(\hat{\phi}_k)) + e_\ell^{\wedge,*}e_j^\wedge(Z_j^*+ k\ol{Z}_j(\hat{\phi}_k))(Z_j+ kZ_j(\hat{\phi}_k))\\
        &= \sum_{j=1}^n (Z_j+ kZ_j(\hat{\phi}_k))(Z_j^*+ k\ol{Z}_j(\hat{\phi}_k)) + \sum_{j,\ell=1}^n e_j^\wedge e_\ell^{\wedge,*}[(Z_j+ kZ_j(\hat{\phi}_k)),(Z_\ell^*+ k\ol{Z}_\ell(\hat{\phi}_k))]\\
        &=\sum_{j=1}^n (Z_j+ kZ_j(\hat{\phi}_k))(Z_j^*+ k\ol{Z}_j(\hat{\phi}_k))= \sum_{j=1}^n Z_jZ_j^* +k \lambda_j + k O(|z|) + kZ_j(\hat{\phi}_k) Z_j^*+ k\ol{Z}_j(\hat{\phi}_k)Z_j.
    \end{aligned}$$
    Then the the scaling is of the form $$\kla_{(k\hat{\phi}_k)} = \sum_{j} Z_j Z_j^* + \lambda_j + o(1). $$ Here $o(1)$ is $kZ_j(\hat{\phi}_k) Z^*_j + k \ol{Z}_j(\hat{\phi}_k)Z_j$ after scaling. Thus, $ \kla_{(k\hat{\phi}_k)}$ converges to $  \kla_{\phi_0}$, where the convergence is understood in terms of the coefficients of the corresponding partial differential equation. Then $\kla_{\phi_0} B = \lim_{k\to \infty} \kla_{(k\hat{\phi}_k)} B_{(k\hat{\phi}_k)}  =0$. Let $u = z^\alpha \in H^0 (\bbC^n,\phi_0)$ with $\alpha\in \bbN^n_0$, take $v_k (z)= u(z) \chi(\frac{z}{\log k}) $ where $\chi$ is a cut-off function. Then $v_k \in L^2(\bbC^n, \delta_k(k\hat{\phi}_k),d\hat{\vol}_{(k)}) $ and $\|\dbar v_k\|_{L^2_{(0,1)}(\bbC^n, \delta_k(k\hat{\phi}_k),d\hat{\vol}_{(k)})} = O(k^{-\infty} ) $. By Hörmander $L^2$-estimate, there are $r_k $ such that $\dbar r_k = \dbar v_k$ and $$\|r_k\|_{L^2(\bbC^n, \delta_k(k\hat{\phi}_k),d\hat{\vol}_{(k)})} \le \|\dbar v_k \|_{L^2_{(0,1)}(\bbC^n, \delta_k(k\hat{\phi}_k),d\hat{\vol}_{(k)})} = O(k^{-\infty}) .$$ Then $u_k = v_k -r_k$ is in $H^0(\bbC^n,\delta_k(k\hat{\phi}_k),d\hat{\vol}_{(k)}) $ and converge to $u$ as $k\to \infty$.

    Now for every $g\in \cC^\infty_c(\bbC^n) $, we have $(B_{(k\hat{\phi}_k)}u_k | g)_{\delta_k(k\hat{\phi}_k)} = (u_k | g)_{\delta_k(k\hat{\phi}_k)}$ since $u_k \in H^0(\bbC^n,\delta_k(k\hat{\phi}_k)) $. By the fact $u_k$ converges $u$ pointwisely, we have $(u_k| g)_{\delta_k(k\hat{\phi}_k)}$ converge to  $(u|g)_{\phi_0}$. Let $\rho\in\cC^\infty_c(\bbC^n,[0,1])$ and $\rho=1 $ on $B(0,1) $ and $\rho=0$ outsider $B(0,2)$. Then we consider $\rho_\ell= \rho(\frac{z}{\ell})$, 

    $$
    \begin{aligned}
    (B_{(k\hat{\phi}_k)}u_k | g)_{\delta_k(k\hat{\phi}_k)}  =(B_{(k\hat{\phi}_k)}\rho_\ell u_k | g)_{\delta_k(k\hat{\phi}_k)}+(B_{(k\hat{\phi}_k)}(1-\rho_\ell)u_k | g)_{\delta_k(k\hat{\phi}_k)} \\
    \end{aligned}
    $$ 
    The first term converge to $\(B\rho_\ell u | g\)_{\phi_0}$ and we estimate the second term $$\begin{aligned}
        \left|(B_{(k\hat{\phi}_k)}(1-\rho_\ell)u_k | g)_{\delta_k(k\hat{\phi}_k)}\right|^2 &\le \|g\|_{\delta_k(k\hat{\phi}_k)}\|B_{(k\hat{\phi}_k)}(1-\rho_\ell)u_k\|_{\delta_k(k\hat{\phi}_k)}\\
        &\le \|g\|_{\delta_k(k\hat{\phi}_k)}\|(1-\rho_\ell)u\|_{\delta_k(k\hat{\phi}_k)}.
    \end{aligned}$$
    Since for any $\varepsilon>0,\, \delta>0$ we can take $k_0$ and $\ell_0$ such that for all $k>k_0$ and $\ell>\ell_0$, we have $|\tau(\frac{\sqrt{k}z}{\log k})\phi_1(z)|\le C\frac{\log^3 k}{\sqrt{k}}\le \delta$ and $\|(1-\rho_\ell) u\|_{\phi_0}\le \varepsilon$. Thus, the second term can be sufficiently small.
    Combine this observation with $\rho_\ell u $ converge to $u$, we conclude that  $B_{(k\hat{\phi}_k)} u_k$ converge to $Bu$ weakly. Thus $B=B_{\phi} $ in the distribution sense.
\end{proof}

\begin{rem}
\label{general phi}
    We can replace $\hat{\phi}_k$ to be $$\phi_k(z) = \sum_{j} \lambda_{j,k} |z_j -p_{k,j}|^2 + \tau(\frac{\sqrt{k}}{\log k} (z-p_{k,j}))O(|z-p_k|^3),$$
    $$\hat{\go}_k =i \sum_{j,\ell=1}^n( \delta_{j,\ell} + \tau(\frac{\sqrt{k}}{\log k }(z-p_k))O(|z-p_j|))dz_j\wedge d\ol{z}_\ell$$ and 
    $$d\hat{vol}_{k}(z) = (1+ \tau(\frac{\sqrt{k}}{\log k}(z-p_k)) O(|z-p_j|)) d\lambda(z) $$
    where $\lambda_{j,k} \rightarrow \lambda_j$ and $p_k\rightarrow 0$ as $k\to \infty$, this will also preserve the locally uniform convergence of the scaling Bergman kernel with the scaling respect to $0$. 
\end{rem}

\begin{rem}
    For the quadratic weight $\phi_0(z) = \sum_{j=1}^n \lambda_j |z_j|^2$ with $\lambda_j>0$, the Bergman projection $B_{\phi_0}: L^2(\bbC^n,\phi_0) \rightarrow H^0(\bbC^n,\phi_0)$ is well-known. We can compute the Bergman kernel $B_{\phi_0}(x,y) $ by using the orthonormal basis $\{c_\alpha z^\alpha\}_{\alpha \in \bbN_0^n}$. That is $$B_{\phi_0}(x,y) =  \(\frac{1}{\pi}\)^n \prod_{j=1}^n \lambda_j  e^{ 2\sum_j \lambda_j (x_j \ol{y}_j -|y_j|^2)}. $$
    For the detailed computation, we refer to \cite[Proposition 11]{hou2022asymptotic}. 
\end{rem}

\subsection{Compact Manifold Case}
\label{Compact Complex Manifold}
Let $M$ be a complex manifold and $(L,h^L)$ be a Hermitian line bundle over $M$. We denote $L^k:= L^{\otimes k}$ the $k$-th tensor of the line bundle and  $\(\cdot |\cdot \)_k$ the $L^2$-inner product of $L^2(M,L^k)$ which is the set of all $L^2$ finite setion frome $M$ to $L^k$, see subsection \ref{Complex Manifold}. Let $\dbar_k$ be the Cauchy-Riemann operator and $\dbar_k^*$ be the formal adjoint with respect to the previous inner product. The Kodaira Laplacian is denoted by $\kla_k = \dbar_k^*\dbar_k$, and $P_k$ is the Bergman projection with its distribution kernel $P_k(x,y) $, which is called the Bergman kernel. 

For a local trivialization $s: U \rightarrow L$ with local weight $\phi(z) = \sum_j \lambda_j |z_j|^2 + O(|z|^3) $. We have an identify between $L^2(U, L^k)$ and the $L^2$ space $L^2(U,\vol_M)$ by 
$$
\begin{aligned}
\tau:\gO_c(U,L^k) &\rightarrow \gO_c(U)\\
u\otimes s^k &\mapsto  ue^{-k\phi}.
\end{aligned}
$$
This identify also give the localized Cauchy-Riemann operator $\dbar_{k,s}$ satisfy $\tau(\dbar_k u) = \dbar_{k,s}\tau(u)$, its formal adjoint $\dbar_{k,s}^*$ with respect to the $L^2$-inner product in $L^2(U,\vol_M)$ and the localized Kodaira Laplacian $\kla_{k,s}$. Also, the localized Bergman projection $P_{k,s} $ defined by $$ P_k(\hat{u}\otimes s^k) =( e^{k\phi}P_{k,s} e^{-k\phi} \hat{u}) \otimes s^k$$ and localized Bergman kernel $P_{k,s}(x,y)$. As in subsection \ref{Differntial Operators}, we can derive the local expression of the localized Kodaira Laplacian. In the following, we will assume $M$ is compact and $L$ is positive. Then use this expression and the localized Bergman kernel to asymptotically study the large $k$ behavior of the Bergman kernel.

We will assume the manifold $M$ is compact and the line bundle $(L,h^L)$ is positive in the following.

\begin{thm}
\label{MFD-Main-Thm}
Let $M$ be a compact complex manifold with a positive Hermitian line bundle $L$. Let $P_k$ is the Bergman projection $P_k: L^2(M, L^k) \rightarrow H^0(M,L^k)$. For any $p\in M $, there is a local trivialization $(s, U)$ with Chern-Moser coordinate on $U$ center at $p$ and the Bergman kernel has the following expansion on the neighborhood $B(0, \frac{\log k }{\sqrt{k}})$
         $$P_{k,s}(x,y) = k^n \det(\frac{\dot{R}_L}{2\pi})(p) e^{k \sum_j \lambda_j (2x_j \ol{y}_j - |x_j|^2- |y_j|^2)} +o(k^n). $$
\end{thm}

Before presenting the proof of the theorem (see after Proposition \ref{Mfd Hörmander L2}), we first establish key properties of the Bergman kernel and Kodaira Laplacian in this case.

\begin{pro}
\label{Manifodl Global Spectral Gap}
    For $q\ge 1$ and large enough $k$, there is a positive constant $C$ such that  $$(\kla_k^{(q)} u| u )_{k}\ge C k \|u\|_{k}^2$$ for all $u\in \gO^{(0,q)}(M, L^k)$.
\end{pro}

\begin{proof}
    First, we have this proposition for localized Kodaira Laplacian with compact support $(0,q)$-forms by directly computing the localized Kodaira Laplacian, see (\ref{Localized Kodaira Laplacian}). That is, for any local trivialization $(s,U)$, we have $$\(\kla_{k,s}^{(q)} u | u \)_{k,U}\ge C k\|u\|_{k,U}^2 $$ for all $u\in \gO_c^{(0,q)} (U)$. Then we use the partition of unity to prove the global version.

 Let $\{\sigma_j\}_{j=1}^m$ be a partition of unity with $(s_j,U_j)$ be the local trivialization with $\supp \sigma_j \Subset U_j$
\begin{equation}
\label{partiton spectral gap}
\begin{aligned}
    \(\kla_k^{(q)} u|u\)_k&= \(\kla_k^{(q)} (\sum_{j=1}^M \sigma_j )u |(\sum_{j=1}^M \sigma_j)u\)_k\\
    & = \sum_{j=1}^M \(\kla_k^{(q)} \sigma_j u| \sigma_j u \)_k + \sum_{j\ne \ell} \(\kla_k^{(q)} \sigma_j u| \sigma_\ell u \)_k\\
    &= \sum_{j=1}^m \(\kla_{k,s_j}^{(q)} \sigma_j u| \sigma_j u \)_k + \sum_{j\ne \ell} \(\kla_k^{(q)} \sigma_j u| \sigma_\ell u \)_k\\
    &\ge \sum_{j=1}^m C_j k \|\sigma_j u\|_k^2-  \sum_{j\ne \ell} \left|\(\kla_k^{(q)} \sigma_j u| \sigma_\ell u \)_k\right|\\
    &\ge C' k \|u\|_k^2-  \sum_{j\ne \ell} \left|\(\dbar_k \sigma_j u| \dbar_k\sigma_\ell u \)_k\right|-\sum_{j\ne \ell}\left|\(\dbar^*_k \sigma_j u| \dbar^*_k\sigma_\ell u \)_k\right|\\
\end{aligned}
\end{equation}

Then we estimate 
$$
\begin{aligned}
\left|\(\dbar_k \sigma_j u| \dbar_k\sigma_\ell u \)_k\right| &\le \left|\(L_s(\sigma_j)  \hat{u}| L_s(\sigma_\ell) \hat{u} \)_{k,s}\right|+\left|\(\sigma_j  (\dbar_{k,s} \hat{u})| L_s(\sigma_\ell)  \hat{u} \)_{k,s}\right|\\
&\quad+\left|\(L_s(\sigma_j)  \hat{u}| \sigma_\ell (\dbar_{k,s} \hat{u}) \)_{k,s}\right|+\left|\(\sigma_j \dbar_k u| \sigma_\ell \dbar_k u \)_{k,s}\right|\\
&\le C''\(\|u\|_k^2 + \|u\|_k\|\dbar_k u \|_k\) + \(\sigma_j \dbar_k u| \sigma_\ell \dbar_k u \)_k
\end{aligned}
$$
and similarly $$\left|\(\dbar^*_k \sigma_j u| \dbar^*_k\sigma_\ell u \)_k\right|\le C'''\(\|u\|_k^2 + \|u\|_k\|\dbar_k^* u \|_k\) + \(\sigma_j \dbar_k^* u| \sigma_\ell \dbar_k^* u \)_k$$
where $L_s$ is the first order differential operator, thus $L_s(\sigma_j)$ has uniform bound on $M$ and $j$. We may choose $C''=C'''$, say $C_1$.  For any $\varepsilon>0$, we have $$\|u\|_k\|\dbar_k u\|_k\le \frac{1}{2}\(\frac{1}{\varepsilon}\|u\|_k^2 + \varepsilon\|\dbar_k u\|_k^2\) $$ and 
$$ \|u\|_k\|\dbar_k^* u\|_k\le \frac{1}{2}\(\frac{1}{\varepsilon}\|u\|_k^2 + \varepsilon\|\dbar_k^* u\|_k^2\).$$ Since $\sum_{j=1}^m \sigma_j = 1 $ and $$\frac{1}{2} \ge \frac{\sigma_j+ \sigma_\ell}{2}\ge \sqrt{\sigma_j \sigma_\ell},$$ it follows that $\sigma_j\sigma_\ell\le \frac{1}{4}$. 

Then we can finish our proof, from \ref{partiton spectral gap}, 
$$
\begin{aligned}
\(\kla_k^{(q)} u|u\)_k&\ge C'k\|u\|_k^2 - C''\(\|u\|_k^2 + \|u\|_k\|\dbar_k u \|_k\) - \(\sigma_j \dbar_k u| \sigma_\ell \dbar_k u \)_k\\
&\quad-C'''\(\|u\|_k^2 + \|u\|_k\|\dbar_k^* u \|_k\) -\(\sigma_j \dbar_k^* u| \sigma_\ell \dbar_k^* u \)_k\\
&\ge  (C'k-2C_1- \frac{C_1}{\varepsilon})\|u\|_k^2\\
&\quad- (\frac{C_1\varepsilon}{2}+\frac{1}{4}) \|\dbar_k u \|_k^2- (\frac{C_1\varepsilon}{2}+\frac{1}{4})\|\dbar_k^* u \|_k^2.
\end{aligned}
$$
Thus $\(\kla_k^{(q)} u|u\)_k \ge \(1+ \frac{C_1\varepsilon}{2}+\frac{1}{4}\)^{-1}(C'k)\|u\|^2$ and we get $$\(\kla_k^{(q)} u|u\)_k \ge Ck\|u\|_k^2$$
\end{proof}

\begin{fact}
\label{Manifold global equation}
    Let $q\in \{1,\ldots,n\}$. There exists a bounded operator $N_k^{(q)}:\gO^{(0,q)}(M,L^k) \rightarrow \gO^{(0,q)}(M,L^k)$ which can be extend to $L^2(M,L^k) \rightarrow L^2(M,L^k)$ such that $$\begin{cases}
        N_k^{(q)} \kla_k^{(q)} = I &\text{on } \Dom \kla_k^{(q)}\\
        \kla_k^{(q)}N_k^{(q)} = I & \text{on } L^2_{(0,q)}(M,L^k)
    \end{cases}$$
    Consequently, from the above proposition and this fact, for $u\in L^2_{(0,q)}(M, L^k) $, we obtain the inequality $$C k\|N_k^{(q)} u\|_k^2\le\|\kla_k^{(q)} N_k^{(q)} u\|_k^2 = \|u\|_k^2.$$ It follows that $N_k^{(q)} = O(\frac{1}{k})$.
\end{fact}

\begin{pro}[Hörmander $L^2$-Estimate]
\label{Mfd Hörmander L2}
    For any $\dbar_k$-closed $(0,1)$-form $f$, there exists a solution $g\in \cC^\infty(M,L^k)$ for the equation $\dbar_k g = f$ such that $\|g\|^2_k \le \frac{1}{Ck}\|f\|^2_k$ for some constant $C$.
\end{pro}

\begin{proof}
    Combine the previous Proposition \ref{Manifodl Global Spectral Gap} and Fact\ref{Manifold global equation}, we have $$\|v\|_k \|N_k^{(q)}v\|_k \ge \(\kla_k^{(q)} N_k^{(q)} v | N_k^{(q)}v\)_k \ge C k \| N_k^{(q)}v\|_k^2,\, \forall v\in \gO^{(0,q)}(M,L^k).$$
    Then $$\|v\|_k \ge C k \| N_k^{(q)} v\|_k.$$
    For a  $\dbar$-closed $ f \in \Omega^{0,1} (M,L^k)$,
    $$\dbar_k^* \dbar_k N_{k}^{(1)}  f = N_k^{(1)} \kla_k^{(1)} \dbar^*_k \dbar_k N_k^{(1)}f= N_k^{(1)} \dbar^*_k \dbar_k \kla_k^{(1)} N_k^{(1)}f = N_k^{(1)} \dbar_k^* \dbar_k f = 0.$$
    By the Fact \ref{Manifold global equation} and the previous equation, $$ f=\kla_k^{(1)} N_k^{(1)} f = \dbar_k \dbar_k^* N_k^{(1)}  f.$$
    Thus we can take $g= \dbar_k^* N_k^{(1)} f$ and $$\|g\|_k^2 = \(\dbar_k^* N_k^{(1)} f | \dbar_k^* N_k^{(1)} f \)_k =\( \kla_k^{(1)} N_k^{(1)} f| N_k^{(1)} f\)_k \le \| f\|_k \| N_k^{(1)}f\|_k\le \frac{1}{Ck } \|f\|^2 .$$
\end{proof}

Then we can start proving the main theorem. 
First, we fix $p\in U\subset  M $ and get Chern-Moser trivialization $(z,s)$ on $U$ centered at $p$ with local weight $\phi= \phi_0+ \phi_1$ and a cut-off function $\tau =1 $ on $B(0,1)$ and vanish outside $B(0,2)$. We defined a weight $\hat{\phi}_k = \phi_0 + \tau(\frac{\sqrt{k}z}{\log k}) \phi_1$, the Hermitian form $\hat{\go}_k =i \sum_{j,\ell=1}^n( \delta_{j,\ell} + \tau(\frac{\sqrt{k}z}{\log k })O(|z|))dz_j\wedge d\ol{z}_\ell$, and a volume form $d\hat{vol}_{M,k}(z) = (1+ \tau(\frac{\sqrt{k}z}{\log k}) O(|z|)) d\lambda(z)$.

Similar as Euclidean case, we consider the \textbf{scaling localized Bergman kernel} $$P_{(k),s}(x,y) = k^{-n} P_{k,s} (\frac{x}{\sqrt{k}},\frac{y}{\sqrt{k}})$$ and $$B_{(k)}(x,y) = e^{\delta_k(k\hat{\phi}_k)(x)} P_{(k),s}(x,y) e^{-\delta_k(k\hat{\phi}_k)}.$$

We let $$(B_{(k)}g) (x) = \int_{B(0,\log k)} B_{(k)}(x,y)g(y)d\hat{\vol}_{M,(k)},\, \forall g\in \cC^\infty_c(B(0,\log k))$$ where $d\hat{\vol}_{M,(k)}(z) = d\hat{\vol}_{M,k} (\frac{z}{\sqrt{k}})$.  
Notice that for $u,v\in \cC^\infty_c(B(0,\log k))$, we have $$\(B_{(k)}u|v\)_{B(0,\log k),\delta_k(k\hat{\phi}_k)}  = \(u|B_{(k)}v\)_{B(0,\log k),\delta_k(k\hat{\phi}_k)}.$$
We then define 
 $$B_{(k)}: L^2(B(0,\log k ), \delta_k(k\hat{\phi}_k),d\vol_{M,(k)})\rightarrow L^2(B(0,\log k ), \delta_k(k\hat{\phi}_k),d\vol_{M,(k)})$$
by $B_{(k)} u = v$ such that for any test function $g\in \cC^\infty_c(B(0,\log k))$, we have $$\(B_{(k)} u | g\)_{B(0,\log k),\delta_k(k\hat{\phi}_k)} = \(v|g\)_{B(0,\log k),\delta_k(k\hat{\phi}_k)} = \(u|B_{(k)}g\)_{B(0,\log k),\delta_k(k\hat{\phi}_k)}.$$ 

Also defined \textbf{scaling localized Kodaira laplacian} defined by $$\delta_k(\kla_{k,s} u ) = k \kla_{(k),s} \delta_k(u), $$ denote
$$\kla_{(k)} = e^{\delta_k(k\hat{\phi}_k)} \kla_{(k),s}e^{-\delta_k(k\hat{\phi}_k)},$$ and still denote 
$$\kla_{(k)}:\Dom \kla_{(k)}\subset L^2(B(0,\log k),\delta_k(k\hat{\phi}_k), d\hat{\vol}_{M,(k)}) \rightarrow L^2(B(0,\log k),\delta_k(k\hat{\phi}_k),d\hat{\vol}_{M,(k)})$$  be the Gaffney extension. Note that $\kla_{(k)}u = \kla_{(k\hat{\phi}_k)} u$ on   $B(0,\log k)$ and $u\in \cC^\infty_c(B(0,\log k))$. Then we have 
$$
\(B_{(k)} u|\kla_{(k)} v\)_{B(0,\log k),\delta_k(k\hat{\phi}_k)} = \(u| B_{(k)}\kla_{(k)} v\)_{B(0,\log k),\delta_k(k\hat{\phi}_k)} =0,
$$ for all $u\in L^2(B(0,\log k ), \delta_k(k\hat{\phi}_k),d\vol_{M,(k)})$ and $v\in \cC^\infty_c(B(0,\log k ))$, which means $$\kla_{(k)} B_{(k)}= 0$$ in the sense of distribution. Since $\kla_{(k)}$ is self-adjoint, so is $I+\kla_{(k)}$. As $I+\kla_{(k)}$ is bounded below by $1$, the operator $$(I+ \kla_{(k)}) ^{-1}:L^2(B(0,\log k ), \delta_k(k\hat{\phi}_k),d\vol_{M,(k)})\rightarrow \Dom \kla_{(k)}$$ is well-define. Moreover, combine this with $\kla_{(k)} B_{(k)} = 0$, we have  $$(I+ \kla_{(k)}) ^{-s}B_{(k)}= B_{(k)}$$ on $L^2(B(0,\log k), \delta_k(k\hat{\phi}_k), d\hat{\vol}_{M,(k)}).$

Then for $K\Subset U\Subset V$ where $K$ is compact and $U,\,V$ are bound open subset of $\bbC^n$. Since for large $k$, the three subsets are contained in $B(0,\log k )$. Using the same idea in the Euclidean case, we have the following lemma 
\begin{lem}
\label{compact manifold lemma}
Let $u\in\cC^\infty_c(\bbC^n)$. For all $\ell \in \bbN$, $U\Subset V\subset \bbC^n$ be open subsets and $K\Subset U$ be a compact subset, there is a constant $C_{\ell,U,V}$ and $C_\ell$ independent of $k$ and $(x,y)\in K\times K$ such that for all $k\gg 1$, 
     $$\| B_{(k\hat{\phi}_k)}  u \|_{2\ell,U}\le C_{\ell, U,V}\|u\|_{-2\ell,V}.$$ 
     and 
     $$\sup_{\genfrac{}{}{0pt}{}{\ga+\gb \le \ell}{ x,y\in K}} |\dr_x^\ga \dr_y^\gb B_{(k\hat{\phi}_k)}(x,y)|\le C_\ell.$$ 
\end{lem}

The argument by Azerlà-Ascoli Theorem is still valid here, then there is a $B(x,y)$ that is the locally uniform limit of $B_{(k)}(x,y) $. The corresponding operator is denoted by $B$. We remain to prove that $B$ is exactly $B_{\phi_0} $.

First, since on a compact subset $K$ and for $k>k_0 $, we have $\kla_{(k)} = \kla_{(k\hat{\phi}_k)}$ on $K$ , thus $$\kla_{(k)}\rightarrow \kla_{\phi_0}$$ in sense of PDE and $$\kla_{\phi_0} B = \lim_{k\to \infty} \kla_{(k)} B_{(k)} = 0 .$$ 

For $\hat{u} = z^\alpha\in H^0(\bbC^n, \phi_0)$ with $\alpha\in\bbN_0^n$ and a suitable cut-off function $\chi$, we consider $\hat{h}_k(z) = \chi(\frac{\sqrt{k}}{\log k } z ) \hat{u}(\sqrt{k}z)$ and $h_k = \hat{h}_k\otimes s^k$. By the Hörmander $L^2$-Estimate, we have $h_k - r_k = v_k  \in H^0(M,L^k)$ with $\|r_k\|_k \le \|\dbar_k h_k\| = O(k^{-\infty})$. We have $$\hat{v}_k \otimes s^k = v_k = P_kv_k = e^{k\phi}P_{k,s}(e^{-k\phi} \hat{v}_k)\otimes s^k.$$
Let $\hat{u}_k = \delta_k(\hat{v}_k)$, we have the pointwise limit of $\hat{u}_k$ is $u$. As in the Euclidean case, for any test function $g$ we have $$\(B_{(k)} \hat{u}_k|g\)_{B(0,\log k), \delta_k(k\hat{\phi}_k)} =\(\hat{u}_k|g\)_{B(0,\log k), \delta_k(k\hat{\phi}_k)} +O(k^{-\infty})$$ and $$  \(u_k|g\)_{B(0,\log k), \delta_k(k\hat{\phi}_k)}\to \(u|g\)_{\phi_0}, \,\(B_{(k)} u_k|g\)_{B(0,\log k), \delta_k(k\hat{\phi}_k)}\to \(Bu|g\)_{\phi_0}.$$
Thus $B_{(k)}u_k$ converge to $Bu$ weakly and $B= B_\phi$ in distribution sense.

Thus we have $$B_{(k)}(x,y)  = B_{\phi_0}(x,y)+o(1) =  \(\frac{1}{\pi}\)^n \prod_{j=1}^n \lambda_j  e^{ 2\sum_j \lambda_j (x_j \ol{y}_j -|y_j|^2)} + o(1)$$
and
\begin{equation}
\label{manifold Bergman equation}
\begin{aligned}
P_{(k),s}(x,y) &=  e^{\delta_k(k\hat{\phi}_k)(x)}B_{(k)}(x,y) e^{-\delta_k(k\hat{\phi}_k)(y)} (x,y)\\  
&= e^{\delta_k(k\hat{\phi}_k)(x)} \(\(\frac{1}{\pi}\)^n \prod_{j=1}^n \lambda_j  e^{ 2\sum_j \lambda_j (x_j \ol{y}_j -|y_j|^2)} \)e^{-\delta_k(k\hat{\phi}_k)(y)} +o(1)\\
&= \det(\frac{\dot{R}_L}{2\pi})(p) e^{\sum_j \lambda_j (2x_j \ol{y}_j - |x_j|^2- |y_j|^2)}  +o(1)
\end{aligned}
\end{equation}

Furthermore, we obtain uniform behavior in the diagonal part.
\begin{thm}
    There exists $k_0, C$ such that for all $p\in M$ and $k\ge k_0$, $$|P_k(p,p)| \ge C k^n$$
\end{thm}
\begin{proof}
    Suppose not, there is $p\in M$ and $q_{k_j} \in M$ such that $q_{k_j} \to p$ with $$|k_j^{-n}P_{k_j}(q_{k_j}, q_{k_j})|\le \frac{1}{j}. $$
    By  \ref{general phi}, we know we can have a similar result that scaling the Bergman kernel at $q_{k_j}$ is also convergent to the known kernel corresponding to $p$. 
    Which is a contradiction to the assumption. 
\end{proof}

From the main theorem in the Manifold Case, we have the following application \cite{hsiao2014bergman}. See also the Theorem \ref{Kodaira-Baily Theorem} and its proof, which is the generalization of the following theorem. 
\begin{thm}[Kodaira Embedding Theorem]
    Let $M$ be a compact complex manifold. If there is a positive holomorphic line bundle $L$ over $M$, then $M$ can be holomorphic embedded into $\bbC\bbP^N$, for some $N\in \bbN$.
\end{thm}

\subsection{Manifold Case Without Spectral Gap}
\label{No Spectral Gap Manifold Case}

Follow the setting in the previous section without the compact and positive conditions. We will use the spectral projection, which is introduced in subsection \ref{pre of Bergman kernel}. We will denote $P_{k,\mu_k}= \chi_{[0,\mu_k]}(\kla_k)$ where $\chi_{[0,\mu_k]}$ is the characteristic function of $[0,\mu_k]$ and $E_{\mu_k}(X, L^k)$ be its image. Note that the Bergman kernel $P_k$ is just the case $\mu_k=0 $. We also use the identify between $L^2(U,L^k) $ and $L^2(U,\vol_M)$ to define the localized spectral projection $P_{k,\mu_k,s}$ and localized spectral kernel $P_{k,\mu_k,s}(x,y) $. We will assume $L$ is positive near a point $p$ and use the spectral projection to asymptotic the large $k$-behavior of the Bergman kernel near $p$.

\begin{thm}
\label{spectral kernel theorem}
   Let $M$ be a complex manifold with Hermitian line bundle $L$ and assume $\mu_k= o(k)$ and $\frac{k}{\mu_k}=o(\log k)$. For any $p\in M $, there is a local trivialization $(s, U)$ with Chern-Moser coordinate on $U$ center at $p$ and the spectral kernel has the following expansion on the neighborhood $B(0, \frac{\log k }{\sqrt{k}})$,
         $$P_{k,\mu_k,s}(x,y) = k^n P_{\phi_0}(\sqrt{k}x, \sqrt{k}y)+o(k^n)$$ where $B_{\phi_0}: L^2(\bbC^n,\phi_0)\rightarrow H^0(\bbC^n, \phi_0)$ and $P_{\phi_0} = e^{-\phi_0} B_{\phi_0} e^{\phi_0}$. 
         Moreover, if $L$ is positive over $U$, we have 
         $$P_{k,\mu_k,s}(x,y) = k^n \det(\frac{\dot{R}_L}{2\pi})(p) e^{k \sum_j \lambda_j (2x_j \ol{y}_j - |x_j|^2- |y_j|^2)} +o(k^n), $$ and if $L$ is non-positive over $p$, $P_{\phi_0}= 0$. 
\end{thm}

Let us review some properties of Spectral projection and Spectral kernel, then we will give the proof soon (see after Lemma \ref{non-compact lem 2}). From the definition of the spectral projection, we can easily obtain the following fact, 
\begin{fact}
\label{Spectral projection property}
For all $\ell\in \bbN$, $u\in L^2(M,L^k)$
\begin{equation}
    \label{mu-1}
    \|\kla_k^\ell   P_{k,\mu_k} u \|_k\le \mu_k^\ell \| u\|_k
\end{equation}  
and
\begin{equation}
    \label{mu-2}
    \|(I- P_{k,\mu_k}) u\|_k \le \frac{1}{\mu_k^\ell}\|\kla_k^\ell u\|_k.
\end{equation}
\end{fact}

Similar to the Euclidean case, we first notice that the localized spectral projection $P_{k,\mu_k,s}$ satisfying $$ P_{k,\mu_k} (u) = P_{k,\mu_k}(\hat{u}\otimes s^k) = (e^{k\phi}P_{k,\mu_k,s}(e^{-k\phi}\hat{u}))\otimes s^k$$ and its distribution kernel $P_{k,\mu_k,s}(x,y)$.
We can defined the \textbf{scaling localized spectral kernel} $$P_{(k), \mu_k, s} (x,y) = k^{-n} P_{k,\mu_k, s} \(\frac{x}{\sqrt{k}}, \frac{y}{\sqrt{k}}\)$$ and $$B_{(k),\mu_k}(x,y) = e^{\delta_k(k\hat{\phi}_k)(x)} P_{(k),\mu_k,s}(x,y) e^{-\delta_k(k\hat{\phi}_k)(x)}  $$ with $$B_{(k),\mu_k}:L^2(B(0,\log k),\delta_k(k\hat{\phi}_k), d\hat{\vol}_{M,(k)})\rightarrow L^2(B(0,\log k),\delta_k(k\hat{\phi}_k), d\hat{\vol}_{M,(k)})$$
defined by $$(B_{(k),\mu_k}u) (x) = \int_{B(0,\log k)} B_{(k),\mu_k}(x,y)u(y)d\hat{\vol}_{M,(k)},\, \forall u\in \cC^\infty(B(0,\log k))$$ and extend $B_{(k),\mu_k} $ to whole $L^2$ space by in the distribution sense as the discussion before Lemma \ref{compact manifold lemma}. As in the compact manifold case, we can define $\kla_{(k)}$.

Similarly to the Fact \ref{Spectral projection property}, we still have
 $$\|\kla_{(k)} B_{(k),\mu_k} u\|_{\delta_k(k\hat{\phi}_k)} \le \frac{\mu_k^\ell}{k^\ell} \|u\|_{\delta_k(k\hat{\phi}_k)}$$ and $$\|(I-B_{(k),\mu_k} )u\|_{\delta_k(k\hat{\phi}_k)}^2 \le \frac{k^\ell}{\mu_k^\ell} \(\kla_{(k)}^\ell u|u \)_{\delta_k(k\hat{\phi}_k)}.$$

As in the compact manifold case, for $K\Subset U\Subset V$ where $K$ is compact and $U,\, V$ are bounded open subsets of $\bbC^n$. Since for large $k$, the three subsets are contained in $B(0,\log k )$.
\begin{lem} If $\mu_k= o(k)$ and  $u\in \cC^\infty_c(\bbC^n)$. For all $U\Subset V$ and $\ell \in \bbN$, there is a constant $C_{\ell,U,V}$ such that for $k\gg 1$,
    $$\| B_{(k),\mu_k} u  \|_{2\ell, U}\le C_\ell \| u \|_{-2\ell,V}.$$
\end{lem}
\begin{proof}
    Let $U\Subset W \Subset V$. From the Gårding inequality and the property of spectral projection, we have 
    $$
    \begin{aligned}
    \| B_{(k),\mu_k}  u\|_{2\ell,U}&\le C^1(\|\kla_{(k)}^\ell B_{(k),\mu_k} u \|_{0,V} + \| B_{(k),\mu_k} u \|_{0,W})\\
    &\le C^1 \(\frac{\mu_k}{k}\)^\ell \|u\|_{0,W}+ C^2 \|(I+ \kla_{(k)})^{-\ell} u \|_{0,W}\\
    &\le C \|u\|_{-2\ell,V}.
    \end{aligned}
    $$
\end{proof}

From this lemma, we can immediately get the following lemma through the same method used in the Euclidean case.
\begin{lem}
\label{non-compact lem 2}
    If $\mu_k= o(k)$. Let $K$ be a compact subset of $\bbC^n$. For all $\ell \in \bbN$, there exists a constant $C_\ell $ independent of $k$ and $(x,y) \in K\times K$ such that for $k\gg 1$, $$\sup_{\genfrac{}{}{0pt}{}{\ga+\gb \le \ell}{ x,y\in K}} |\dr_x^\ga \dr_y^\gb B_{(k), \mu_k}(x,y)|\le C_\ell.$$ 
\end{lem}

Using the Azerlà-Ascoli Theorem argument, we have a locally uniform limit $B(x,y)$ of $B_{(k),\mu_k}(x,y)$. We can similarly define an operator $B$ using $B(x,y)$ as the kernel. The remaining thing again is to prove $B$ is exactly $B_{\phi_0}$.

\begin{proof}[Proof of Theorem \ref{spectral kernel theorem}]
First, from the directly computation, $\kla_{(k)}$ converges $\kla_{\phi_0}$. For any $u\in \cC^\infty_c(\bbC^n)$, we have $$ \|\kla_{\phi_0}B u\| = \lim_{k\to \infty}\|\kla_{(k)} B_{(k),\mu_k} u\|\le \lim_{k\to \infty}\frac{\mu_k}{k} = 0 .$$

For $u = z^\alpha\in H^0(\bbC^n , \phi_0)$ with $\alpha\in \bbN_0^n$ and a suitable cut-off function $\chi$, let $u_k= \chi(\frac{z}{\log k})u(z)$. Since we have $u_k \in L^2(B(0,\log k), \delta_k(k\hat{\phi}_k), d\hat{\vol}_{M,(k)})$, then $$
\begin{aligned}
\|(I-B_{(k),\mu_k})  u_k\|_{\delta_k(k\hat{\phi})}^2 &\le \frac{k}{\mu_k}\(\kla_{(k)} u_k | u_k\)_{\delta_k(k\phi)}\\
&= \frac{k}{\mu_k} \|\dbar u_k\|_{\delta_k(k\hat{\phi}_k)}^2 = O(k^{-\infty}) + \frac{k}{\mu_k\log^2 k }\|u_k\|_{\delta_k(k\hat{\phi}_k)}^2 = o(1).
\end{aligned}
$$
The $o(1)$ is from the assumption of $\mu_k$. Then we have $B_{(k)}= I + o(1)$ as $k$ tend to $\infty$. 

Thus $$B_{(k),\mu_k} u_k = u_k + o(1)$$, 
$$B_{(k),\mu_k}(x,y) = B_{\phi_0}(x,y)+o(1)$$
and 
$$
\begin{aligned}
P_{(k),\mu_k,s} (x,y) &= e^{-\delta_k(k\hat{\phi}_k)(x)}B_{(k),\mu_k}(x,y) e^{\delta_k(k\hat{\phi}_k)(y)} (x,y) \\
&= e^{-\delta_k(k\hat{\phi}_k)(x)} \(B_{\phi_0}(x,y) \)e^{\delta_k(k\hat{\phi}_k)(y)} +o(1)\\
&= e^{-\delta_k(k\phi)(x)} \(B_{\phi_0}(x,y) \)e^{\delta_k(k\phi)(y)} +o(1)
\end{aligned}
$$ on $B(0,\log k)$.

For $L$ positive over $p$, we have $\phi_0 = \sum_{j=1}^n \lambda_j |z_j|^2$w with $\lambda_j>0$, then 
$$B_{(k),\mu_k}(x,y) = B_{\phi_0}(x,y)+o(1) =  \(\frac{1}{\pi}\)^n \prod_{j=1}^n \lambda_j  e^{ 2\sum_j \lambda_j (x_j \ol{y}_j -|y_j|^2)} + o(1) .$$
and thus
$$
\begin{aligned}
P_{(k),\mu_k,s} (x,y) &= e^{-\delta_k(k\hat{\phi}_k)(x)}B_{(k),\mu_k}(x,y) e^{\delta_k(k\hat{\phi}_k)(y)} (x,y)\\  
&= e^{-\delta_k(k\hat{\phi}_k)(x)} \(\(\frac{1}{\pi}\)^n \prod_{j=1}^n \lambda_j  e^{ 2\sum_j \lambda_j (x_j \ol{y}_j -|y_j|^2)} \)e^{\delta_k(k\hat{\phi}_k)(y)} +o(1)\\
&= \det(\frac{\dot{R}_L}{2\pi})(p) e^{\sum_j \lambda_j (2x_j \ol{y}_j - |x_j|^2- |y_j|^2)} + o(1)
\end{aligned}
$$
\end{proof}

Recall the local spectral gap condition (see Definition \ref{Local Spectral Gap}), that is for a subset $U\subset M$ we have $$\|(I- P_k) u \|_k^2 \le C k^r \(\kla_k u |u \)_k \text{ for all } u \in \gO_c(U, L^k)$$ for some $C>0$ and $r\in \bbR$. For a fixed point $p\in U$, we take the Chern-Moser trivialization $(z,s)$ on $V\Subset U$ centered at $p$. For any $\hat{u}= z^\alpha\in H^0(\bbC^n, \phi_0)$, we take a suitable cut-off function $\chi$ such that $u_k= (\chi_k \hat{u})\otimes s^k\in \gO_c(V,L^k)$, here $(\chi_k\hat{u})(z) = \chi(\frac{\sqrt{k}z}{\log k})\hat{u}(\sqrt{k}z)$. Thus we have $$\|u_k - P_ku_k\|_k^2 = \|(I-P_k) u_k\|_k^2 \le Ck^r \(\kla_k u_k |u_k\)_k = Ck^r\|\dbar_k u_k\|_{k}^2 = O(k^{-\infty}).$$ Thus we know $P_k u_k = u_k + O(k^{-\infty})$, this will replace the role of Hörmander $L^2$-Estimate in the proof of Theorem \ref{MFD-Main-Thm}. Then we repeat the method of localized and scaling, we have 
$$
\begin{aligned}
P_{k,s}(x,y) &= e^{-k\phi(x)} k^nB_{\phi_0}(\sqrt{k}x, \sqrt{k}y) e^{k\phi(y)} + o(k^n) \\
&= k^n e^{-k\phi_0(x)} B_{\phi_0}(\sqrt{k}x,\sqrt{k}y) e^{k\phi_0(y)} + o(k^n)
\end{aligned}
$$
on $B(0,\frac{\log k}{\sqrt{k}})$. Thus, we conclude the following theorem.

\begin{thm}
    Let $M$ be a complex manifold with Hermitian line bundle $L$. Let $\kla_k$ satisfy the local spectral gap condition over $U\subset M$. For any $p\in U$, there is a local trivialization $(s, U)$ with Chern-Moser coordinate on $U$ center at $p$ and the Bergman kernel has the following expansion on the neighborhood $B(0, \frac{\log k }{\sqrt{k}})$
          $$P_{k,s} = k^n P_{\phi_0}(\sqrt{k}x, \sqrt{k}y)+o(k^n)$$ where $B_{\phi_0}: L^2(\bbC^n,\phi_0)\rightarrow H^0(\bbC^n, \phi_0)$ and $P_{\phi_0} = e^{-\phi_0} B_{\phi_0} e^{\phi_0}$. 
         Moreover, if $L$ is positive over $U$, we have 
         $$P_{k,s}(x,y) = k^n \det(\frac{\dot{R}_L}{2\pi})(p) e^{k \sum_j \lambda_j (2x_j \ol{y}_j - |x_j|^2- |y_j|^2)} +o(k^n), $$ and if $L$ is non-positive over $p$, $P_{\phi_0}= 0$. 
\end{thm}

\section{Asymptotic Behavior of Bergman Kernels on Complex Orbifold}

\label{Orbifold Case}
In this section, we prove the leading term of the Bergman kernel on an orbifold and use the result to give a pure analytic proof of the Kodaira-Baily Embedding Theorem. 

\subsection{Notation and Set-up}
\label{orbifold}
We introduce the background we need about orbifolds, most of the content in this section is from  \cite{ma2007holomorphic},  \cite{ma2012berezintoeplitz},  \cite{adem_leida_ruan_2007}, and  \cite{SasakianGeomertry}.

Let $X$ be a paracompact Hausdorff topological space. A \textbf{complex orbifold chart} on $X$ is a triple $(\ti{U}, G, \varphi)$, where $\ti{U} $ is a connected open subset of $\bbC^n$, $G:\ti{U}\rightarrow \ti{U}$ is a finite group action, and $\varphi: \ti{U} \rightarrow U $ is a continuous map onto a open subset $U\subset X$ such that $\varphi\circ g = \varphi$ for all $g\in G$ and induced a homeomorphism from $\ti{U}/G $ to $U$. An \textbf{injection} between two orbifold charts $(\ti{U} , G, \varphi) $ and $(\ti{V} , H, \psi)$ is a holomorphic embedding $\iota:\ti{U} \rightarrow \ti{V} $ such that $\psi\circ \iota = \varphi$. A \textbf{complex orbifold atlas} on $X$ is a family $\cU  = \{ (\ti{U}_j , G_j, \varphi_j)\}$ of complex orbifold charts such that \begin{enumerate}[(i)]
    \item $X= \bigcup_j \varphi(\ti{U}_j) = \bigcup_j U_j$,
    \item Given two charts $(\ti{U}_j , G_j, \varphi_j)$ and $(\ti{U}_\ell , G_\ell , \varphi_\ell)$ and a point $x\in \varphi_j(\ti{U}_j) \cap \varphi_\ell(\ti{U}_\ell)$, there exists a open neighborhood $U_k$ of $x$ and a complex orbifold charts $(\ti{U}_k, G_k , \varphi_k)$ such that there are two injections $\iota_{jk}:(\ti{U}_k, G_k , \varphi_k)\rightarrow (\ti{U}_j , G_j, \varphi_j)$ and $\iota_{\ell k} : (\ti{U}_k , G_k, \varphi_k)\rightarrow(\ti{U}_\ell , G_\ell, \varphi_\ell) $.
\end{enumerate}
If we can find a complex atlas on $X$, we call $X$ a complex orbifold. A general orbifold is defined by replacing open subsets $\ti{U}$ of $\bbR^N$ with quotients by finite group $G$ actions, and requiring that the transition maps between charts are smooth and compatible with the group actions. Note that sometimes we will denote the orbifold structure by $(\ti{U}, G)$.

For each $x\in X$, the isotropy group $G_x$ is defined as the stabilizer of $\ti{x}\in \ti{U}$ in a local chart $(\ti{U}_x, G_x)$ near $x$. Note that $G_x$ is well-defined up to isomorphism. For $x\in X$ whose isotropy subgroup $ G_x\ne \{ e\} $ is called a \textbf{singular} point. Those points with $G_x = \{e\} $ are called \textbf{regular} points. The set of singular points is called the \textbf{orbifold singular locus} or \textbf{orbifold singular set}, and denoted by $X_{sing}$

We say that $E$ is an \textbf{orbibundle} over an orbifold $X$ if $E$ is an orbifold and for $(\ti{U}, G_U )\in \cU$, $(\ti{E}_U, G^{E}_U, \ti{\pi}_U: \ti{E}_U\rightarrow \ti{U}) $ is a $G^{E}_U$-equivariant vector bundle and $( \ti{E}_U,G^E_U)$ is the orbifold structure of $E$ with $(\ti{U}, G_U)$ is the orbifold structure of $X$ where $G_U = G^E_U/ K^E_U$, $K^E_U$ is the kernel of the map $G^E_U $ to the diffeomorphism over $\ti{U}$. If $K^E_U$ is trivial, we call $E$ a proper orbifold vector bundle. For a section $s:U\rightarrow E$, we will denote $\ti{s}: \ti{U}\rightarrow \ti{E}_U$.
\begin{rem}
    The orbibundle we mention in the following will be assumed to be proper.
\end{rem}

Analogous to manifolds, we can define orbifold tangent bundle, cotangent bundle, and so on. And we will use same notation as in manifold, for an open subset $U$ of $X$, denote $\cC^\infty(U)= \gO(U)$ the smooth function on $U$ and $\cC^\infty_c(U) = \gO_c(U) $ be the subspace of $\cC^\infty(U) $ whose elements have compact support in $U$. Let $H^0(U)$ denote the space of holomorphic functions on $U$. Let $E\rightarrow X$ be a complex orbifold vector bundle, we denote $\cC^\infty(U,E)= \gO(U,E)$ the smooth section over $U$ and $\cC^\infty_c(U,E) = \gO_c(U,E) $ be the subspace of $\cC^\infty(U,E) $ whose elements have compact support in $U$ for any open subset $U\subset X$. For $E$ a holomorphic vector bundle, we let $H^0(U, E)$ be the space of holomorphic sections of $E$ over $U$. We observe that the complex structure, eigenspace decomposition of complexified orbifold tangent and cotangent bundle, and local weight can be defined on orbifolds in the same way. Note that all the functions, sections, or local weights on orbifolds will be $G_U$ invariant over an orbifold chart $(\ti{U}, G_U)$.

We define the integral of $\eta\in \cC^\infty(X, T^*X)$ with $\supp(\eta)\subset U$ as 
\begin{equation}
\label{integral on orbifold}
\int_X \eta = \frac{1}{|G_U|}\int_{\ti{U}} \ti{\eta}_U.
\end{equation}
Similar to manifold, when we fix a Hermitian metric $\go$ on $X$, we can define the corresponding volume form $d\vol_X = \frac{\go^n}{n!}$ and the inner product on $(0,q)$-form. Using this, we can construct the inner product, the $L^2$-norm, and the $L^2$-space for functions, $(0,q)$-forms, or sections.


Let us discuss about the kernel on orbifold. Let $X$ be a complex orbifold and $E$ be a proper orbibundle over $X$. To simplify the notation, for any open chart $(\ti{U},G)$ of $X$, we will add $\sim$ to the corresponding object on $\ti{U}$. Assume that $\ti{K}(\ti{x}, \ti{y})\in \cC^\infty (\ti{U}\times \ti{U} , \ti{E}\boxtimes \ti{E})$ satisfy $$(g,1)\ti{K}(g\invs \ti{x}, \ti{y} ) = (1,g\invs) \ti{K}(\ti{x}, g\ti{y} )$$ where $(g_1, g_2)$ acts on $\ti{E}_{\ti{x}}\times \ti{E}_{\ti{y}} $ by $(g_1, g_2) (\xi_1, \xi_2) = (g_1\xi_1, g_2\xi_2)$. We can define an operator $\ti{K} : \cC^\infty_c(\ti{U},\ti{E}) \rightarrow \cC^\infty(\ti{U}, \ti{E}) $ by $$(\ti{K}\ti{s}) (\ti{x}) = \int_{\ti{U} } \ti{K}(\ti{x}, \ti{y}) \ti{s}(\ti{y} ) d\vol_{\ti{U}}(\ti{y}) \text{ for } \ti{s} \in \cC^\infty_c(\ti{U} , \ti{E}).$$ Thus we can define an operator $K: \cC^\infty_c(U , E)\rightarrow \cC^\infty(U, E) $ by $$(Ks)(x) = \frac{1}{|G|} \int_{\ti{U}} \ti{K}(\ti{x},\ti{y}) \ti{s}(\ti{y}) d\vol_{\ti{U} }(\ti{y}) \text{ for } s \in \cC^\infty_c(U , E). $$ Then the smooth kernel of the operator $K$ with respect to $d\vol_X$ is given by $$K(x, y) = \sum_{g\in G} (g,1) \ti{K}(g\invs \ti{x}, \ti{y}) .$$

In the following, we let $X$ be a complex orbifold with dimension $n$, and $L\rightarrow X$ be a holomorphic orbifold line bundle with Hermitian metric $h^L$.

\begin{rem}
    For any holomorphic orbifold line bundle and any point $p\in X$, there exists an integer $N_p>0$ such that $L^{N_p}$ is trivially acted near $p$. This means there exists a complex orbifold chart $(\ti{U},G)$ around $p$ such that $$\wt{L^{N_p}|_U} = \ti{U}\times \bbC\text{, with }g\cdot (\ti{x},v)= (g\cdot \ti{x},v) \text{ for all } g\in G,\, (\ti{x},v)\in \ti{U}\times \bbC.$$ Moreover, there exists a global integer $N>0$ such that $L^N$ is trivially acted on globally if $X$ is compact. Since we will work locally near the point $p$ in the following, we will assume without loss of generality that $L$ is already trivially acted.
\end{rem}

Let $$P_k: L^2(X,L^k)\rightarrow \ker \kla_k$$ be Bergman projection and $P_k(x,y)$ be Bergman kernel. Also, $$P_{k,\mu_k}: L^2(X,L^k)\rightarrow E_{\mu_k}(X,L^k)$$ be spectral projection with $P_{k,\mu_k}(x,y) $ the  spectral kernel. Let $s$ be a local holomorphic trivialization of $L$ on an open subset of $X$ and $\phi$ be the local weight of $h^L$ with respect to $s$. We will denote $\dbar_k$ the Cauchy-Riemann operator and $\dbar_k^*$ its formal adjoint with respect to the inner product for $L^2(X, L^k)$.

Notice that in the case of smooth manifolds, we can define the curvature forms and the curvature operators, and the following lemma holds.

\begin{lem}
    Let $X$ be a complex orbifold, and let $(L,h^L)$ be a holomorphic orbifold line bundle on $X$. Fix a point $p\in X$, we can choose a complex orbifold chart $(\ti{z}_1, \ldots, \ti{z}_n,G) $ on an open neighborhood $U\subset M$ of $p$  with $G$ is the isotropy group at $p$ and a holomorphic trivializing section $s\in H^0(U, L) $ such that $\ti{z}_j(p)=0 $, $\<\frac{\dr}{\dr \ti{z}_j}|\frac{\dr }{\dr \ti{z}_\ell}\>  = \delta_{j,\ell}+O(|z|)$ for $j,\ell=1,\ldots, n$ and $|s(z)|^2_{h^L} = e^{-2\phi(z)} $  with $G$-invariant local weight $\phi(\ti{z}) = \sum_{j=1}^n \lambda_{j,p} |\ti{z}_j|^2+ O(|\ti{z}|^3)$, where $2\lambda_{j,p} $ are eigenvalues of curvature operator at $p$. We usually denote $\phi_0(\ti{z}) = \sum_{j=1}^n \lambda_{j, p} |\ti{z}_j|^2$.
\end{lem}
We will still call $(\ti{z}, G,s)$ a Chern-Moser trivialization on $U$ centered at $p$.

Recall the local spectral gap condition \ref{Ck Local Spectral Gap}, we will assume our Kodaira Laplacian satisfies this condition in the rest of this section. Note that when $X$ is compact and $L$ is positive, the condition will always be satisfied.

We will study the asymptotic behavior of the Bergman kernel near $p$ with $L$ positive around the point $p$.

\begin{thm}
\label{Orbifold-Main-Thm}
    For any $p\in U\subset X$ with $L$ positive over $U$ and $\kla_k$ satisfying local spectral gap, then there is a Chern-Moser trivialization $(z,s)$ on $U$ centered at $p$ such that $$P_k(x,y) = k^n \sum_{g\in G} P_{\phi_0} (\sqrt{k}\ti{x}, g\cdot \sqrt{k}\ti{y}) +\delta_k(x,y)  $$ on $B(0,\frac{\log k}{\sqrt{k}})$ where $\phi(z) = \phi_0(z)+ O(|z|^3)$ is the local wight with respect to this trivialization, $G$ is the isotropy group at $p$, and $\pi(\ti{x}) = x$ is the natural projection. The remaining term $\delta_k$ satisfying $$\left\| \delta_k\(\frac{\ti{x}}{\sqrt{k}}, \frac{ \ti{y}}{\sqrt{k}}\)\right\|_{\cC^\ell(K)}= o(k^n)$$ for all $\ell\in \bbN_0$ and any compact subset $K\subset B(0,\log k)$.
\end{thm}

\subsection{Proof of Theorem \ref{Orbifold-Main-Thm}}
 We need the following off-diagonal property of the Bergman kernel and the Spectral kernel.

\begin{Def}
Let $A_k: L^2(X,L^k) \rightarrow L^2(X,L^k)$ be a family of continuous linear operators. We say that $A_k$ is $k$-negligible, denote by $A_k= O(k^{-\infty}) $, if for two local holomorphic trivializations $(s,U)$ and $(t,V)$ and for any compact subset $K\subset U\times V$, $\alpha,\beta\in \bbN^{2n} $ and $N\in \bbN$, the kernels with trivialization satisfying $$| \dr_x^\alpha \dr_y^\beta A_{k,s,t}(x,y) |\le C_{\alpha, \beta ,K ,N} k^{-N}.$$
\end{Def}
From \cite[Proof of Theorem 3.8]{Hsiao-Marinescu17}, we have the following: 

\begin{pro}[Off-diagonal Property]
Let $Y$ be a complex orbifold and $L$ positive over $Y$. Assume $\mu_k = o(k)$. Let $\tau_1, \tau_2 \in \mathcal{C}^\infty_c(Y)$ with $\supp \tau_1 \cap \supp \tau_2 = \emptyset$. Then
$$ \tau_1 P_{k,\mu_k,\ell} \tau_2 = O(k^{-\infty}). $$
\end{pro}

\begin{pro}
Let $Y$ be a complex orbifold and $L$  positive over $Y$. Assume $\mu_k = o(k)$. Let $\tau, \chi \in \mathcal{C}^\infty_c(Y)$ with $\tau = 1$ on $\supp \chi$. Define $$ \tau_k(z) = \tau\left( \frac{\sqrt{k} z}{\log k} \right), \quad \chi_k(z) = \chi\left( \frac{\sqrt{k} z}{\log k} \right). $$
Then $$ (1 - \tau_k) P_{k,\mu_k} \chi_k = O(k^{-\infty}). $$
\end{pro}

To asymptotic the Bergman kernel, we focus on the open neighborhood $U$ with Chern-Moser trivialization $(\ti{z}, G,s)$ on $U$ centered at $p$. We denote $P_{k,\mu_k}$ the spectral projection on $\ti{U}$ associate to the Kodaira Laplacian $\kla_{k,s}$ with respect to $k\phi$ and its distribution kernel, the spectral kernel, $P_{k,\mu_k}(\ti{x}, \ti{y} )$ on $\ti{U}$. Let $$\hat{P}_{k,\mu_k}(x,y) = \sum_{g\in G} P_{k,\mu_k}(\ti{x}, g\cdot \ti{y})$$ be a kernel on $U $ and let $\hat{P}_{k,\mu_k}$ be the induced operator.

Let us recall the well-known global result.

\begin{fact}
There exists a bounded operator $N_k:L^2(X,L^k) \rightarrow \Dom \kla_k $ such that 
    \begin{equation}
    \label{orbifold global equation 1}
    \begin{cases}   
        N_k \kla_k + P_k =I  \text{ on } \Dom \kla_k\\
        \kla_k N_k + P_k = I \text{ on } L^2(X,L^k).
    \end{cases}
    \end{equation}
    Also, there exists bounded operators $N_{k,\mu_k}: L^2(\ti{U},d\vol_X)\rightarrow \Dom \kla_{k}$
    \begin{equation}
    \begin{cases}
        N_{k,\mu_k} \kla_{k,s} + P_{k,\mu_k} = I \text{ on } \Dom \kla_{k,s}\\
        \kla_{k,s} N_{k,\mu_k} + P_{k,\mu_k} = I \text{ on }L^2(\ti{U}, d\ti{\vol}_X)
    \end{cases}.
    \end{equation}
    For $\hat{P}_{k,\mu_k}$, we have $\hat{N}_{k,\mu_k}$ such that
     \begin{equation}
    \label{orbifold global equation 2}
    \begin{cases}
        \hat{N}_{k,\mu_k} \kla_k + \hat{P}_{k,\mu_k} = I\text{ on } \Dom \kla_k\\
        \kla_k \hat{N}_{k,\mu_k} + \hat{P}_{k,\mu_k} = I \text{ on }L^2(U,L^k)
    \end{cases}.
    \end{equation}
\end{fact}

Let $\chi,\tau,\hat{\chi}\in \cC^\infty_c (U)$ with $\hat{\chi}=1$ on $\supp \chi$ and $\tau=1$ on $\supp \hat{\chi}$.

\begin{pro}
For $\mu_k = \frac{k}{\log^{1-\varepsilon}k}$ with $\varepsilon>0$ in the Definition \ref{Ck Local Spectral Gap}, we have, for all $\ell\in \bbN_0$,
    $$\begin{aligned}
        \|\hat{\chi}P_k\chi - \hat{\chi}\hat{P}_{k,\mu_k} \chi\|_\ell = o(k^{\ell}).
    \end{aligned}$$
\end{pro}
\begin{proof}
First, from \eqref{orbifold global equation 2}, we have $$\chi\hat{N}_{k,\mu_k} \kla_k + \chi\hat{P}_{k,\mu_k}  = \chi.$$
Composing $\tau P_k \hat{\chi}$ from the right, we obtain $$\chi\hat{N}_{k,\mu_k} \kla_k  \tau P_k \hat{\chi}+ \chi\hat{P}_{k,\mu_k} \tau P_k \hat{\chi} = \chi\tau P_k \hat{\chi} = \chi P_k \hat{\chi}.$$
By the off-diagonal property and the known properties of the Bergman kernel, it follows that $$\hat{\chi}P_k \chi - \hat{\chi} P_k \tau \hat{P}_{k,\mu_k} \chi =O(k^{-\infty}).$$
Next, we use the $C_k$ local spectral gap condition, 
$$
\begin{aligned}
\|\hat{\chi}P_k\tau\hat{P}_{k,\mu_k}\chi u -\hat{\chi}\tau\hat{P}_{k,\mu_k}\chi  u\|&= \|\hat{\chi}\(I-P_k\)\tau\hat{P}_{k,\mu_k} \chi u\|\\
&\le C_k\(\kla_k \tau \hat{P}_{k,\mu_k} \chi u | \tau \hat{P}_{k,\mu_k} \chi u\)_k\\
&\lesssim C_k \mu_k \|u\|^2+ O(k^{-\infty}).
\end{aligned}
$$
Note that $\lim_{k\to \infty} C_k\mu_k= 0$ and $\mu_k$ ensure that  $P_{k,\mu_k}$ have an asymptotic expansion as in Theorem \ref{spectral kernel theorem}. 

Finally, for $2\ell$-norm, by Gårding inequality and noting that $\mu_k= o(k)$, we have $$\|\hat{\chi}P_k\chi u - \tau \hat{P}_{k,\mu_k}\chi u\|_{2\ell}\lesssim \(C_k\mu_k+ C_k \mu^{2\ell+1}_k\)\|u\|_{-2\ell} = o(k^{2\ell}).$$
\end{proof}

The order of the right-hand side allows us to apply the similar scaling method and obtain the following asymptotic behavior of the Bergman kernel, 
$$\begin{aligned}
    \lim_{k\to \infty } P_k\(x, y\) = \lim_{ k\to \infty}  \hat{P}_{k,\mu_k}\(x, y\) = \sum_{g\in G} k^n P_{\phi_0} (\sqrt{k}\ti{x}, \sqrt{k}(g\cdot \ti{y})) +\delta_k(x,y),
\end{aligned}$$
where $\|\delta_k\(\frac{\ti{x}}{\sqrt{k}}, \frac{ \ti{y}}{\sqrt{k}}\)\|_{\cC^\ell(K)} = o(k^n) $ for all $\ell\in \bbN_0$ and any compact subset $K\subset B(0,\log k)$.

\subsection{Kodaira Baily Embedding Theorem}
As an application of the asymptotic of the Bergman kernel, we will give a purely analytic proof of the Kodaira-Baily Embedding Theorem \cite{Baily1957} through the Bergman kernel with an extra condition.

Let $X$ be a compact complex orbifold and $(L,h^L)$ be a positive orbifold line bundle over $X$. The Kodaira-Baily $\Phi_k: H^0(X,L^k)\rightarrow \bbC\bbP^{d_k-1}$ where $d_k$ is the dimension of $H^0(X,L^k)$ is given by $x\mapsto [s_1(x): \cdots: s_{d_k}(x)] $ where $\{s_j\}_{j=1}^{d_k} $ is an orthonormal basis.

From the asymptotic of Bergman kernel, we have that on the diagonal of Bergman kernel $P_k(p,p)$ will greater than $k^n $ for $k\ge k_p $ thus there always exists a small neighborhood $U_p$ of $p$ such that there is $u_p\in H^0(X,L^{k_0})$ with $|u_p(x)|\ge 1$ if $x\in U_p$. From the compactness of $X$, we can have a finite open cover by $\{U_{p_\ell}\}_{\ell=1}^N$ and take $M$ be the least common multiple of $k_{p_\ell} $. 
We will denote $L$ as $L^M$ in the following. Thus there is $u_{p_\ell} \in H^0(X,L)$ such that $|u_{p_\ell}(x) | \ge 1 $ if $x\in U_{p_\ell}$, which implies that the Kodaira-Baily map is well-defined.

Suppose that the Kodaira-Baily map is not injective, that is there is $k_j$ such that there are $\{x_{k_j}\}$ and $\{y_{k_j}\}$ with $x_{k_j} \ne y_{k_j}$ $\Phi_{k_j}(x_{k_j}) = \Phi_{k_j}(y_{k_j}) $. We will simplify write $k_j$ to $k$, and let $g(x_k) = \lambda_k g(y_k) $ with $|\lambda_k|\ge 1$ for all $g\in H^0(M,L^k) $. From the compactness, we assume $x_k\to p$ and $y_k \to q$ as $k\to \infty$. Let $\chi\in\cC^\infty_c(\bbR)$ be a cut-off function, and defined $$\chi_{G,k}(\ti{z}) = \sum_{g\in G} \chi(\sqrt{k}|g\cdot \ti{z}|).$$ Let $(\ti{z},G,s)$ be a Chern-Moser trivializations  centered at $p$. Consider $u_k = P_k\(e^{k\phi}\chi_{G,k}(\ti{z})\otimes s\)$. By the Hörmander $L^2$-estimate, we have $u_k(p)\ne 0$ and $u_k(q) = O(k^{-\infty}) $.  This leads to a contradiction if $p\ne q$, so we conclude $p=q$.
    \begin{enumerate}[\bf \text{Case} 1.]
        \item $\limsup_{k\to \infty} \sqrt{k} \min_{g,h\in G} |g\cdot\ti{x}_k-h\cdot\ti{y}_k|= M >0 $\\
            We may assume that $ \lim_{k\to \infty} \sqrt{k}|x_k-y_k| = M $. Let $v_k(z)  = P_k(z, y_k)$, we may have $$|v_k(x_k)| = |P_k(x_k, y_k)|< | P_k(y_k , y_k)| = |v_k(y_k)|$$
            from the explicit expansion of $P_k$. This induced a contradiction. 
        \item $\limsup_{k\to \infty} \min_{g,h\in G} |g\cdot\ti{x}_k-h\cdot\ti{y}_k|= 0 $\\
            Let $f_k(t) = \frac{|P_k(tx_k +(1-t) y_k, y_k)|^2}{P_k(tx_k+(1-t)y_k) P_k(y_k)}$. By Cauchy inequality, we have $ 0 \le f_k(t) \le 1$ and $f_k(0) = f_k(1) = 1$. There is a contradiction that $f''(t)< 0 $ by the direct computation and there exists $\{t_k\} $ such that $\liminf \frac{f''(t_k)}{k|x_k-y_k|^2} \ge 0  $.
    \end{enumerate}
    Thus, the Kodaira-Baily map is globally injective.

We will prove the following version of the Kodaira-Baily theorem, 
\begin{thm}
\label{Kodaira-Baily Theorem}
    The Kodaira-Baily map $\Phi_k$ is injective and immersion near every regular point. Near singular point, let $p$ be a singular point and let $(\ti{U}, G_U)$ be an orbifold chart defined near $p$. Then, there is an immersion $\ti{\Phi}_k$ defined near $\ti{V}$ such $\Phi_k(x)=\sum_{g\in G_U}\ti{\Phi}_k(g\cdot \ti{x})$ where $\pi(\ti{p})=p, \pi:\ti{U}\to U$ is the natural projection.
\end{thm}
Because near the regular point, we know that the isotropy group is trivial. Thus, the asymptotic of the Bergman kernel is just the same as the manifold case; we refer to \cite{hsiao2014bergman} for the regular point case.

For the singular point $p$ and the orbifold chart $(\ti{U}, G_U)$ with natural projection $\pi: \ti{U} \rightarrow U \simeq \ti{U}/G_U$. Recall that we define $\chi_{G,k}(\ti{z})= \sum_{g\in G}  \chi(\sqrt{k}|g\cdot \ti{z}|) $ and  $\chi_k(\ti{z}) =\chi(\sqrt{k}|\ti{z}|).$
For $j=1,\ldots, n$, we denote $\ti{f}_{j,k}(\ti{z}) = \sqrt{k}\ti{z}_j$ and $\hat{f}_{j,k}(\ti{z})= \sum_{g\in G}\sqrt{k}(g\cdot \ti{z}_j)$. Let $$v_k^{(j)} = P_k\(e^{k\phi}\hat{f}_{j,k} \chi_{G,k}\otimes s^k\)$$ and $$\ti{v}_k^{(j)} = P_{k,\mu_k}(\ti{f}_{j,k}\chi_k).$$  We know that $P_k(x,y)= \sum_{g\in G_U} P_{k,\mu_k}(\ti{x}, g \cdot \ti{y}) + \delta_k(x,y) $. Hence,
$$
\begin{aligned}
v_{k}^{(j)}(x) &= \sum_{g,h\in G} \frac{1}{|G|} \int  P_{k,\mu_k} (\ti{x}, g\cdot \ti{y}) \ti{f}_{j,k}(h\cdot\ti{y}) \chi_{k}(h\cdot \ti{y}) d \ti{y}+ \varepsilon_k(x)\\
&= \sum_{g\in G} \int  P_{k,\mu_k} (\sqrt{k}\ti{x}, \sqrt{k} \ti{y}) \ti{f}_{j,k}(g\cdot \ti{y})\chi_k(g\cdot \ti{y})d \ti{y}+ \varepsilon_k(x)\\
&= \sum_{g\in G} \ti{v}_k^{(j)}(g\cdot\ti{z})+ \varepsilon_k(x).
\end{aligned}
$$
Let $u_k = P_k(e^{k\phi}\chi_{G,k}\otimes s^k)\in H^0(X,L^k)$ be a global holomorphic section such that $u_k(p)\ne 0 $. We consider two maps
$$
\begin{aligned}
\hat{\Phi}_k: X &\rightarrow \bbC\bbP^{d_k-1}\\
x&\mapsto [u_{k}(x):v_{k}^{(1)}(x):\cdots: v_{k}^{(n)}(x) : g^{(k)}_{n+2}(x) : \cdots:g^{(k)}_{d_k}(x)]
\end{aligned}
$$
and 
$$
\begin{aligned}
\ti{\Phi}_k: \ti{U} &\rightarrow \bbC\bbP^{d_k-1}\\
\ti{x} &\mapsto [u_{k}(\ti{x} ):\ti{v}_{k}^{(1)}(\ti{x}):\cdots: \ti{v}_{k}^{(n)}(\ti{x}) : \ti{g}^{(k)}_{n+2}(\ti{x}) : \cdots:\ti{g}^{(k)}_{d_k}(\ti{x})].
\end{aligned}
$$
Since $$\ti{\Psi}_k(\ti{x}) = \(\frac{\ti{v}_k^{(1)}}{u_k}(\ti{x}),\ldots, \frac{\ti{v}_k^{(n)}}{u_k}(\ti{x})\)$$ is an immersion near the origin by apply the same method as for a regular point, it follows that $\ti{\Phi}_k$ is also an immersion. 
Therefore, locally we have $$\hat{\Phi}_k(x)= \sum_{g\in G_U} \ti{\Phi}_k(g\cdot \ti{x})  $$ and $\ti{\Phi}_k$ is an immersion at $p= 0$. When we switch back to an orthonormal basis of $H^0(X, L^k)$, there exists a biholomorphic map $F: \mathbb{CP}^{d_k - 1} \rightarrow \mathbb{CP}^{d_k - 1}$ such that $$\Phi_k(x)= F(\hat{\Phi}_k(x)) = F\(\sum_{g\in G_U} \ti{\Phi}_k(g\cdot \ti{x})\) .$$

\begin{rem}
\label{Extra Condition}
    As we assume that $$\sum_{g\in G_p}d\hat{g} $$ has full rank where $(\ti{U}_p, G_p)$ is orbifold chart near $p$ and $\hat{g}:\ti{U} \rightarrow\ti{U}$ denote the action of $g$. The immersion can be proof by take $\hat{f}_k^{(j)}(z) = \sum_{g\in G} (g\cdot \ti{z})_j$ and consider $u_{k}^{(j)}= P_k(e^{k\phi} s^k f_k^{(j)})u_{p} $ where $u_p$ is as above. Thus, the embedding is clearer in this case. 
\end{rem}

\subsection{Toeplitz Operator}
Recall that $X$ is not assumed to be compact, but with $C_k$ local spectral gap condition and $R^L>0$ on $U$ near a point $p\in X$. Let $f\in \cC^\infty(X)$ and $$T_{f,k} = P_k \circ M_f \circ P_k:L^2(X,L^k)\rightarrow L^2(X,L^k)$$ be the Toeplitz operator where $M_f$ is multiple operator. 
We want to study the large $k$-behavior of $T_{f,k}$ near $p$. 

\begin{thm}
\label{toeplitz theorem}
For $k$ large enough, we have the following statement 
\begin{enumerate}[(i)]
    \item For $\chi,\tau \in \cC^\infty_c(X)$ with $\supp \chi \cap \supp \tau = \emptyset$, $$\chi T_{f,k}\tau = O(k^{-\infty}),$$
    \item $$T_{f,k} (x,y)=k^n \sum_{g\in G_p} P_{\phi_0}(\sqrt{k}\ti{x} , \sqrt{k}(g\cdot\ti{y})) f(p) + \varepsilon_k(x,y)  $$ for $(x,y)$ near $(p,p)$, where $\varepsilon_k(x,y)$ satisfying $$\left\|\varepsilon_k\(\frac{\ti{x}}{\sqrt{k}}, \frac{\ti{y}}{\sqrt{k}}\)\right\|_{\cC^\ell(K)}=o(k^n)$$ for all $ \ell\in \bbN_0$ and any compact subset $K\subset B(0,\log k)$,
    \item $\lim_{k\to \infty} \|T_{f,k}\| = \| f\|_{L^\infty}$.
\end{enumerate}
\end{thm}

Before proving this theorem, we recall the stationary phase formula(see Subsection \ref{Stationary Phase Formula}).

\begin{proof}[Proof of Theorem \ref{toeplitz theorem}]
For $x\ne y$, take two cut-off functions with disjoint support. Then the off-diagonal property of Bergman projection will give us the $O(k^{-\infty}) $. 

For the diagonal, we use the asymptotic of the Bergman kernel to estimate $P_k$, we know the convergence is in $\cC^\infty$-topology. Fix $p\in X$,

    $$\begin{aligned}
        \frac{1}{k^n}T_{f,k}(x,y) &= \frac{1}{k^n}\int  P_k(x,z) f(z) P_{k}(z,y)d\vol_X \\
        &= \sum_{g, h\in G_p} \frac{1}{|G_p|}\int  P_{\phi_0}((\sqrt{k}\ti{x}, g\cdot\ti{z}) f(\frac{\ti{z}}{\sqrt{k}}) P_{\phi_0}(\sqrt{k}\ti{z}, \sqrt{k}(h\cdot\ti{y}))  d\ti{z} \\
        &+\sum_{g, h\in G_p} \frac{1}{|G_p|}\int  P_{\phi_0}(\sqrt{k}\ti{x},\sqrt{k}(g\cdot \ti{z})) f(\ti{z}) \delta_k(\ti{z}, h\cdot\ti{y})  d\ti{z}\\
       &+\sum_{g, h\in G_p} \frac{1}{|G_p|}\int  \delta_k(\ti{x},g\cdot \ti{z}) f(\ti{z}) P_{\phi_0}(\ti{z}, \sqrt{k}(h\cdot\ti{y}))  d\ti{z}\\
       &+\frac{1}{k^n}\int \delta_k(x,z)f(z)\delta_k(z,y) d\vol_X+O(k^{-\infty})\\
       &= \sum_{g\in G_p} \int  P_{\phi_0}((\sqrt{k}\ti{x},\ti{z}) f(\frac{\ti{z}}{\sqrt{k}}) P_{\phi_0}(\ti{z}, \sqrt{k}(g\cdot\ti{y}))  d\ti{z}+\varepsilon_k'(x,y)\\
       &= \sum_{g\in G_p} f(p)\int P_{\phi_0} (\sqrt{k}\ti{x}, \ti{z}) P_{\phi_0}(\ti{z}, \sqrt{k}(g\cdot \ti{y}))d\ti{z}\\
       &+\sum_{g\in G_p} \int  P_{\phi_0}((\sqrt{k}\ti{x},\ti{z}) \(f(\frac{\ti{z}}{\sqrt{k}})-f(p)\) P_{\phi_0}(\ti{z}, \sqrt{k}(g\cdot\ti{y}))  d\ti{z}+ \varepsilon'_k(x,y)\\
       &= \sum_{g\in G_p}f(p) P_{\phi_0}(\sqrt{k}\ti{x}, \sqrt{k}(g\cdot \ti{y}))+ \varepsilon_k(x,y)
    \end{aligned}$$
    where $G_p$ is the isotropy group of $p$ and $\pi(\ti{x}) =x $ is the natural projection. We use the stationary phase formula to derive that the $\cC^\ell$-norm tends to zero.

    For the third statement, we have $\|T_{f,k} \| \le \|f\|_{L^\infty}$ because the operator norm of Bergman projection is $1$ and $M_f$ is less than the sup-norm of $f$. 
    For the reverse inequality, we will find some $\chi_{k} \in L^2(X, L^k) $ with $\|\chi_k\|_k = 1 $ such that $\| T_{f,k} \chi_k \| _k \to \|f\|_{L^\infty}$ as $k\to \infty$. 
    
    We take $\chi_k(\ti{z})= k^{\frac{n}{2}} \chi(\sqrt{k}\ti{z})$ where $\chi(\ti{z}) = \begin{cases}
        1&\text{on }B(0,r)\\
        0&\text{outside } B(0,2r)
    \end{cases}$ and $\|\chi\|_{L^2} = 1$. Note that we choose the point $f$ tend its supremum to be $0$.
    By computation, we have $(T_{f,k} \chi_k)(\ti{x}) = k^{\frac{n}{2}}(T_{f_k, \phi_0}\chi)(\sqrt{k}\ti{x})+o(1) $ where $f_k( \ti{x}) = f(\frac{\ti{x} }{\sqrt{k}}) $ and $T_{f, \phi_0} = P_{\phi_0} \circ M_f\circ P_{\phi_0}$. We also have $\|T_{f,k}\chi_k\|_{L^2} = \| T_{f_k, \phi_0} \chi\|_{L^2} +o(1)$. Thus, when $k$ is large enough and tends to $\infty$, the $L^2$-norm will be monotone increasing since $P_{\phi_0} \chi $ has most mass near $0$ and $f_k$ are outspread. Then when $k$ tend to $\infty$, we will get the $L^2$-norm converge to $\|f\|_{L^\infty}$.
\end{proof}

\section{Full Expansion}
\label{Full Expansion}

In this section, we will introduce a new method to obtain the full expansion of the Bergman kernel asymptotic and establish the Berezin-Toeplitz for pseudodifferential operators.  

\subsection{Bergman Kernel}
\label{bergman full expansion}

As in the leading term, we first focus on the Euclidean cases. To get the full expansion, we compare the Bergman kernel and the leading term in the previous sections. 
\subsubsection{Euclidean Cases}
Let $\phi\in \cC^\infty(\bbC^n,\bbR)$, $\phi(z)= \phi_0(z) + \phi_1(z) = \sum_{j=1}^n \gl_j |z_j|^2 + O(|z|^3) $, $\lambda_j>0$, $\phi_1=O(|z|^3)$, and $\tau\in \cC^\infty_c(\bbC^n,[0,1])$ with $\tau = 1$ on $B(0,1)$ and vanish outside $B(0,2)$. Let $\hat{\phi}_k = \phi_0+ \tau(\frac{\sqrt{k}z }{\log k}) \phi_1 $, $\hat{\go}_k = \sum_{j,\ell=1}^n \(\delta_{j,\ell} + \tau(\frac{\sqrt{k}z}{\log k})(w_{j,\ell}- \delta_{j,\ell})\) dz_j\wedge d\ol{z}_\ell$ and $d\hat{\vol}_k = (1+ \tau(\frac{\sqrt{k} z}{\log k }) O(|z|)) d\lambda(z)$ with $d\lambda(z)$ is the standard volume form. Recall (cf. Subsection \ref{Euclidean Case}) that $B_{k\hat{\phi}_k}$ is the Bergman projection $$B_{k\hat{\phi}_k}: L^2(\bbC^n, k\hat{\phi}_k ,d\hat{\vol}_k) \rightarrow  H^0(\bbC^n, k\hat{\phi}_k, d\hat{\vol}_k).$$ We also have a standard Bergman projection $$B_{k\phi_0}: L^2(\bbC^n ,k\phi_0, d\lambda) \rightarrow H^0(\bbC^n, k\phi_0, d\lambda).$$
The two Bergman kernels satisfying $$(B_{k\hat{\phi}_k}u)(x) = \int_{\bbC^n} B_{k\hat{\phi}_k} (x,y) u(y) d\vol_k(y)$$ and $$(B_{k\phi_0}u)(x) = \int_{\bbC^n} B_{k\phi_0} (x,y) u(y) d\lambda(y)$$ respectively. We will denote $B_{k\phi_0}^*$ be the adjoint operator with respect to $\(\cdot|\cdot\)_{k\hat{\phi}_k} $. Notice we have the following Facts:
\begin{fact}
    Since the difference of two weight functions $k\phi_0$ and $k\hat{\phi}_k$ and two volume forms have compact support, we have the equivalent of two $L^2$ spaces, that is $$L^2(\bbC^n,k\hat{\phi}_k,d\hat{\vol}_k) = L^2(\bbC^n,k\phi_0, d\lambda)$$ and $$H^0(\bbC^n,k\hat{\phi}_k,d\hat{\vol}_k) = H^0(\bbC^n,k\phi_0, d\lambda).$$
    Moreover, there is an operator $A_k$ where $A_k u= e^{2k(\phi_0- \hat{\phi}_k)}\vol_k u$ such that $$\(u|v\)_{k\hat{\phi}_k} = \(u|A_k v\)_{k\phi_0}$$ and we have the equvialent of $L^2$ norms $$\|\cdot\|_{k\hat{\phi}_k}^2 = \(1+O(\frac{\log^3 k}{\sqrt{k}})\) \|\cdot\|_{k\phi_0}.$$
\end{fact}

\begin{fact}
    $$B_{k\phi_0} \circ B_{k\hat{\phi}_k} = B_{k\hat{\phi}_k}$$
    and
    $$B_{k\hat{\phi}_k} \circ B_{k\phi_0} = B_{k\phi_0}.$$
\end{fact}

Since $B_{k\hat{\phi}_k} $ is self-adjoint with respect to $\(\cdot|\cdot\)_{k\hat{\phi}_k}$. Thus we have 
\begin{equation}
\label{full eqn1}
B_{k\hat{\phi}_k}= B_{k\hat{\phi}_k}^* = (B_{k\phi_0}B_{k\hat{\phi}_k})^* = B_{k\hat{\phi}_k} B_{k\phi_0}^* = B_{k\phi_0} + B_{k\hat{\phi}_k} (B_{k\phi_0}^*- B_{k\phi_0})
\end{equation}
where $^*$ is the adjoint with respect to $\(\cdot| \cdot\)_{k\hat{\phi}_k}$ and we will denote $r_k = B_{k\phi_0}^* - B_{k\phi_0}$. Since 
$$
\begin{aligned}
r_k =B_{k\phi_0}^*- B_{k\phi_0}&= A_k^{-1} B_{k\phi_0}A_k - B_{k\phi_0}\\
&=(A_k^{-1}-I)B_{k\phi_0} A_k + B_{k\phi_0} (A_k-I),
\end{aligned}
$$
the operator norm of $r_k$ is $O(\frac{\log^3 k }{\sqrt{k}})$ and we can write $$(1-r_k)^{-1} = 1+ r_k + r_k^2 + \cdots = \sum_{\ell=0}^\infty r_k^\ell,$$ where the series coverges in $L^2$-sense. By direct computation, we know that $$r_k(x,y) =  B_{k\phi_0}(x,y) (e^{2k(\psi_k(x) - \psi_k(y))}\vol_{k}^{-1}(x)\vol_k(y)-1)$$ where $\psi_k = \hat{\phi}_k - \phi_0$ and $\vol_k(z) d\lambda(z) = d\hat{\vol}_k(z)$ with $$(r_k u) (x) = \int_{\bbC^n} r_k(x,y)u(y) d\lambda(y),\,\forall u \in L^2(\bbC^n, k\phi_0,d\lambda).$$

From \eqref{full eqn1}, we have $B_{k\hat{\phi}_k} = B_{k\phi_0} (1+r_k + r_k^2+ \cdots)$. Thus we can compute $B_{ k\hat{\phi}_k} $ by compute each term of $B_{k\phi_0} \circ r_k^\ell$. The main idea is to use the stationary phase formula (see subsection \ref{Stationary Phase Formula}) to compute the $(B_{k\phi_0} \circ r_k^\ell) (x,y)$. Recall that for $\phi_0(z) = \sum_{j=1}^n \lambda_j |z_j|^2$ with $\lambda_j>0$, we have $$B_{k\phi_0} = k^n B_{\phi_0}(\sqrt{k}x, \sqrt{k} y) =k^n C_0 e^{-2k \sum_{j=1}^n \lambda_j(|y_j|^2- x_j \ol{y}_j)},$$ where $C_0 = \prod_{j=1}^n\( \frac{\lambda_j}{\pi}
\)$. Thus 
$$
\begin{aligned}
(B_{k\phi_0}\circ r_k^\ell)(z,z) &= \int_{\bbC^n} B_{k\phi_0}(z,w) r_k^\ell(w,z) d\lambda(w)\\
&= k^{n(\ell+1)} C_0^{\ell+1} \int_{\bbC^{n\ell} } e^{ikF_{\ell,z}(w^1,\ldots, w^\ell)} u_{k,\ell,z}(w^1,\ldots, w^\ell) dw^1 \cdots dw^\ell
\end{aligned}
$$
where $$F_{\ell,z} (w^1,\ldots, w^\ell) = \Psi(z,w^1) + \Psi(w^\ell,z) + \sum_{\nu=1}^{\ell-1} \Psi(w^\nu, w^{\nu+1})$$
and
$$u_{k,\ell,z}(w^1,\ldots, w^\ell) = v_{k}(w^\ell,z)\prod_{\nu=1}^{\ell-1} v_k(w^\nu, w^{\nu+1})$$
with 
$$\Psi(x,y) = 2i\(\sum_j \lambda_j(|y_j|^2 -z_j\ol{y}_j)\)$$ and 
$$v_k(x,y) = \(e^{2k(\psi_k(x) - \psi_k(y))} \vol_k^{-1}(x) \vol_k(y) -1\).$$

For fix $x\in \bbC^n$, $F_{\ell,z} $ has unique critical point at $(w^1,\ldots, w^\ell) = (z,\ldots, z)$ and satisfy the condition in stationary phase formula with $$\det \(\frac{k F_{\ell,z}''(z)}{2\pi i}\)^{\frac{-1}{2}} = \frac{1}{C_0^\ell k^{n\ell}}.$$ We now use the same notations as in subsection \ref{Stationary Phase Formula}. Since $g_z^{(\ell)}= 0$, we know  that $L_j^{(\ell),z} u_{k,\ell,z} = (\Delta_{\ell,z}^j u_{k,\ell,z})(z)$ where $$\Delta_{\ell,z} =  i \<(F''_{\ell,z}(z))^{-1}D,D\>  = \sum_{j=1}^n \frac{1}{\lambda_j} \(\sum_{\nu=1}^\ell \frac{\dr^2}{\dr w_j^\nu \dr \ol{w}_j^\nu} +\sum_{\nu<\mu}\frac{\dr^2}{\dr \ol{w}_j^\nu \dr w_j^{\mu}} \).$$

Thus 
\begin{equation}
\label{Bergman kernel equation}
\begin{aligned}
    B_{k\hat{\phi}_k} (0,0)&= \(B_{k\phi_0}\(\sum_{\ell=0}^\infty r_k^\ell\)\)(0,0)\(\vol(0)\)^{-1}\\
    &= B_{k\phi_0}(0,0)\sum_{\ell=0}^\infty \sum_{j=0}^\infty \frac{(2k)^{-j}}{j!} \Delta^j_{\ell,0}(u_{k,\ell,0})(0)\(\vol(0)\)^{-1}
\end{aligned}
\end{equation}

Since the derivative of $\psi_k$ at $0$ vanishing up to the order $2$ and  $\(L_j^{(\ell),0}u_{\ell,k,0}\)(0)$ equal to zero for $2j<\ell$. Futhermore, we need to compute $\(L_j^{(\ell), 0}u_{k,\ell,0}\)(0)$ for $j\in [m,3m]$ and $\ell\le j-m$ for the coefficient of $k^{n-m}$, i.e. it depend on the derivative of $\phi$ at zero up to $3m$. Otherwise, we also know $\(L_j^{(\ell), 0}u_{k,\ell,0}\)(0)$ is relative to $k^{n-m} $ for $m\in [\frac{2j-\ell+1}{3}, j]$.

Here is a table about the pair $(j,\ell)$ relate to the coefficient of $k^{n-m}$ when $\phi_1 = O(|z|^3)$. 

\begin{table}[h]
    \centering
    \begin{tabular}{c|ccccc}
        \diagbox{$j$}{$\ell$}& $0$ &$1$ &$2$&$3$&$\cdots$  \\
         \hline
        $0$ &  $k^n$& & & & \\
        $1$&&$k^{n-1}$&$k^{n-1}$&&\\
        $2$&&$k^{n-2}-k^{n-1}$&$k^{n-2}-k^{n-1}$&$k^{n-2}$&\\
        $3$&&$k^{n-3}-k^{n-1}$&$k^{n-3}-k^{n-1}$&$k^{n-3}-k^{n-2}$&$\cdots$\\
        $4$&&$k^{n-4}-k^{n-2}$&$k^{n-4}-k^{n-2}$&$k^{n-4}-k^{n-2}$&\\
        $\vdots$&&&$\vdots$&&$\ddots$
    \end{tabular}
\end{table}

\subsubsection{Manifold Cases}
Let $M$ be a compact complex manifold and $(L,h^L)$ be a positive Hermitian line bundle over $M$. Recall (cf. subsection \ref{pre of Bergman kernel}) that $P_k: L^2(M,L^k) \rightarrow H^0(M,L^k) $ is the Bergman projection and $P_k(x,y)$ is the Bergman kernel. For a fixed point $p\in M$, we can take a Chern-Moser trivialization $(z,s)$ on $U$ centered at $p$. We localized the Bergman kernel $P_{k,s}:L^2_{comp}(U,d\vol_M) \rightarrow  L^2(U, d\vol_M)$ and its corresponding diestribution kernel $P_{k,s}(x,y) $. We will consider $$B_{k,s} = e^{k\phi} P_{k,s} e^{-k\phi}.$$

From \cite[Proof of Theorem 3.8]{Hsiao-Marinescu17}, we have the following:
\begin{pro}
\label{k-infty approx} 
    For cut-off functions $\chi, \tau$ with $\tau= 1$ on $\supp \chi$, and let $\chi_k (z) = \chi(\frac{\sqrt{k}z}{\log k})$ and $\tau_k(z)= \tau(\frac{\sqrt{k}z}{\log k} )$. We have 
    $$\tau_k B_{k,s} \chi_k - \tau_k B_{k\hat{\phi}_k} \chi_k = O(k^{-\infty}) .$$
\end{pro}

With this we can using the full expansion of $B_{k\hat{\phi}_k}$ to get the full expansion of $B_{k,s} $ near $p$ with neighborhood $B(0,\frac{1}{\sqrt{k}})$.

\begin{thm}
\label{Bergman Full Expansion Theorem}
    There are $a_j(z) \in \cC^\infty(M)$ for $j=1,2,\ldots$ such that $$B_k(p)\sim \sum_{j=0}^\infty k^{n-j} a_j(p)$$ for all $p\in M$.
\end{thm}
\begin{proof}
    For all $N$ and a point $p$ with local trivialization $s$, there exists $c_{\frac{j}{2},s}(z)$ and $C_{N,s}>0$ such that $$|B_{k,s}(z) - \sum_{j=0}^{2N}k^{n-\frac{j}{2}}c_{\frac{j}{2},s}(z)| \le C_{N,s}k^{n-N},\, \forall |z|\le \frac{C}{\sqrt{k}}.$$ So we have there is $c_{\frac{j}{2}}$ and $C_N>0$ such that $$|B_k(p) - \sum_{j=0}^{2N}k^{n-\frac{j}{2}} c_{\frac{j}{2}}(p)| \le C_{N} k^{n-N}.$$ Also, when we fix a point $p\in M$, there is $a_{j}(p) $ such that $$|B_k(p) - \sum_{j=0}^N k^{n-j} a_j(p) | \le C_{N,p} k^{n-N}.$$ Combine two inequalities on manifold, we have $c_{\frac{j}{2}}(p)\equiv 0$ for odd $j$ and $a_j$ is smooth since all $c_{\frac{j}{2}} $ are smooth by the stationary phase formula.
\end{proof}

Then we will compute the first two coefficients. 
\begin{exe}
     Using the Chern-Moser trivialization $(z,s)$ centered at $p$ with local weight is $\phi = \phi_0+ \phi_1$. We will compute the first two coefficients of $P_k $ with $\phi_1=O(|z|^4)$. Here is a table about the pair $(j,\ell)$ relate to the coefficient of $k^{n-m}$ when $\phi_1 = O(|z|^4)$. 
    \begin{table}[h]
        \centering
        \begin{tabular}{c|ccccc}
            \diagbox{$j$}{$\ell$}& $0$ &$1$ &$2$&$3$&$\cdots$  \\
             \hline
            $0$ &  $k^n$& & & & \\
            $1$&&$k^{n-1}$&$k^{n-1}$&&\\
            $2$&&$k^{n-2}-k^{n-1}$&$k^{n-2}$&$k^{n-2}$&\\
            $3$&&$k^{n-3}-k^{n-2}$&$k^{n-3}-k^{n-2}$&$k^{n-3}-k^{n-2}$&$\cdots$\\
            $4$&&$k^{n-4}-k^{n-2}$&$k^{n-4}-k^{n-3}$&$k^{n-4}-k^{n-3}$&\\
            $\vdots$&&&$\vdots$&&$\ddots$
        \end{tabular}
    \end{table}\\

    The leading term at origin is just $$B_{k\phi_0}(0,0).$$ The second term relate to the case $(j,\ell) = (1,1),\,(1,2)$ and $(2,1)$ in \eqref{Bergman kernel equation}.
     $$\begin{aligned}
         (1,1)&: \frac{1}{2k}B_{k\phi_0}(0,0)\sum_{j=1}^n \frac{1}{\lambda_j}\(-\frac{\dr^2 \vol}{\dr z_j\dr\ol{z}_j} + 2 \left|\frac{\dr \vol}{\dr z_j}\right|^2\)\\
         (1,2)&:  \frac{1}{2k}B_{k\phi_0}(0,0)\sum_{j=1}^n \frac{1}{\lambda_j}\(- \left|\frac{\dr \vol}{\dr z_j}\right|^2\)\\
         (2,1)&: \frac{1}{8k^2}\sum_{j,\ell=1}^n \frac{1}{\lambda_j \lambda_\ell} \frac{\dr^4 2k \phi}{\dr z_j\dr \ol{z}_j\dr z_\ell\dr \ol{z}_\ell}\\
     \end{aligned}$$ 
     Then we get 
     \begin{equation}
    \label{Bergman second coeff}
     B_{k\phi_0}(0,0)\(\frac{1}{2}\sum_{j=1}^n \frac{1}{\lambda_j} \(\left|\frac{\dr \vol}{\dr z_j}\right|^2 - \frac{\dr^2 \vol}{\dr z_j\dr \ol{z}_j}\) + \frac{1}{4}\sum_{j,\ell=1}^n \frac{1}{\lambda_j\lambda_\ell}\frac{\dr^4\phi}{\dr z_j\dr\ol{z}_j\dr z_\ell\dr\ol{z}_\ell}\)(0).
     \end{equation}

    From \eqref{Bergman second coeff}, we can verify that $$a_1 (z) = \det\(\frac{\dot R^L}{2\pi}\)(z)\( \frac{1}{4\pi} \hat{r}(z) - \frac{1}{8\pi}r(z)\),$$
    where $\hat{r}$ denotes the scalar curvature associated with the curvature form  $R^L$ of the line bundle, and $r$ is the scalar curvature corresponding to the Hermitian metric on the tangent bundle of $M$. For precise definitions of $r$ and $\hat{r}$, we refer the reader to \cite{hsiao2011coefficientsasymptoticexpansionkernel}. This result was originally obtained by \cite{hsiao2011coefficientsasymptoticexpansionkernel}, \cite{lu1999lowerordertermsasymptotic}, and \cite{ma2012berezintoeplitz}.
\end{exe}

\subsection{Toeplitz Operator}
\label{Toeplitz operator}
Let $M$ be a compact complex manifold with a positive Hermitian line bundle $L\rightarrow M$ and the Bergman projection $P_k: L^2(M, L^k) \rightarrow H^0(M, L^k)$. Recall that the Toeplitz operator $T_{f,k} = P_k\circ M_f\circ P_k$ is defined in Section \ref{Introduction}. 

 First, $$\chi_k T_{f,k} \chi_k = \chi_k B_{k,s} \tau_k M_f \tau_k B_{k,s} \chi_k +O(k^{-\infty})$$ by the off-diagonal property of Bergman kernel where $\chi_k(z) =\chi\(\frac{\sqrt{k}z}{\log k}\)$ and $\tau_k(z)= \tau\(\frac{\sqrt{k}z}{\log k}\)$ with suitable cut-off functions $\chi,\tau$. From Proposition \ref{k-infty approx}, we have $$\chi_k T_{f,k} \chi_k = \chi_k B_{k\hat{\phi}_k} \tau_k M_f \tau_k B_{k\hat{\phi}_k} \chi_k +O(k^{-\infty}).$$

 We now apply our method, $$(\chi_kT_{f,k}\chi_k)= \chi_k B_{k\phi_0}\circ(1+r_k+\cdots) \circ \tau_k M_f \tau_k \circ B_{k\phi_0}\circ (1+r_k+\cdots)\chi_k +O(k^{-\infty}).$$

 Thus we need to compute near $0$: $$\(B_{k\phi_0}\circ r_k^{\ell_1} \circ \tau_k M_f \tau_k \circ B_{k\phi_0}\circ r_k^{\ell_2}\)(z,z)= k^{n(\ell+2)} C_0^{\ell+2} \int_{\bbC^{n(\ell+1)}}e^{iF_{\ell+1,z}(w)} u_{k,\ell_1,f,\ell_2,z}(w) dw,$$ where $\ell= \ell_1+ \ell_2$ and $$u_{k,\ell_1,f,\ell_2,z}(w) = u_{k,\ell_1,w^{\ell_1+1}}(w^1,\ldots, w^{\ell_1})\tau_k^2(w^{\ell_1+1}) f(w^{\ell_1+1}) u_{k,\ell_2,z}(w^{\ell_1+2},\ldots, w^{\ell+1}).$$

Let $T_{f,k}(x,y)$ be the distribution kernel of the Toeplitz operator defined by $$(T_{f,k}u)(x) = \int_{M} T_{f,k} (x,y) u(y) d\vol_M(y).$$ As in the case of the Bergman kernel, we also write $T_{f,k}(x) =T_{f,k}(x,x)$. Then we will get 
 \begin{equation}
 \label{Toeplitz operator equation}
 (\chi_k T_{f,k} \chi_k)(z,z) = B_{k\phi_0}(z,z) \sum_{\ell=0}^\infty \sum_{j=0}^\infty \frac{(2k)^{-j}}{j!} \Delta^j_{\ell+1,z} \(\sum_{\ell_1+\ell_2= \ell} u_{k,\ell_1, f,\ell_2,z}\)(z).\end{equation}
 Using the similar argument of Theorem \ref{Bergman Full Expansion Theorem}, we get 
\begin{thm}
    There exist $a_{j,f}\in \mathcal{C}^\infty(M)$ such that
    $$
    T_{f,k}(x)\sim \sum_{j=0}^\infty k^{n-j} a_{j,f}(x),
    $$
    which gives the full expansion of the Toeplitz operator.
\end{thm}

\begin{exe}
     Using the Chern-Moser trivialization $(z,s)$ centered at $p$ with local weight is $\phi = \phi_0+ \phi_1$. We will compute the first two coefficients of $T_{f,k} $ with $\phi_1=O(|z|^4)$. 
     Here is a table about the pair $(j,\ell)$ relate to the coefficient of $k^{n-m}$ when $\phi_1 = O(|z|^4)$. 
    \begin{table}[h]
        \centering
        \begin{tabular}{c|ccccc}
            \diagbox{$j$}{$\ell$}& $0$ &$1$ &$2$&$3$&$\cdots$  \\
             \hline
            $0$ &  $k^n$& & & & \\
            $1$&$k^{n-1}$&$k^{n-1}$&$k^{n-1}$&&\\
            $2$&$k^{n-2}$&$k^{n-2}-k^{n-1}$&$k^{n-2}$&$k^{n-2}$&\\
            $3$&$k^{n-3}$&$k^{n-3}-k^{n-2}$&$k^{n-3}-k^{n-2}$&$k^{n-3}-k^{n-2}$&$\cdots$\\
            $4$&$k^{n-4}$&$k^{n-4}-k^{n-2}$&$k^{n-4}-k^{n-3}$&$k^{n-4}-k^{n-3}$&\\
            $\vdots$&&&$\vdots$&&$\ddots$
        \end{tabular}
    \end{table}\\
     The leading term is just $$Cf(p)$$ where $C$ is the leading coefficient of $B_{k\phi_0}$. The second term is much complicated, we have to compute the case $(j,\ell) = (1,0),\,(1,1),\,(1,2)$ and $(2,1)$ in \eqref{Toeplitz operator equation}. The result is  $$C\(\frac{1}{2}\sum_{j=1}^n \frac{1}{\lambda_j} \(\left|\frac{\dr \vol}{\dr z_j}\right|^2 - \frac{\dr^2 \vol}{\dr z_j\dr \ol{z}_j}\)f + \frac{1}{4}\sum_{j,\ell=1}^n \frac{1}{\lambda_j\lambda_\ell}\frac{\dr^4\phi}{\dr z_j\dr\ol{z}_j\dr z_\ell\dr\ol{z}_\ell}+ \sum_{j=1}^n \frac{1}{\lambda_j}\frac{\dr^2 f}{\dr z_j\dr \ol{z}_j}\)(0)$$ where $C$ is as above. Here, we also reproduce the results obtained in \cite{hsiao2011coefficientsasymptoticexpansionkernel}.
\end{exe}

Another interesting result for the Toeplitz operator is the deformation quantization. Before stating the result, we first give some definitions. 
\begin{Def}
\label{Line bundle Poisson bracket}
    We denote by $\{\cdot, \cdot\}_L$ the Poisson bracket associated with the line bundle $L$, which is defined for $f,g\in \cC^\infty(M)$ by $$\{f,g\}_L = (\sqrt{-1}R^L)^{-1}(df,dg)$$ . Here, $(\sqrt{-1}R^L)^{-1} $ is determined by the relation $$(\sqrt{-1}R^L)^{-1}(\alpha, \gb)= \sqrt{-1}R^L(\alpha^\sharp\wedge \beta^\sharp)\text{ for all } \alpha, \beta \in T^*M,$$ where $ \alpha^\sharp \in TM$ is given by $$\sqrt{-1}R^L(\alpha^\sharp\wedge\ol{v}) = \alpha(v) \text{ for all }v\in TM.$$
\end{Def}

Recall that in deformation quantization \cite{Kontsevich2003}, we work with the space of formal power series $\cC^\infty(M)[[k^{-1}]]$ and define a noncommutative product(the star product $\star_k$) of $f,g\in \cC^\infty(M)\subset\cC^\infty(M)[[k^{-1}]]$ by  $$f\star_k g = \sum_{j=0}^\infty k^{-j} C_j(f,g),$$ where $C_0(f,g) = fg$. This extends naturally to the entire space $\cC^\infty(M)[[k^{-1}]]$ via the distributive law: $$\(\sum_{\ell_1\ge 0} k^{-\ell_1} f_{\ell_1}\)\star_k \(\sum_{\ell_2\ge 0 } k^{-\ell_2}g_{\ell_2}\) = \sum_{\ell_1,\ell_2\ge 0} k^{-\ell_1-\ell_2} f_{\ell_1} g_{\ell_2} + \sum_{\ell_1,\ell_2\ge 0, \,j\ge 1}C_j(f_{\ell_1},g_{\ell_2})k^{-\ell_1-\ell_2-j} $$ To endow $\cC^\infty(M)[[k^{-1}]]$ with an algebra structure under $\star_k$, we require that the product be associative: $$(f\star_k g) \star_k h = f\star_k(g\star_k h),$$ for all $f,g,h\in \cC^\infty(M)[[k^{-1}]]$. This is equivalent to 
\begin{equation}
\label{associative law}
\sum_{j+\ell=m} C_j(f,C_\ell(g,h)) = \sum_{j+\ell=m} C_\ell(C_j(f,g),h)
\end{equation}
 for all $m\ge 0$. 

We give a new proof of deformation quantization for projective manifolds.
\begin{thm}
    There exist functions $C_j(f,g)\in \cC^\infty(M)$, $j=0,1,\ldots$, such that 
    $$ T_{f,k} \circ T_{g,k} \sim \sum_{j=0}^\infty k^{n-j} T_{C_j(f,g),k}$$ in $L^2$-sense. Furthermore, $C_0(f,g) = fg$, $\{C_j(f,g)\}_{j\ge 0}$ satisfies \eqref{associative law}, and $C_1(f,g) - C_1(g,f) = i \{f,g\}_L$ where $\{\cdot , \cdot\}_L$ is the Poisson bracket.
\end{thm}

Similarly, we consider $T_{f,k}$ and $T_{g,k}$ with $\chi_k T_{f,k} \circ T_{g,k} \chi_k$. Then we compute 
$$
\begin{aligned}
 ( \chi_kB_{k\phi_0,s}\circ r_{k,s}^{\ell_1} \circ M_f\circ B_{k\phi_0,s}\circ r_{k,s}^{\ell_2}\circ M_g \circ B_{k\phi_0,s}\circ r_{k,s}^{\ell_3}\chi_k)(z,z) \\
= k^{n(\ell+3)} C_0^{\ell+3} \int_{\mathbb{C}^{n(\ell+2)}}e^{ikF_{\ell+2,z}(w)}u_{k,\ell_1,f,\ell_2,g,\ell_3,s,z}(w) dw,
\end{aligned}
$$
where $\ell=\ell_1+ \ell_2+\ell_3$ and
$$
\begin{aligned}
u_{k,\ell_1,f,\ell_2,g,\ell_3,z}(w) = u_{k,\ell_1,w^{\ell_1+1}}(\rho_k^2f)(w^{\ell_1+1}) u_{k,\ell_2,w^{\ell_2+2}} (\tau_k^2 g)(w^{\ell_1+\ell_2+2})u_{k,\ell_3,z},
\end{aligned}
$$
and
\begin{equation}
\label{Toeplitz composition equation}
\begin{aligned}
(\chi_k T_{f,k}\circ T_{g,k}&\chi_k)(z,z) \\&=B_{k\phi_0}(z,z)\sum_{\ell=0}^\infty \sum_{j=0}^\infty  \frac{(2k)^{-j}}{j!}\Delta_{\ell+2,z}^j\left(\sum_{\ell_1+\ell_2+\ell_3 = \ell} u_{k,\ell_1,f,\ell_2,g,\ell_3,z}\right)(z).
\end{aligned}
\end{equation}

Using the method of stationary phase on the Toeplitz operator and its composition, we can take $C_0(f,g)=fg$. For $C_\ell(f,g)$ with $\ell\ge 1$, we consider the asymptotic expansion $$\(T_{f,k}\circ T_{g,k} - \sum_{j=0}^{\ell-1} k^{-j} T_{C_j(f,g),k}\)(x)\sim k^{n} \sum_{j=\ell}^\infty k^{-j} b_j^{(\ell)}(x). $$ We then define $C_\ell(f,g)$ by identifying the leading coefficient on the right-hand side $$b_\ell^{(\ell)}(x) = a_0(x)C_\ell(f,g)(x)$$ where $a_0$ is the leading coefficient in the asymptotic of Bergman kernel (cf. Theorem~\ref{Bergman Full Expansion Theorem}). This procedure can be iterated to determine all $C_j(f,g)$ inductively. This construction automatically satisfies the associative law, as it follows from the properties of the Toeplitz operators. It remains to prove that $$C_1(f,g) - C_1(g,f) = i\{f,g\}_L=i\cdot i \sum_{j=1}^n \frac{1}{2\lambda_j}\(\frac{\dr f}{\dr z_j} \frac{\dr g}{\dr \ol{z}_j}- \frac{\dr f}{\dr \ol{z}_j} \frac{\dr g}{\dr z_j}\). $$ Since $C_0(f,g)= fg = gf = C_0(g,f)$, we have $T_{C_0(f,g),k} = T_{C_0(g,f),k}$. Then
$$
\chi_kT_{f,k}\circ T_{g,k}\chi_k -\chi_kT_{g,k}\circ T_{f,k}\chi_k  = k^{n-1} C (C_1(f,g)- C_1(g,f))
$$ where $B_{k\phi_0}(0,0)= Ck^n$. 
Calculating the coefficient of $k^{n-1}$ for $$\chi_k \(T_{f,k}\circ T_{g,k} - T_{g,k}\circ T_{f,k}\)\chi_k $$ corresponds to $(j,\ell) = (1,0),\,(1,1),\,(1,2)$ and $(2,1)$ in \eqref{Toeplitz composition equation} and we get 
$$
\begin{aligned}
    \(C_1(f,g) - C_1(g,f)\)(0)  &= \(\frac{1}{2}\sum_{j=1}^n \frac{1}{\lambda_j} \(\frac{\dr f}{\dr \ol{z}_j} \frac{\dr g }{\dr z_j}-\frac{\dr g}{\dr \ol{z}_j} \frac{\dr f}{\dr z_j}\)\)(0)\\
    &= i\{f,g\}_L(0) .
\end{aligned}$$

\subsection{Toeplitz Operator with Pseudodifferential Operator}
\label{Toeplitz operator with PDO}

\subsubsection{Euclidean Cases}
To generalize the Toeplitz operator, we consider replacing functions by pseudodifferential operators. For the details of pseudodifferential operators and microlocal analysis, we refer to \cite{grigis1994microlocal}, \cite{hörmander2007analysis}, and \cite{hintz_introduction_nodate}.

We define the symbol space 
\begin{Def}
    Let $m\in \bbR$, the space of symbols of order $m$, denoted by $$S^m(\bbC^n)\subset \cC^\infty(\bbC^{n}\times \bbC^n),$$ is the set of all functions $p(z, \xi) \in \cC^\infty(\bbC^{n}\times \bbC^n)$ satisfying the following: for all $\ga,\gb\in \bbN_0^{2n}$ and every compact subset $K\subset \bbC^n$, there exists $C_{K,\ga,\gb}>0$ such that $$|\dr_{z}^\ga \dr_\xi^\gb p(z,\theta)|\le C_{K,\ga,\gb} \<\theta\>^{m-|\gamma|},\, (z,\theta)\in K\times  \bbC^n$$ where $\<\theta\> = \sqrt{1+ |\theta|^2}$.
\end{Def}
We also introduce pseudodifferential operators 
\begin{Def}
    A pseudodifferential operator $P$ is a continuous operator from $\cC_c^\infty(\bbC^n) $ to $\cC^\infty(\bbC^n)$ satisfying $$(Pu)(z) = \frac{1}{(2\pi)^{2n}}\int_{\bbR^{2n}\times \bbR^{2n}} e^{i(z-w)\cdot \theta} p(z,\theta) u(w) dw d\theta +Fu$$ where $F$ is smoothing and $p(z,\theta) \sim \sum_{j=0}^\infty p_j(z,\theta) $ where $p_j\in S^{m-j}(\bbC^n)$. We also defined $P$ is classical operator if $p_j(z,\lambda \theta) = \lambda^{m-j} p_j(z,\theta) $ for all $\lambda\ge 1$, $\theta\ne 0$ and  $j$.
\end{Def}

\begin{rem}
    We will only discuss the classical pseudodifferential operator case in the following.
\end{rem}

To define Toeplitz operators with pseudodifferential operators, we have to define the classical pseudodifferential operator of order $m$ with $\frac{1}{2}$-density. For a classical symbol $p(z,,\theta)$, we can define $$\ti{P}_k: L^2(\bbC^n, k\hat{\phi}_k ,d\hat{\vol}_k) \rightarrow L^2(\bbC^n, k\hat{\phi}_k ,d\hat{\vol}_k) $$ by $$(\ti{P}_ku)(z) = \frac{e^{k\hat{\phi}_k}}{(2\pi)^{2n}\sqrt{\vol_k(z)}} \int_{\bbR^{2n}\times \bbR^{2n} } e^{i(z-w)\cdot \theta} p(z,\theta)(e^{-k\hat{\phi}_k(w)} \sqrt{\vol_k(w)}u(w))dw d\theta$$ where $d\hat{\vol}_k(z) = \vol_k(z)d \lambda(z)$, we still denote this by $P$. 

Then we start to define the Toeplitz operator with pseudodifferential operator, that is $$T_{P,k}= B_{k\hat{\phi}_k} \circ \ti{P}_k \circ B_{k\hat{\phi}_k}: L^2(\bbC^n, k\hat{\phi}_k,d\hat{\vol}_k) \rightarrow L^2(\bbC^n, k\hat{\phi}_k,d\hat{\vol}_k).$$ We want to compute the full expansion of the Toeplitz operator by our method. Replace $B_{k\hat{\phi}_k} $ by $ B_{k\phi_0} \circ(1+r_k+ \cdots)$, we get $$T_{P,k} =B_{k\phi_0} \circ(1+r_k+ \cdots)\circ \ti{P}_k \circ B_{k\phi_0} \circ(1+r_k+ \cdots). $$ The main trick here is combine $\ti{P}_k$ and $B_{k\phi_0}$ first, 
$$
\begin{aligned}
(\tau_k\ti{P}_k \tau_k\circ B_{k\phi_0})(x,y) &= \frac{\tau_k(x)e^{k\hat{\phi}_k(x)}}{(2\pi)^{2n} \sqrt{\vol(x)}}\int_{\bbR^{2n}\times \bbR^{2n}} e^{i(x-z)\cdot \theta} p(x,\theta) e^{-k\hat{\phi}_k(z)} \tau_k(z)\sqrt{\vol(z)}B_{k\phi_0}(z,y)d z d\theta\\
&=\frac{\tau_k(x)C_0}{(2\pi)^{2n}\sqrt{\vol(x)}}  k^{3n}e^{k\hat{\phi}_k(x)- k\hat{\phi}_k(y)} \int_{\bbR^{2n}\times \bbR^{2n}} e^{ik \Psi(z,\theta,x,y)}  a_k(z,\theta,x,y)d z d\theta\\
&= \sum_{\ell=0}^\infty \frac{\tau_k(x)C_0 k^{3n+m-\ell}}{(2\pi)^{2n}\sqrt{\vol(x)}} e^{k\hat{\phi}_k(x)- k\hat{\phi}_k(y)} \int_{\bbR^{2n}\times \bbR^{2n}} e^{ik \Psi(z,\theta,x,y)} b_\ell(z,\theta,x,y) d z d\theta\\
\end{aligned}
$$
where $$a_k(z,\theta,x,y) = e^{-k\psi_k(z)+k\psi_k(y)}p(x,k\theta)\tau_k(z)\sqrt{\vol(z)},$$ $$b_\ell(z,\theta,x,y)= e^{-k\psi_k(z)+k\psi_k(y)}p_\ell(x,\theta)\tau_k(z)\sqrt{\vol(z)},$$ 
$$\Psi(z,\theta,x,y) =(x-z)\cdot \theta+i\sum_{j=1}^n \lambda_j(|z_j|^2+|y_j|^2- 2z_j \ol{y}_j) $$ with $\psi_k = \hat{\phi}_k - \phi_0$,  $x_j = x_j^1+ i x_j^2$, $y_j= y_j^1+ i y_j^2 $, $z_j= z_j^1+ i z_j^2$ and $\theta_j = \theta_j^1 + i\theta_j^2$. Using the Stationary phase formula (see Theorem \ref{stationary phase 2}), we have $$(\tau_k\ti{P}_k \tau_k\circ B_{k\phi_0})(x,y) = \frac{B_{k\phi_0}(x,y)}{\sqrt{\vol(x)}}e^{k\psi_k(x)-k\psi_k(y)}\sum_{\ell=0}^\infty  \sum_{j=0}^\infty  \frac{k^{m-\ell-j}}{2^j}(\Delta_{z,\theta,x,y}^j\wt{b_\ell})(\ti{z}(x,y), \ti{\theta}(x,y),x,y)$$
where $$\Delta_{z,\theta,x,y}  = \sum_{j=1}^n  2\lambda_j \(\frac{\dr^2 }{\dr \theta_j^1\dr \theta_j^1 }+\frac{\dr^2 }{\dr \theta_j^2\dr \theta_j^2 }\)-2i \sum_{j=1}^n \(\frac{\dr^2}{\dr z_j^1\dr \theta_j^1}+ \frac{\dr^2}{\dr z_j^2\dr \theta_j^2}\), $$ 
$$\ti{z}(x,y) = x,\, \ti{\theta}^1_j(x,y) =  2i \lambda_j(x_j^1- \ol{y}_j),\, \ti{\theta}_j^2(x,y) =  2i \lambda_j (x_j^2- i \ol{y}_j
), $$ and $\ti{}$ means the almost analytic extension. We will simplify the notation $$(\Delta_{z,\theta,x,y}^j\wt{b_\ell})(\ti{z}(x,y), \ti{\theta}(x,y),x,y) = (\Delta_{z,\theta}^j\wt{b_\ell})(x,y).$$

When applying our method $$
\begin{aligned}
(\chi_kT_{P,k}\chi_k)(z,z)
&= k^m B_{k\phi_0}(z,z)\sum_{j,\ell=0}^\infty k^{-j}\sum_{j_1,j_2=0}^\infty \frac{(2k)^{-j_1}}{j_1!}\frac{(2k)^{-j_2}}{j_2!}\sum_{\ell_1+\ell_2=\ell}\\
&\Delta_{\ell+1}^{j_1} \[u_{k,\ell_1,w^{\ell_1+1}}\frac{e^{k\psi_k(w^{\ell_1+1})-k\psi_k(w^{\ell_1+2})}}{\sqrt{\vol(w^{\ell_1+1})}}(\Delta_{z,\theta}^{j_2}\wt{b_j})(w^{\ell_1+1},w^{\ell_1+2})u_{k,\ell_2,z}\](z)
\end{aligned}
$$  if $\ell_2 = 0 $, $w^{\ell_1+2} = z$.

This formula tells us that we can have a full expansion of the Toeplitz operator near the origin. 
\begin{thm}
    Let $P$ be a classical Pseudodifferential operator with order $m$, then $T_{P,k} $ has an asymptotic as $$T_{P,k}(z,z)\sim\sum_{j=0}^{\infty} c_{j,P}(z) k^{n+m-j}$$ near origin, where $c_{0,P}(z) = C_0(z) p_0(z,-Jd\phi(z))$ with $C_0$ is the leading term of Bergman kernel.
\end{thm}

\begin{exe}
    We compute the first two coefficients at the origin. With assume that $\phi_1= \phi-\phi_0 = O(|z|^4)$. We will write $z_j = x_j+ i y_j$ and $p(z,\theta)\sim p_0(z,\theta) + p_1(z,\theta) + \cdots$. For the leading term, namely the coefficient of  $k^{n+m} $, we obtain $$c_0(0) = Cp_0(0,0)$$ where $C$ denotes the coefficient of $k^n$ in the expansion of $B_{k\phi_0}(0,0)$.   
    For the second term, we get 
    $$
    \begin{aligned}
    &C\Biggl[p_1 + \frac{i}{2}\sum_{j=1}^n \(\frac{\dr^2 p_0}{\dr x_j\dr \theta_j^1}+\frac{\dr^2 p_0}{\dr y_j\dr \theta_j^2}  \)\\
    &+\frac{1}{2}\sum_{j=1}^n\(\frac{1}{4\lambda_j}(\frac{\dr^2 p_0}{\dr x_j \dr x_j} +\frac{\dr^2 p_0}{\dr y_j \dr y_j})+ \lambda_j(\frac{\dr^2 p_0}{\dr \theta^1_j\dr \theta^1_j}+\frac{\dr^2 p_0}{\dr \theta^2_j\dr \theta^2_j})\) \\
    &+\frac{1}{2}\sum_{j=1}^n\frac{1}{\lambda_j}\(\left|\frac{\dr \vol}{\dr z_j}\right|^2 - \frac{\dr^2 \vol}{\dr z_j\dr \ol{z}_j}\)p_0+ \frac{1}{4} \sum_{j=1}^n \frac{1}{\lambda_j\lambda_\ell} \frac{\dr^4 \phi}{\dr z_j\dr\ol{z}_j\dr z_\ell\dr\ol{z}_\ell} 
    p_0\Biggr]\Bigg|_{z=0,\theta=0}
    \end{aligned}
    $$
    where $$p_1 + \frac{i}{2}\sum_{j=1}^n \(\frac{\dr^2 p_0}{\dr x_j\dr \theta_j^1}+\frac{\dr^2 p_0}{\dr y_j\dr \theta_j^2}  \)$$ is the subprincipal symbol and 
    $$
    \begin{aligned}
    \frac{1}{2}&\sum_{j=1}^n\(\frac{1}{4\lambda_j}(\frac{\dr^2 p_0}{\dr x_j \dr x_j} +\frac{\dr^2 p_0}{\dr y_j \dr y_j})+ \lambda_j(\frac{\dr^2 p_0}{\dr \theta^1_j\dr \theta^1_j}+\frac{\dr^2 p_0}{\dr \theta^2_j\dr \theta^2_j})\) \\
    &=\frac{1}{2}\sum_{j=1}^n \frac{1}{4\lambda_j}\(\frac{\dr^2}{(\dr x_j)^2}+\frac{\dr^2}{(\dr y_j)^2} \)(p_0(z,-Jd\phi(z)))+\frac{1}{2}\sum_{j=1}^n\[\frac{\dr^2 p_0}{\dr y_j \dr \theta_j^1}- \frac{\dr^2p_0}{\dr x_j \dr \theta_j^2}\](z,-Jd\phi(z))\\
    &= \frac{1}{2}\Delta_L(p_0(z,-Jd\phi(z))+ \hat{\sigma}_{p_0}(z)
    \end{aligned}
    $$ where $\Delta_L = \sum_{j=1}^n \frac{1}{4\lambda_j}\(\frac{\dr^2}{(\dr x_j)^2}+\frac{\dr^2}{(\dr y_j)^2} \)$ locally and $\hat{\sigma}_{p_0}$ is defined below.
\end{exe}

\begin{Def}
\label{second term of pseudotoeplitz}
    $\hat{\sigma}_{p_0}$ is a smooth function on $\bbC^n$ defined as follows: for all $p\in \bbC^n$, take a coordinate $z(p) =0 $, $\phi(z) = \sum_{j=1}^n \lambda_j |z_j|^2 +O(|z|^3)$ and $\<\cdot|\cdot\>$ flat at $p$, then $$\hat{\sigma}_{p_0}(p) := \frac{1}{2}\sum_{j=1}^n\[\frac{\dr^2 p_0}{\dr y_j \dr \theta_j^1}- \frac{\dr^2p_0}{\dr x_j \dr \theta_j^2}\](p,-Jd\phi(p)).$$ This is well-defined by the following.
\end{Def}
\begin{proof}[Proof of well-defined of Definition~\ref{second term of pseudotoeplitz}]
For a holomorphic cooridnate change $w = (u,v) = (u(x,y), v(x,y)) = (u(z),v(z))$ under the same trivialization satisfying $(u(0,0),v(0,0))=(0,0)$,  there is an induced natural coordinate change on the cotangent bundle, given by $\kappa((x,y,\theta_j^1, \theta_j^2))= (u,v,\eta_j^1,\eta_j^2)$ where $$\eta_j^1 = \sum_{\ell=1}^n \( \frac{\dr x_\ell}{\dr u_j}\theta_\ell^1+ \frac{\dr y_\ell}{\dr u_j} \theta_\ell^2\)$$ and $$\eta_j^2 = \sum_{\ell=1}^n \( \frac{\dr x_\ell}{\dr v_j}\theta_\ell^1+ \frac{\dr y_\ell}{\dr v_j} \theta_\ell^2\).$$
    
    Let $p_0(x,y,\theta_j^1, \theta_j^2)$ and $\hat{p}_0(u,v,\eta_j^1,\eta_j^2)$ denote the principal symbol using coordinates $(x,y)$ and $(u,v)$ respectively.  We know that $$\hat{p}_0 \circ \kappa = \hat{p}_0\(u(x,y), v(x,y), \eta_j^1(x,y,\theta_j^1,\theta_j^2),\eta_j^2(x,y,\theta_j^1,\theta_j^2)\) = p_0(x,y,\theta_j^1,\theta_j^2).$$
    
    Applying the Cauchy-Riemann equation, we have 
    $$
    \begin{aligned}
    &\quad\!~\sum_{j=1}^n \[\frac{\dr^2 p_0}{\dr y_j \dr \theta_j^1}- \frac{\dr^2 p_0}{\dr x_j \dr \theta_j^2}\] (x,y,\theta_j^1,\theta_j^2)= \sum_{j=1}^n \[\frac{\dr^2 \hat{p}_0\circ \kappa}{\dr y_j \dr \theta_j^1}- \frac{\dr^2 \hat{p}_0\circ \kappa}{\dr x_j \dr \theta_j^2}\] (x,y,\theta_j^1,\theta_j^2) \\
    &=\sum_{r=1}^n\sum_{s=1}^n\bigg[\frac{\dr^2 \hat{p}_0}{\dr u_r \dr \eta_s^1}\(\sum_{j=1}^n \frac{\dr u_r}{\dr y_j}\frac{\dr x_j}{\dr u_s}- \frac{\dr u_r}{\dr x_j}\frac{\dr y_j}{\dr u_s}\)+\frac{\dr^2 \hat{p}_0}{\dr v_r \dr \eta_s^1}\(\sum_{j=1}^n \frac{\dr v_r}{\dr y_j}\frac{\dr x_j}{\dr u_s}- \frac{\dr v_r}{\dr x_j}\frac{\dr y_j}{\dr u_s}\)\\
    &+\frac{\dr^2 \hat{p}_0}{\dr u_r \dr \eta_s^2}\(\sum_{j=1}^n \frac{\dr u_r}{\dr y_j}\frac{\dr x_j}{\dr v_s}-\frac{\dr u_r}{\dr x_j}\frac{\dr y_j}{\dr v_s}\)+\frac{\dr^2 \hat{p}_0}{\dr v_r \dr \eta_s^2}\(\sum_{j=1}^n \frac{\dr v_r}{\dr y_j}\frac{\dr x_j}{\dr v_s}-\frac{\dr v_r}{\dr x_j}\frac{\dr y_j}{\dr v_s}\)\bigg]\\
    &=\sum_{r=1}^n\sum_{s=1}^n\bigg[\frac{\dr^2 \hat{p}_0}{\dr u_r \dr \eta_s^1}\(\frac{\dr v_r}{\dr u_s}\)+\frac{\dr^2 \hat{p}_0}{\dr v_r \dr \eta_s^1}\(\frac{\dr u_r }{\dr u_s}\)+\frac{\dr^2 \hat{p}_0}{\dr u_r \dr \eta_s^2}\(-\frac{\dr v_r}{\dr v_s}\)+\frac{\dr^2 \hat{p}_0}{\dr v_r \dr \eta_s^2}\(\frac{u_r}{v_s}\)\bigg]\\
    &=\[ \sum_{j=1}^n\(\frac{\dr^2 \hat{p}_0}{\dr v_j \dr \eta_j^1}- \frac{\dr^2 \hat{p}_0}{\dr u_j \dr \eta_j^2}\)\circ \kappa\](x,y,\theta_j^1, \theta_j^2).
    \end{aligned}
    $$
    Thus, this term is also invariant under the coordinate change.
\end{proof}

\begin{thm}
    The first two coefficients of Toeplitz operator with pseudodifferential operator $P$ of order $m$ is the following $$c_{0,p} (z) = a_0(z) p_0(z,-Jd\phi(z))$$ and $$c_{1,P}(z) = a_1(z)p_0(z,-Jd\phi(z)) +a_0(z)\(\sigma_{sub}(P)(z,-Jd\phi(z))+ \frac{1}{2}\Delta_{L} (p_0(z,-Jd\phi(z))) + \hat{\sigma}_{p_0}(z)\)$$ where $a_0,\,a_1$ are the first two coefficients of the Bergman kernel and $\sigma_{sub}$ means the subprincipal symbol.
\end{thm}

\subsubsection{Manifold Cases}
Now, back to the setting of a compact complex manifold $M$ with a positive Hermitian line bundle $(L,h^L)$.
\begin{Def}
A pseudodifferential operator is a continuous operator $Q:\cC^\infty(M,L^k)\rightarrow \cC^\infty (M,L^k)$ satisfying the following two condition: 
\begin{enumerate}
    \item For $\supp \chi \cap \supp\tau = \emptyset$, we have a operator $K$ with smooth kernel such that $K= \chi Q \tau$. 
    \item For $\chi $ is a cut-off function, we have $\chi Q \chi$ is a pseudodifferential operator in a chart with local trivialization. That is for $\chi u = \hat{u}\otimes s^k$ and $(\chi Q\chi)(u) = \hat{v}\otimes s^k$ we have 
    $$\hat{v}(z) = \frac{e^{k\phi(z)}}{(2\pi)^{2n}\sqrt{\vol(z)}}\int_{\bbR^{2n}\times \bbR^{2n}}e^{i(z-w)\cdot \theta} q(z,\theta)\(e^{-k\phi(w)} \sqrt{\vol(w)}\hat{u}(w)\)dw d\theta$$ where $q(z,\theta)$ is the symbol of $Q$.
\end{enumerate}
\end{Def}

From this and the off-diagonal property, we can work on $$\chi_k T_{Q,k} \chi_k = \chi_k P_k \tau_k Q_k \tau_k P_k \chi_k +O(k^{-\infty} )$$ which reduce to Euclidean case by $\chi_k P_k \tau_k = \chi_k B_{k\hat{\phi}_k} \tau_k +O(k^{-\infty} )$.
We can formulate the Toeplitz operators with pseudodifferential operators as follows: 
\begin{thm}
    Let $P$ be a classical pseudodifferential operator with order $m$ and the symbol $p(z, \theta) \sim p_0(z,\theta)+p_1(z,\theta)+ \cdots$ acting on $L^2(M,L^k)$, then $T_{P,k} $ has an asymptotic as $$T_{P,k}\sim\sum_{j=0}^{\infty}  k^{n+m-j}c_{j,P}$$ where $c_{j, P} $ is smooth function on $M$. Here, the first two coefficients $c_{0, P}$ and $c_{1, P}$ can be expressed similarly to the Euclidean case.
\end{thm}

We also want to understand the composition of Toeplitz operator $$T_{P,k} \circ T_{Q,k}$$ with $P$ and $Q$ are classical pseudodifferential operator of order $m_1 $ and $m_2 $ with symbol $p(z,\theta)\sim \sum_{j=0}^\infty p_j(z,\theta)$ and $q(z,\theta)\sim \sum_{j=0}^\infty q_j(z,\theta)$ respectively. So we need to compute $$
\begin{aligned}
&(\chi_kT_{P,k}\circ T_{Q,k}\chi_k)(z,z)\\
&= k^{m_1+m_2}B_{k\phi_0}(z,z)\sum_{s_1,s_2,\ell=0}^\infty k^{-s_1-s_2}\sum_{j_1,j_2,j_3=0}^\infty \frac{(2k)^{-j_1-j_2-j_3}}{j_1!j_2!j_3!}\sum_{\ell_1+\ell_2+\ell_3=\ell}\Delta_{\ell+2}^{j_1} \\
&\Biggl[u_{k,\ell_1,w^{\ell_1+1}}\frac{e^{k\psi_k(w^{\ell_1+1})-k\psi_k(w^{\ell_1+2})}}{\sqrt{\vol(w^{\ell_1+1})}}(\Delta_{z,\theta}^{j_2}\wt{a_{s_1}})(w^{\ell_1+1},w^{\ell_1+2})u_{k,\ell_2,w^{\ell_1+\ell_2+2}}\\
&\frac{e^{k\psi_k(w^{\ell_1+\ell_2+2})-k\psi_k(w^{\ell_1+\ell_2+3})}}{\sqrt{\vol(w^{\ell_1+\ell_2+2})}}(\Delta_{z,\theta}^{j_3}\wt{b_{s_2}})(w^{\ell_1+\ell_2+2},w^{\ell_1+\ell_2+3})u_{k,\ell_3,z}\Biggr](z)
\end{aligned}
$$
where $w^{\ell_1+\ell_2+3}= z$ if $\ell_3=0$ and  $a_{s_1} $, $b_{s_2}$ are relate to $p_{s_1} $, $q_{s_2}$ respectively.

\begin{Def}
    The Poisson bracket of the symbols on $M$ is define as $$\{p,q\}_{\Psi} = \sum_{j=1}^n\[ \frac{\dr p}{\dr \theta_j^1} \frac{\dr q}{\dr x_j}+\frac{\dr p}{\dr \theta_j^2} \frac{\dr q}{\dr y_j}-
    \frac{\dr p}{\dr x_j}\frac{\dr q}{\dr \theta_j^1} -\frac{\dr p}{\dr y_j}\frac{\dr q}{\dr \theta_j^2} \].$$ Note that this Poisson bracket is associated with the principal symbol of the commutator of two pseudodifferential operators $P$ and $Q$, with principal symbols $p$ and $q$, respectively.
\end{Def}
The quantization of pseudodifferential operators can similarly be formulated as follows:
\begin{thm}
    Let $P$ and $Q$ be pseudodifferential operators of orders $m_1$ and $m_2$, respectively. There exist pseudodifferential operators $C_j(P, Q; k)$, for $j = 0, 1, \dots$, of the form $$C_j(P, Q; k) = \sum_{s=0}^{2j} k^s D_{j,s}(P, Q),$$ where $D_{j,s}(P, Q)$ is a pseudodifferential operator of order $m_1 + m_2 - s$, and $C_0(P, Q; k) = P \circ Q$, such that 
    $$
    T_{P,k} \circ T_{Q,k} \sim \sum_{j=0}^\infty k^{n-j} T_{C_j(P,Q;k),k}
    $$ in $L^2$-sense.
\end{thm}
\begin{rem}
    There exists a pseudodifferential operator $D_1(P, Q)$ of order $m_1 + m_2$ such that the commutator satisfies
    $$
    \[T_{P,k} \circ T_{Q,k} - T_{Q,k}\circ T_{P,k} \](x)\sim\[T_{P\circ Q-Q\circ P,k} + \frac{1}{k}T_{D_1(P,Q),k} + o(k^{n+m_1+m_2-2})\](x).
    $$
\end{rem}
\begin{rem}
    The above theorem suggests that we should quantize pseudodifferential operators of the form $\sum_{s=0}^\infty k^s D_s(P, Q)$, where $D_s(P, Q)$ is a pseudodifferential operator of order $m_1 + m_2 - s$.
\end{rem}

\begin{exe}
    Using the same local trivialization $(z,s)$ centered at $p$ with local weight $\phi= \phi_0+ \phi_1$ with $\phi_1= O(|z|^4)$. 
    Let $P$ and $Q$ be classical pseudodifferential operators of orders $m_1$ and $m_2$, respectively.
    We can compute $T_{P\circ Q,k}$ and $T_{P,k}\circ T_{Q,k}$. We will write $p(z,\theta) \sim \sum_{j=0}^\infty p_j(z,\theta)$ and $q(z,\theta)\sim \sum_{j=0}^\infty q_j(z,\theta)$. We know that 
    $$
    \begin{aligned}
    T_{P\circ Q,k} (0,0) &= k^{m_1+m_2}B_{k\phi_0}(0,0) p_0(0,0)q_0(0,0)\\
    &+k^{m_1+m_2-1}B_{k\phi_0}(0,0)\Biggl[ p_1q_0+p_0q_1 +(-i) \sum_{j=1}^n \frac{\dr p_0}{\dr \theta_j^1} \frac{\dr q_0}{\dr x_j} + \frac{\dr p_0}{\dr \theta_j^2} \frac{\dr q_0}{\dr y_j}\\
    &+\frac{1}{2}\sum_{j=1}^n  \frac{1}{\lambda_j}\(\left|\frac{\dr \vol}{\dr z_j}\right|^2 - \frac{\dr^2 \vol}{\dr z_j\dr \ol{z}_j}\)p_0q_0+ \frac{1}{4}\sum_{j,\ell=1}^n \frac{1}{\lambda_j\lambda_\ell} \frac{\dr^4\phi}{\dr z_j \dr \ol{z}_j\dr z_\ell \dr \ol{z}_\ell} p_0q_0\\
    &+\frac{1}{2}\sum_{j=1}^n \Biggl(\frac{1}{4\lambda_j} \(\frac{\dr^2p_0}{\dr x_j\dr x_j}q_0+2 \frac{\dr p_0}{\dr x_j}\frac{\dr q_0}{\dr x_j}+ p_0\frac{\dr^2 q_0}{\dr x_j\dr x_j} \)\\
    &+\frac{1}{4\lambda_j} \(\frac{\dr^2p_0}{\dr y_j\dr y_j}q_0+2 \frac{\dr p_0}{\dr y_j}\frac{\dr q_0}{\dr y_j} + p_0\frac{\dr^2 q_0}{\dr y_j\dr y_j}\)\\
    &+i\(\frac{\dr^2 p_0  }{\dr x_j \dr\theta_j^1}q_0+ \frac{\dr p_0  }{\dr x_j}\frac{\dr q_0}{\dr \theta_j^1} + \frac{\dr p_0}{\dr\theta_j^1}\frac{\dr q_0  }{\dr x_j}+ p_0\frac{\dr^2 q_0  }{\dr x_j \dr\theta_j^1}\)\\
    &+i\(\frac{\dr^2 p_0  }{\dr y_j \dr\theta_j^2}q_0+ \frac{\dr p_0  }{\dr y_j}\frac{\dr q_0}{\dr \theta_j^2} + \frac{\dr p_0}{\dr\theta_j^2}\frac{\dr q_0  }{\dr y_j}+ p_0\frac{\dr^2 q_0  }{\dr y_j \dr\theta_j^2}\)\\
    &+\lambda_j\(\frac{\dr^2 p_0}{\dr \theta_j^1 \dr\theta_j^1}q_0+ 2\frac{\dr p_0}{\dr \theta_j^1}\frac{\dr q_0}{\dr \theta_j^1} + p_0\frac{\dr^2 q_0}{\dr \theta_j^1 \dr\theta_j^1}\)\\
    &+\lambda_j\(\frac{\dr^2 p_0}{\dr \theta_j^2 \dr\theta_j^2}q_0+2 \frac{\dr p_0}{\dr \theta_j^2}\frac{\dr q_0}{\dr \theta_j^2}+ p_0\frac{\dr^2 q_0}{\dr \theta_j^2 \dr\theta_j^2}\)\Biggr)\Biggr]\Bigg|_{z=0, \theta=0}+o(k^{m_1+m_2+n-1})
    \end{aligned}
    $$
    and
     $$
    \begin{aligned}
    (T_{P,k}\circ T_{Q,k})(0,0) &= k^{m_1+m_2}B_{k\phi_0}(0,0) p_0(0,0)q_0(0,0)\\
    &+k^{m_1+m_2-1}B_{k\phi_0}(0,0)\Biggl[ p_1q_0+p_0q_1  \\
    &+\frac{1}{2}\sum_{j=1}^n  \frac{1}{\lambda_j}\(\left|\frac{\dr \vol}{\dr z_j}\right|^2 - \frac{\dr^2 \vol}{\dr z_j\dr \ol{z}_j}\)p_0q_0+ \frac{1}{4}\sum_{j,\ell=1}^n \frac{1}{\lambda_j\lambda_\ell} \frac{\dr^4\phi}{\dr z_j \dr \ol{z}_j\dr z_\ell \dr \ol{z}_\ell} p_0q_0\\
    &+\frac{1}{2}\sum_{j=1}^n \Biggl(\frac{1}{4\lambda_j} \(\frac{\dr^2p_0 }{\dr x_j\dr x_j}q_0+ p_0\frac{\dr^2 q_0 }{\dr x_j\dr x_j}+\frac{\dr p_0 }{\dr x_j}\frac{\dr q_0 }{\dr x_j}\)\\
    &+\frac{1}{4\lambda_j}\((-i)\frac{\dr p_0 }{\dr x_j}\frac{\dr q_0 }{\dr y_j}+i\frac{\dr p_0 }{\dr y_j}\frac{\dr q_0 }{\dr x_j}\)\\
    &+\frac{1}{4\lambda_j}\(\frac{\dr p_0 }{\dr y_j}\frac{\dr q_0 }{\dr y_j}+\frac{\dr^2p_0 }{\dr y_j\dr y_j}q_0 + p_0\frac{\dr^2 q_0 }{\dr y_j\dr y_j}\)\\
    &+i\(\frac{\dr^2 p_0 }{\dr x_j \dr\theta_j^1}q_0+ p_0\frac{\dr^2 q_0 }{\dr x_j \dr\theta_j^1}+\frac{\dr^2 p_0 }{\dr y_j \dr\theta_j^2}q_0+ p_0\frac{\dr^2 q_0 }{\dr y_j \dr\theta_j^2}\)\\
    &+\frac{i}{2}\(\frac{\dr p_0 }{\dr x_j}\frac{\dr q_0}{\dr \theta_j^1}+(-i)\frac{\dr p_0 }{\dr x_j}\frac{\dr q_0}{\dr \theta_j^2}+i \frac{\dr p_0 }{\dr y_j}\frac{\dr q_0}{\dr \theta_j^1}+\frac{\dr p_0 }{\dr y_j}\frac{\dr q_0}{\dr \theta_j^2}\)\\
    &-\frac{i}{2}\(\frac{\dr p_0}{\dr \theta_j^1}\frac{\dr q_0 }{\dr x_j}+i\frac{\dr p_0}{\dr \theta_j^2}\frac{\dr q_0 }{\dr x_j}+(-i) \frac{\dr p_0}{\dr \theta_j^1}\frac{\dr q_0 }{\dr y_j}+\frac{\dr p_0 }{\dr \theta_j^2}\frac{\dr q_0}{\dr y_j}\)\\
    &+\lambda_j\(\frac{\dr^2 p_0}{\dr \theta_j^1 \dr\theta_j^1}q_0 + p_0\frac{\dr^2 q_0}{\dr \theta_j^1 \dr\theta_j^1}+\frac{\dr^2 p_0}{\dr \theta_j^2 \dr\theta_j^2}q_0+ p_0\frac{\dr^2 q_0}{\dr \theta_j^2 \dr\theta_j^2}\)\\
    &+\lambda_j\(\frac{\dr p_0}{\dr \theta_j^1}\frac{\dr q_0}{\dr \theta_j^1}+(-i)\frac{\dr p_0}{\dr \theta_j^1}\frac{\dr q_0}{\dr \theta_j^2}+i\frac{\dr p_0}{\dr \theta_j^2}\frac{\dr q_0}{\dr \theta_j^1}+\frac{\dr p_0}{\dr \theta_j^2}\frac{\dr q_0}{\dr \theta_j^2}\)\Biggr)\\
    &\Biggr]+o(k^{m_1+m_2+n-1})
    \end{aligned}
    $$
    Thus we can have $$C_0(P,Q) = P\circ Q$$ and compute $C_1(P,Q)$ on each Chern-Moser trivialization center at $p$ in the following:
    $$
    \begin{aligned}
    C_1(P,Q)= \sum_{j=1}^n &\Biggl[-\frac{1}{8\lambda_j}\frac{\dr p_0 }{\dr x_j}\frac{\dr q_0 }{\dr x_j}-\frac{i}{8\lambda_j}\frac{\dr p_0 }{\dr x_j}\frac{\dr q_0 }{\dr y_j}+\frac{i}{8\lambda_j}\frac{\dr p_0 }{\dr y_j}\frac{\dr q_0 }{\dr x_j}-\frac{1}{8\lambda_j}\frac{\dr p_0 }{\dr y_j}\frac{\dr q_0 }{\dr y_j}\\
    &+\frac{-i}{4} \frac{\dr p_0}{\dr x_j}\frac{\dr q_0}{\dr \theta_j^1}+\frac{1}{4} \frac{\dr p_0}{\dr x_j}\frac{\dr q_0}{\dr \theta_j^2}+\frac{-1}{4} \frac{\dr p_0}{\dr y_j}\frac{\dr q_0}{\dr \theta_j^1}+\frac{-i}{4} \frac{\dr p_0}{\dr y_j}\frac{\dr q_0}{\dr \theta_j^2}\\
    &+\frac{i}{4} \frac{\dr p_0}{\dr \theta_j^1}\frac{\dr q_0}{\dr x_j}+\frac{1}{4}\frac{\dr p_0}{\dr \theta_j^2} \frac{\dr q_0}{\dr x_j}+\frac{-1}{4} \frac{\dr p_0}{\dr \theta_j^1}\frac{\dr q_0}{\dr y_j}+\frac{i}{4} \frac{\dr p_0}{\dr \theta_j^2}\frac{\dr q_0}{\dr y_j}\\
    &+ \frac{-\lambda_j}{2}\frac{\dr p_0}{\dr \theta_j^1}\frac{\dr q_0}{\dr \theta_j^1}+\frac{-i\lambda_j}{2}\frac{\dr p_0}{\dr \theta_j^1}\frac{\dr q_0}{\dr \theta_j^2}+\frac{i\lambda_j}{2}\frac{\dr p_0}{\dr \theta_j^2}\frac{\dr q_0}{\dr \theta_j^1}+\frac{-\lambda_j}{2}\frac{\dr p_0}{\dr \theta_j^2}\frac{\dr q_0}{\dr \theta_j^2}\Biggr].\
    \end{aligned}
    $$

    Thus 
    $$
    \begin{aligned}
    C_1(P,Q)-C_1(Q,P) &=\sum_{j=1}^n \Biggl[\frac{-i}{4\lambda_j} \frac{\dr p_0}{\dr x_j} \frac{\dr q_0}{\dr y_j} + \frac{i}{4\lambda_j}\frac{\dr p_0}{\dr y_j}\frac{\dr q_0}{\dr x_j}\\
    &+\frac{-i}{2}\frac{\dr p_0}{\dr x_j} \frac{\dr q_0}{\dr \theta_j^1}+\frac{-i}{2}\frac{\dr p_0}{\dr y_j} \frac{\dr q_0}{\dr \theta_j^2}+\frac{i}{2} \frac{\dr p_0}{\dr \theta_j^1}\frac{\dr q_0}{\dr x_j}+\frac{i}{2} \frac{\dr p_0}{\dr \theta_j^2}\frac{\dr q_0}{\dr y_j}\\
    &+(-i\lambda_j) \frac{\dr p_0}{\dr \theta_j^1}\frac{\dr q_0}{\dr \theta_j^2}+i\lambda_j\frac{\dr p_0}{\dr \theta_j^2}\frac{\dr q_0}{\dr \theta_j^1}\Biggr]\\
    &= i \left\{p_0(z, -J(d\phi(z))),q_0(z, -J(d\phi(z)))\right\}_L + i\{p_0, q_0\}_{\Psi}(z,-J(d\phi(z))
    \end{aligned}
    $$
    where $\{f,g\}_L $ is the Poisson bracket associated to $L$ and $\{\cdot, \cdot\}_\Psi $ is the Poisson bracket of the symbols on $M$.
\end{exe}

Summing up, we obtain the following. 
\begin{thm}
    $$C_1(P,Q;k) - C_1(Q,P;k) = [P,Q] + \frac{1}{k}\(D_1(P,Q)- D_1(Q,P) \)$$ where $D_1(P,Q)- D_1(Q,P) $ is a pseudodifferential operator of order $m_1+m_2$ whose principal symbol satisfies $$\sigma_0(D_1(P,Q)- D_1(Q,P) )(z,-Jd\phi(z)) = i\{p_0(z,-Jd\phi(z)),q_0(z,-Jd\phi(z))\}_L.$$
\end{thm}

\end{document}